\documentclass[11pt]{amsart}
\usepackage{amsfonts, amsmath, amssymb, amsthm, amscd, enumerate}
\usepackage[utf8]{inputenc}

\allowdisplaybreaks

\theoremstyle{plain}
\newtheorem{thm}{Theorem}[section]
\newtheorem{lem}[thm]{Lemma}

\newtheorem{cor}[thm]{Corollary}

\theoremstyle{remark}
\newtheorem{rem}[thm]{Remark}

\theoremstyle{definition}
\newtheorem{exam}[thm]{Example}
\newtheorem{nota}[thm]{Notation}
\newtheorem{dfn}[thm]{Definition}

\newcommand{\bbC}{\mathbb C}
\newcommand{\bbN}{\mathbb N}
\newcommand{\bbR}{\mathbb R}
\newcommand{\bbZ}{\mathbb Z}

\newcommand{\cC}{\mathcal C}

\newcommand{\cL}{\mathcal L}
\newcommand{\cW}{\mathcal W}

\newcommand{\bd}{\mathbf d}

\newcommand{\bs}{\mathbf s}
\newcommand{\bt}{\mathbf t}
\newcommand{\bx}{\mathbf x}
\newcommand{\by}{\mathbf y}

\newcommand{\rB}{ B}
\newcommand{\rE}{ E}

\newcommand{\rG}{ G}
\newcommand{\rH}{ H}
\newcommand{\rM}{ M}

\newcommand{\rS}{ S}

\newcommand{\rV}{ V}

\newcommand{\rZ}{ Z}

\DeclareMathOperator{\Forall}{\forall}

\DeclareMathOperator{\Ch}{Ch}
\DeclareMathOperator{\CS}{CS}
\DeclareMathOperator{\id}{id}
\DeclareMathOperator{\tr}{tr}
\DeclareMathOperator{\str}{str}
\DeclareMathOperator{\Span}{Span}
\DeclareMathOperator{\op}{op}
\DeclareMathOperator{\HS}{HS}
\DeclareMathOperator{\Hom}{Hom}
\DeclareMathOperator{\Ker}{Ker}

\DeclareMathOperator{\Rg}{Rg}
\DeclareMathOperator{\rk}{rk}
\DeclareMathOperator{\Pf}{Pf}
\DeclareMathOperator{\Ab}{Ab}

\begin{document}
%%%%%%%%%%%%%%%%

\title[Non-commutative analytic torsion form]
{Non-commutative analytic torsion form on the transformation groupoid convolution algebra}

\author{Bing Kwan SO}
\author{GuangXiang SU}
\thanks{Jilin University, Chanchun, P. R. China. bkso@graduate.hku.hk (B.K. So) \\
Chern Institute of Mathematics and LPMC, Nankai University, Tianjin, P. R. China.
guangxiangsu@gmail.com (G. Su)}

\begin{abstract}
Given a fiber bundle $Z \to M \to B$ and a flat vector bundle $E \to M$ with a compatible action
of a discrete group $G$, 
and regarding $B / G$ as the non-commutative space corresponding to the crossed product algebra,
we construct an analytic torsion form as a non-commutative deRham differential form.
We show that our construction is well defined under the weaker assumption of positive Novikov-Shubin invariant.
We prove that this torsion form appears in a transgression formula,
from which a non-commutative Riamannian-Roch-Grothendieck index formula follows.
\end{abstract}

\maketitle

\section{Introduction}
The basic philosophy of non-commutative geometry is to regard some non commutative algebras
as (smooth, continuous, measurable) functions on a space,
and then extending geometric concepts like topological invariants to these algebras \cite{Connes;Book}.
One of such classes of topological invariants that has been particularly successfully generalized to 
``non-commutative spaces" is that of index theory (see \cite[Chapter 2]{Connes;Book} for an introduction).

In this paper, 
we turn to construct another important invariant, namely, the Bismut-Lott analytic torsion form,
for the non-commutative transformation groupoid convolution algebra.
Our approach is based on the non-commutative super-connection formalism of \cite{Lott;EtaleGpoid,Piazza;NonCommEta}, 
developed for local index theory.

Recall that the Bismut-Lott analytic torsion form was constructed as a higher analogue of the Ray-Singer torsion
\cite{Bismut;AnaTorsion}.
Let $Z \to M \xrightarrow{\pi} B $ be a fiber bundle with connected closed fibers 
$Z_x := \pi^{-1}(x)$, $x \in B$ and let
$E \to M$ be a complex vector bundle with flat connection $\nabla ^\rE$
and Hermitian metric $g^\rE$.
Fix a splitting $T \rM = \rV \oplus \rH$ into vertical and horizontal bundles.
Let $D _t$ be the rescaled Dirac operator. 
The Bismut-Lott analytic torsion form is defined as \cite[(3.118)]{Bismut;AnaTorsion}:
\begin{multline}
\label{BLDfn}
\int _0 ^\infty \left\{- F ^\wedge (t)
+ \frac{\chi'(Z;E)}{2}
+ \big(\frac{\dim (\rZ) \rk (\rE) \chi(Z)}{4} - \frac{\chi'(Z;E)}{2} \big)
(1 - 2 t) e ^{-t} \right\}\frac {d t}{t} \\
 \in \Gamma ^\infty (\wedge ^\bullet T ^* \rB),
\end{multline}
where
$$ F ^\wedge (t)
:= (2 \pi \sqrt{-1}) ^{- \frac{N_\Omega }{2}}
\str _\Psi (2 ^{-1} N (1 + 2 (D_t - D ' _t)^2) e ^{- (D _t + D ' _t) ^2}.$$
The Bismut-Lott analytic torsion form appears in a transgression formula,
hence a Riemannian-Roch-Grothendeick index formula follows.
This construction was extended to general foliations with Hausdorff holonomy gropoids 
by Heistech and Lazarov \cite{Heitsch;FoliationTorsion},
using Haefliger cohomology.

When the fiber of the bundle is some non-commutative space $\mathcal B$ 
(i.e. a smooth sub-algebra of some $C^*$-algebra),
Lott \cite{Lott;NonCommTorsion} defined the analytic torsion as
\begin{align}
\label{LottTorsion}
\int _0 ^\infty & \Big( \int _0 ^1 \str _\Psi ( N e ^{- D _t (r) ^2 } ) d r \\ \nonumber
+ \int _0 ^1 & r (1 - r) \int _0 ^1 \str _\Psi \big(N 
\big[D _t - D _t ', e ^{ - r' D _t (r) ^2 } (D _t - D_t ') 
e ^{ -(1 - r') D _t (r) ^2 } \big] \big) d r' d r \Big) \frac {d t}{t}.
\end{align}
Necessarily, our definition is formally the same as \eqref{LottTorsion}.
However, 
we instead regard the base space as some non-commutative space (the transformation groupoid convolution algebra).
Correspondingly, 
we replace the deRham complex (with coefficient) by the non-commutative deRham complex,
and we use the non-commutative Bismut super-connection and the trace defined \cite{Lott;EtaleGpoid}.
By some standard arguments, we obtain a transgression formula and a non-commutative Riemannian-Roch-Grothendieck index theorem.
Thus our work again verifies the power of the super-connection formalism,
as pointed out in \cite{Lott;EtaleGpoid} and \cite{Nistor;EtaleCycilcHomology}.

In order to adapt the standard construction, there is, however, a major technical difficulty we need to overcome --
the integral \eqref{LottTorsion} may not converge as $t \to \infty$.
In \cite{Lott;NonCommTorsion}, the author made a very strong additional assumption that the Laplacian has a spectral gap at $0$.
This assumption is obviously true for a compact fiber bundle, 
but usually false in the non-compact case.
At this point our technical approach differs from \cite{Lott;NonCommTorsion}. 
In \cite{Schick;NonCptTorsionEst}, Azzali, Goette and Schick proved
that the integrand defining the $L^2$ analytic torsion form,
as well as several other invariants related to the signature operator,
decays polynomially provided the Novikov-Shubin invariant is positive.
In \cite{So;CommTorsion}, we proved that its derivatives also satisfy similar estimates
(and as a corollary the $L^2$ analytic torsion form is smooth).
In this paper, 
we use similar arguments to prove that the non-commutative terms and their derivatives in the analogue of \eqref{LottTorsion} 
also decay polynomially under, 
the same condition that the Novikov-Shubin invariant being positive .
Therefore the non-commutative analytic torsion form is well defined and smooth.

The main theme of this paper is thus extending the technical results of 
\cite{So;CommTorsion} to the non-commutative case.
In Section 2, we review the main construction of \cite{So;CommTorsion}, 
namely the Sobolev type norms $\| \cdot \| _{\HS m}$ for kernels, 
and the operator norms. 
The main result is Corollary \ref{Main1}, which concerns the compatibility of the two norms.
In Section 3, we begin with reviewing the non-commutative differential forms 
and the Bismut super-connection \cite{Lott;EtaleGpoid, Lott;NonCommTorsion}.
Then we extend the norms constructed in \cite{So;CommTorsion} to the non-commutative case
(it is essentially $\ell ^2$ in the $d g _{(k)}$ components), 
and generalize Corollary \ref{Main1} to non-commutative forms.
In Section 4, we mainly follow Section 4 of \cite{Schick;NonCptTorsionEst} to compute the large time limit of the 
non-commutative heat kernel.
Here, a major difficulty is that the non-commutative Bisumt super-connection is {\it not} flat, 
unlike the commutative case, 
and which is a major assumption in \cite{Schick;NonCptTorsionEst}.
However, we discover that one can express the bracket involving the connection term of the Bismut super-connection 
as a product of bounded, fiber-wise operators.
Finally in Section 5, we write down the relevant character forms,
compute their short time limit (with rather standard techniques) and prove our transgression and index formulas.
In the last section, we give some more remarks and highlight some open problems.

\begin{nota}
Throughout the paper, given two real valued expressions $f_1 , f_2$ we will write
$$ f _1 \dot \leq f _2 $$
if there exists some constant $C \geq 0$ such that 
$ f _1 \leq C f _2 $.
\end{nota}

\section{Sobolev norms on the fibered product groupoid}
In this section, we review the construction of norms and Sobolev spaces in \cite{So;CommTorsion}.

\subsection{The geometric settings}
\label{Geom}
Let $\rZ \to \rM \xrightarrow{\pi} \rB $ be a fiber bundle with connected fibers $Z_x := \pi^{-1}(x)$, $x \in \rB$. 
We assume $\rB$ is compact, however, $\rM$ is, in general, non-compact.
Denote the vertical tangent bundle by $\rV := \Ker (d \pi ) \subset T \rM$.

We suppose that there is a finitely generated discrete group $\rG$ acting on $\rM$ from the right freely and properly discontinuously. 
We also assume that $\rG$ acts on $\rB$ such that the actions commute with $\pi$ and $\rM _0 := \rM / \rG $ is a compact manifold.
Since the submersion $\pi$ is $\rG$-invariant,
$\rM _0 $ is also foliated and denote such foliation by $\rV _0 $.
Fix a distribution $\rH _0 \subset T \rM _0 $ complementary to $\rV _0$.  
Fix a metric on $\rV_{0}$ and a $\rG$-invariant metric on $\rB$.
Then one obtains a Riemannian metric on $\rM_0$ as $g^{\rV_0} \oplus \pi^* g^\rB$ on $T \rM_0 = \rV_0 \oplus \rH_0$.

Since the projection from $\rM$ to $\rM _0 $ is a local diffeomorphism,
one gets a $\rG$-invariant splitting $T \rM = \rV \oplus \rH $.
Furthermore this local diffeomorphism induces $\rG$-invariant metrics on $\rV$ and $\rM$.
Denote by $P ^\rV , P ^\rH$ respectively the projections to $\rV $ and $\rH$.

Given any vector field $X \in \Gamma ^\infty (T \rB)$,
denote the horizontal lift of $X $ by $X ^\rH \in \Gamma ^\infty (\rH) \subset \Gamma ^\infty (T \rM )$.
By our construction, 
$$| X ^\rH | _{g ^\rM} (p) = | X | _{g ^\rB } (\pi (p)).$$
Denote by $\mu _x , \mu _\rB $ respectively the Reimannian measures on $\rZ _x $ and $\rB $.

\begin{dfn}
\label{BGV}
We will consider several connections on the tangent bundle.
Denote by $\nabla ^\rB , \nabla ^\rM$ respectively the Levi-Cevita connection on $\rB$ and $\rM$.
Define the connection $\nabla ^{\rM / \rB}$ on the vertical bundle $\rV \to \rM$ by \cite[p.322]{BGV;Book}
$$ \nabla ^{\rM / \rB } := P ^\rV \nabla ^\rM P ^\rV,$$
and define another connection $\nabla^{\oplus}$ on $T \rM = \pi ^* \nabla ^\rB \oplus \rV$ by \cite[Proposition 10.2]{BGV;Book}
$$\nabla ^\oplus := \pi ^* \nabla ^\rB + \nabla ^{\rM / \rB}.$$
We denote the curvature of $ \nabla ^{\rM / \rB } $ by $R ^{\rM / \rB}$.
We will also abuse notation to use the same symbol to denote the induced connection on the dual and exterior product bundles.
\end{dfn}

\begin{dfn}
Let $\rE \xrightarrow{\wp} \rM$ be a complex vector bundle.
We say that $\rE $ is a contravariant $\rG$-bundle if $\rG$ also acts on $\rE$ from the right,
such that for any $v \in \rE , g \in \rG $, $\wp (v g) = \wp (v) g \in \rM $,
and moreover $\rG$ acts as a linear map between the fibers.

The group $\rG$ then acts on sections of $\rE$ from the left by
$$ s \mapsto g ^* s, \quad (g ^* s ) (p) := s (p g ) g ^{-1} \in \wp ^{-1} (p), \quad \forall \; p \in \rM .$$
\end{dfn}

We assume that $\rE$ is endowed with a $\rG$-invariant metric $g _\rE$,
and a $\rG$-invariant connection $\nabla ^\rE$
(which is obviously possible if $\rE$ is the pullback of some bundle on $\rM _0$).
In particular, for any invariant section $s$ of $\rE$,
$| s |$ is an invariant function on $\rM$. Let $(\nabla^E)'$ be the adjoint connection of $\nabla^E$ with respect to $g_E$.

In the following, for any vector bundle $\rE$ we denote its dual bundle by $\rE'$.

Recall that the ``infinite dimensional bundle'' over $\rB$ in the sense of Bismut is a vector bundle with typical fiber
$\Gamma _c ^\infty (\rE |_{\rZ _x } ) $ (or other function spaces) over each $x \in \rB$.
We denote such Bismut bundle by $\rE _\flat$.
The space of smooth sections on $\rE _\flat$ is, as a vector space, $\Gamma ^\infty _c (\rE )$.
Each element $s \in \Gamma ^\infty _c (\rE ) $ is regarded as a map
$$ x \mapsto s |_{\rZ _x} \in \Gamma _c ^\infty (\rE |_{\rZ _x } ) , \quad \Forall x \in \rB .$$
In other words, one defines a section on $\rE _\flat $ to be smooth,
if the images of all $x \in \rB$ fit together to form an element in $\Gamma _c ^\infty (\rE )$.
In particular,
$ \Gamma ^\infty _c ((\rM \times \bbC ) _\flat ) = C ^\infty _c (\rM ),$
and one identifies $\Gamma ^\infty _c (T \rB \otimes (\rM \times \bbC ) _\flat ) $
with $\Gamma ^\infty _c (\rH )$ by
$X \otimes f \mapsto f X ^\rH $.

Now we recall the defintion of the Bismut super-connection in the commutative case.
To shorten notations we denote $\rE ^\bullet := \rE \otimes \wedge ^\bullet \rV'$.
\begin{dfn}
\label{BismutDfn}
The Bismut super-connection is an operator of the form
$$ D _\rB
:= d ^{\nabla ^{\rE }} _\rV + L ^{\rE ^\bullet _\flat } + \iota _\varTheta ,$$
where $d ^{\nabla ^{\rE }} _\rV$ is the fiber-wise DeRham differential,
and  $\iota _\varTheta $ is the contraction with the $\rV$-valued horizontal 2-form $\varTheta$ defined by
$$ \varTheta (X _1 ^\rH , X _2 ^\rH) := - P^ \rV [ X _1 ^\rH , X _2 ^\rH ] ,
\quad \Forall X _1 , X _2 \in \Gamma ^\infty (T \rB). $$
\end{dfn}

Here, we recall that the operator 
$D _\rB $ is just the DeRham differential operator \cite[Proposition 10.1]{BGV;Book}.
However, the grading and the identification 
$\wedge ^\bullet \rH' \otimes \wedge ^\bullet \rV' \otimes \rE \cong \wedge ^\bullet T ^* \rM \otimes \rE $, 
depends on the splitting.

On the Bismut bundle one has the standard metric on $\Gamma ^\infty _c (\rE _\flat ) $ given by
\begin{equation}
\label{BismutInner}
\langle s _1 , s _2 \rangle _{\rE _\flat } (x) 
:= \int _{\rZ _x} \langle s _1 (p) , s _2 (p) \rangle _\rE \mu _x (p) .
\end{equation}
The adjoint connection of
$\eth _\rB$ with respect to $\langle \cdot , \cdot \rangle _{\rE _\flat}$, 
which is defined by the relation
\begin{equation}
\label{CommAd}
d _\rB \langle s _1 , s _2 \rangle _{\rE _\flat }
= \langle D _\rB s _1 , s _2 \rangle _{\rE _\flat }
-\langle s _1 , D '_\rB s _2 \rangle _{\rE _\flat },
\end{equation}
is given by 
$$D _\rB' = (d ^{\nabla ^\rE} _\rV )^* + (L ^{\rE ^\bullet _\flat})' -{\varTheta } \wedge,$$
where $(L ^{\rE ^\bullet _\flat} )'$ is the adjoint connection of $L ^{\rE ^\bullet _\flat}$ .
See \cite[Proposition 3.7]{Bismut;AnaTorsion} and \cite{Lopez;FoliationHeat} for explicit formulas for 
$(L ^{\rE ^\bullet _\flat})'$.
Note that the degree $(0, -1)$ component $(d ^{\nabla ^\rE} _\rV )^* $ is the formal adjoint operator of 
$d ^{\nabla ^\rE} _\rV $ (we use the superscript ${}'$ to denote adjoint connections and ${}^*$ to denote adjoint operators).
Recall that $D _\rB '$ is also flat, i.e. $(D _\rB ') ^2 = 0$.

\subsection{Covariant derivatives}

In this section we recall some constructions of \cite[Section 2]{So;CommTorsion}.

From the connection $\nabla^E$, one defines an induced connection on the Bismut bundle $\rE _\flat$ (as a $C ^\infty (\rB)$ module) by
$$ \nabla ^{\rE _\flat } _X s := \nabla ^\rE _{X ^\rH } s ,
\quad \Forall s \in \Gamma ^\infty (\rE _\flat ) \cong \Gamma _{c}^\infty (\rE ).$$
Also, note that $[X ^\rH , Y ] $ is vertical for any vertical vector field $Y \in \Gamma ^\infty (\rV )$.
Therefore 
$$ \nabla ^{\rV _\flat} _X Y
:= [X ^\rH , Y ] , \quad \Forall Y \in \Gamma ^\infty (\rV _\flat ) \cong \Gamma ^\infty (\rV ) $$
naturally defines a connection.

\begin{dfn}
\label{DiffDfn1}
(cf. \cite[Definition 2.2]{So;CommTorsion})
The covariant derivative on $\rE _\flat $ is the map
$$\dot \nabla ^{\rE _\flat } : 
\Gamma ^\infty (\otimes ^\bullet T ^* \rB \bigotimes \otimes ^\bullet \rV' _\flat \bigotimes \rE _\flat )
\to \Gamma ^\infty (\otimes ^{\bullet + 1} T ^* \rB \bigotimes \otimes ^\bullet \rV' _\flat \bigotimes \rE _\flat ),$$
defined by
\begin{align*}
\big(\dot \nabla ^{\rE _\flat } s \big)(X _0 , X _1 , \cdots, X _k ; Y _1 , \cdots,Y _l )
:= \nabla ^{\rE _\flat } _{X _0} s ( X _1 , \cdots, X _k ; Y _1 , \cdots, Y _l ) \\
- \sum _{j=1} ^l s \big(X _1 , \cdots, X _k ; Y _1 , \cdots , \nabla ^{\rV _\flat} _{X _0 } Y _j , \cdots , Y _l \big)\\ 
- \sum _{i=1} ^k s\big(X _1 , \cdots , \nabla ^{\rB } _{X _0 } X _i , \cdots , X _k ; Y _1 , \cdots Y _l\big) ,
\end{align*}
for any $k, l \in \bbN , X _0 , \cdots, X _k \in \Gamma ^\infty (T \rB), Y _1 , \cdots, Y _l \in \Gamma ^\infty (\rV)$.
\end{dfn}

Clearly, taking covariant derivative can be iterated,
which we denote by $(\dot \nabla ^{\rE _\flat }) ^m $, \\ $ m = 1, 2, \cdots$.
Note that $(\dot \nabla ^{\rE _\flat }) ^m $ is a differential operator of order $m$.

Also, we define
$\dot \partial ^{\rV } :
\Gamma ^\infty (\otimes ^\bullet T ^* \rB \bigotimes \otimes ^\bullet \rV' _\flat \bigotimes \rE _\flat )
\to \Gamma ^\infty (\otimes ^\bullet T ^* \rB \bigotimes \otimes ^{\bullet + 1} \rV' _\flat \bigotimes \rE _\flat )$ by
\begin{align}
\label{VertDiff1}
\nonumber
\big(\dot \partial ^{\rV } s \big)(X _1 , \cdots, X _k ; Y _0 , Y _1 , \cdots, Y _l )
:=& \nabla ^\rE _{Y _0} s ( X _1 , \cdots, X _k ; Y _1 , \cdots, Y _l ) \\ 
- \sum _{j=1} ^l & s (X _1 , \cdots, X _k ; Y _1 , \cdots , P^\rV (\nabla ^{\rM } _{Y _0 } Y _j ) , \cdots , Y _l ) .
\end{align}
Note that the operators $\dot \nabla ^{\rE _\flat } $ and $ \dot \partial ^{\rV } $ are just respectively the
$(0, 1)$ and $(1, 0)$ parts of the usual covariant derivative operator.

Since $\rM$ is locally isometric to a compact space $\rM _0$,
it is a manifold with bounded geometry (see \cite[Appendix 1]{Shubin;BdGeom} for an introduction).
On any manifold with bounded geometry one constructs various standard Sobolev spaces \cite[Appendix 1 (1.3)]{Shubin;BdGeom}.
In particular we regard $(\dot \nabla ^{\rE _\flat})^i (\dot\partial^{V}) ^j s
\in \Gamma ^\infty ( \otimes ^i \rH' \bigotimes \otimes ^j \rV' \bigotimes \rE_{\flat} )$, 
and consider:
\begin{dfn}
\label{SobDfn}
For $s\in \Gamma^{\infty}_{c}(E)$, we define its $m$-th Sobolev norm by
\begin{equation}
\| s \| ^2 _m
:= \sum _{i+j \leq m} \int _{x \in \rB} \int _{y \in \rZ _x }
\left| (\dot \nabla ^{\rE _\flat})^i (\dot \partial ^{\rV} ) ^j s \right| ^2 (x, y)
\mu _x (y) \mu _\rB (x).
\end{equation}
Denote by $\cW ^m (\rE)$ be the Sobolev completion of $\Gamma ^\infty _c (\rE ) $ with respect to $\| \cdot \| _m $.
\end{dfn}

\begin{dfn}
We say that a differential operator $A$ is $C^{\infty}$-bounded if in normal coordinates, 
the coefficients and their derivatives are uniformly bounded. 
\end{dfn}

\begin{exam}
Any invariant connection $\nabla ^\rE$ is a $C ^\infty$-bounded differential operator,
because by $\rG $-invariance the Christoffel symbols of
$\nabla ^\rE$ and all their derivatives are uniformly bounded.
It follows that using normal coordinate charts and parallel transport with respect to $\nabla ^\rE$ as trivialization,
one sees that $\rE$ is a bundle with bounded geometry.
\end{exam}

%One has the elliptic regularity for these Sobolev spaces:
%\begin{lem}
%\label{EllReg1}
%\cite[Lemma 1.4]{Shubin;BdGeom}
%Let $A $ be any $C ^\infty$-bounded, uniformly elliptic, differential operator of order $m$.
%For any $i\geq 0, j\geq 0$ there exists a constant $C$ such that for any $s \in \Gamma ^\infty _c (\rE) $
%$$ \| s \| _{i + m} \leq C ( \| A s \| _i + \| s \| _j) .$$
%\end{lem}

%\begin{rem}
%Throughout this paper, by an ``elliptic operator" on a manifold,
%we mean elliptic in all directions,
%without taking any foliation structure into consideration.
%We use the term ``fiber-wise elliptic operators" to refer to differential operators 
%that are fiber-wise and elliptic restricted to fibers.
%\end{rem}

\subsection{The fibered product}
\begin{dfn}
The fibered product of the submersion $\rM \to \rB$ is defined to be the manifold
$$\rM \times _\rB \rM := \{ (p, q) \in \rM \times \rM : \pi (p) = \pi (q) \} .$$
It is endowed with maps $\bs, \bt : \rM \times _\rB \rM \to \rM$ defined by
$$ \bs (p, q) := q, \quad \bt (p, q) := p .$$
The manifold $\rM \times _\rB \rM $ is a fiber bundle over $\rB $, with typical fiber $\rZ \times \rZ $.
One naturally has the splitting \cite[Section 2]{Heitsch;FoliationHeat}
$$T (\rM \times _\rB \rM) = \hat \rH \oplus \rV _\bt \oplus \rV _\bs ,$$
where
$$\rV _\bs := \Ker (d \bt) , \quad \rV _\bt := \Ker (d \bs ).$$
Denote by $P ^{\rV _\bs }, P ^{\rV _\bt}$ the projections onto $\rV _\bt $ and $\rV _\bs$.
\end{dfn}
Note that $\rV _\bs \cong \bs ^{*} \rV$ and $\rV _\bt \cong \bt ^{*} (\rV )$.
As in Section 1.1, we endow $\rM \times _\rB \rM $ with a metric by lifting the metrics on $\rH _{0}$ and $\rV_{0}$.
Then $\rM \times _\rB \rM $ is a manifold with bounded geometry.

\begin{nota}
With some abuse in notations,
we shall often write elements in $\rM \times _\rB \rM $ as triples $(x, y, z)$, where $x \in \rB , y , z \in \rZ _x $.
Using these notations $\bs (x, y, z) = (x, z), \bt (x, y, z) = (x, y) \in \rM $.
\end{nota}

Let $\rG$ act on $\rM \times _\rB \rM $ by the diagonal action
$$ (p, q) g := (p g , q g ). $$
Let $\rE \to \rM $ be a contravariant $\rG$-vector bundle and $\rE '$ be its dual.
We will consider
$$\hat \rE \to \rM \times _\rB \rM := \bt ^* \rE \otimes \bs ^* \rE ' .$$
Given a $\rG$-invariant connection $\nabla ^\rE $ on $\rE $, let
$$\nabla ^{\hat \rE } := \bt ^* \nabla ^{\rE } \otimes I_{s^*{E'}} + I_{t^*{E}} \otimes \bs ^* \nabla ^{\rE ' }$$
be the tensor of the pullback connections.

Similar to Definition \ref{DiffDfn1}, we define the covariant derivative operators on \\
$\Gamma ^\infty (\otimes ^\bullet T ^* \rB \bigotimes \otimes ^\bullet (\rV' _\bt)_\flat
\bigotimes \otimes ^\bullet (\rV' _\bs)_\flat \bigotimes \hat \rE _\flat )$.
\begin{dfn}
Define
\begin{align*}
\nonumber
\big(\dot \nabla ^{\hat \rE _\flat } \psi \big) 
(X _0 , X _1 , \cdots &,X _k ; Y _1 , \cdots, Y _l , Z _1 , \cdots ,Z _{l'}) \\
:=& \nabla ^{\hat \rE _\flat } _{X _0} \psi
( X _1 , \cdots ,X _k ; Y _1 , \cdots,Y _l , Z _1 , \cdots, Z _{l'}) \\ 
&- \sum _{1 \leq j \leq l} \psi (X _1 , \cdots, X _k ;
Y _1 , \cdots, \nabla ^{\rV _\flat} _{X _0 } Y _j , \cdots , Y _l , Z _1 , \cdots, Z _{l'}) \\ 
&- \sum _{1 \leq j \leq l'} \psi (X _1 , \cdots, X _k ;
Y _1 , \cdots, Y _l , Z _1 , \cdots, \nabla ^{\rV _\flat} _{X _0 } Z _j , \cdots, Z _{l'}) \\ 
&- \sum _{1 \leq i \leq k} \psi (X _1 , \cdots, \nabla ^{\rB } _{X _0 } X _i , \cdots, X _k ; Y _1 , \cdots Y _l
, Z _1 , \cdots, Z _{l'}), \\
\nonumber
\big(\dot \partial ^{\bs } \psi \big) (X _1 , \cdots, X _k &; Y _0 , Y _1 , \cdots, Y _l , Z _1 , \cdots, Z _{l'}) \\
:=& \nabla ^{\hat \rE} _{Y _0} \psi ( X _1 , \cdots, X _k ; Y _1 , \cdots ,Y _l , Z _1 , \cdots ,Z _{l'}) \\ 
&- \sum _{1 \leq j \leq l} \psi (X _1 , \cdots ,X _k ;
Y _1 , \cdots , P ^{\rV _\bs} (\nabla ^{ \rM} _{Y _0 } Y _j) , \cdots , Y _l , Z _1 , \cdots, Z _{l'}) \\ 
&- \sum _{1 \leq j \leq l'} \psi (X _1 , \cdots ,X _k ;
Y _1 , \cdots ,Y _l , Z _1 , \cdots , P ^{\rV _\bt} [Y _0 , Z _j] , \cdots , Z _{l'}), \\ 
\nonumber
\big(\dot \partial ^{\bt } \psi \big)(X _1 , \cdots, X _k &;  Y _1 , \cdots ,Y _l , Z _0 , Z _1 , \cdots, Z _{l'}) \\
:=& \nabla ^{\hat \rE} _{Y _0} \psi ( X _1 , \cdots, X _k ; Y _1 , \cdots ,Y _l , Z _0 , Z _1 , \cdots, Z _{l'})\\ 
&- \sum _{1 \leq j \leq l} \psi (X _1 , \cdots, X _k ;
Y _1 , \cdots , P ^{\rV _\bs}[Y _0 , Z _j] , \cdots , Y _l , Z _1 , \cdots, Z _{l'}) \\ 
&- \sum _{1 \leq j \leq l'} \psi (X _1 , \cdots ,X _k ;
Y _1 , \cdots ,Y _l , Z _1 , \cdots , P ^{\rV _\bt} (\nabla ^{\rM} _{Z _0 } Z _j) , \cdots , Z _{l'}).
\end{align*}
\end{dfn}
Note that $\dot \partial ^{\bs } $ and $\dot \partial ^{\bt }$ are essentially 
$\dot \partial ^{\rV }$ for the fiber bundle $\rM \times _\rB \rM \to \rM$ with $\bs $ (resp. $\bt$) as the projection. 

For any $(x, y, z) \in \rM \times _\rB \rM $,
let $\bd _x (y, z) $ be the Riemannian distance between $y, z \in \rZ _x$.
We regard $\bd $ as a continuous, non-negative function on $\rM \times_\rB \rM$.

\begin{dfn}
\label{NWX}
(See  \cite{NWX;GroupoidPdO}). As a vector space,
$$\Psi ^{- \infty } _\infty (\rM \times _\rB \rM , \rE ) :=
\left\{
\begin{array}{ll}
& \text{For any } m \in \bbN, \varepsilon > 0 , \exists C _m > 0  \\
\psi \in \Gamma ^\infty (\hat \rE ) : &  \text{such that } \Forall i+j+k \leq m, \\
& |(\dot \nabla ^{\hat \rE _\flat } )^i (\dot \partial ^{\rV _\bs} )^j (\dot \partial ^{\rV _\bt} )^k \psi |
\leq C _m e ^{- \varepsilon \bd}
\end{array}
\right\}.
$$
The convolution product structure on $\Psi ^{- \infty } _\infty (\rM \times _\rB \rM , \rE ) $ is defined by
$$ \psi _1 \star \psi _2 (x, y, z)
:= \int _{\rZ _x } \psi _1 (x, y, w) \psi _2 (x, w , z ) \mu _{x } (w) .$$
\end{dfn}

Now we introduce a Sobolev type norm on 
$\Psi ^{- \infty } _\infty (\rM \times _\rB \rM , \rE ) $.
Fix a non-negative function $\chi \in C ^\infty _c (\rM )$ such that
\begin{equation}
\label{Partition}
\sum _{g \in \rG} g ^* \chi = 1.
\end{equation}
We may further assume $\chi ^{\frac{1}{2}}$ is smooth.

\begin{dfn}
\label{NormDfn}
For any $g \in \rG$, $\psi \in \Psi ^{- \infty } _\infty (\rM \times _\rB \rM , \rE )$, define
\begin{align*}
\| \psi \| ^2 _{\HS m} (g)
:= \sum _{i+j+k \leq m} \int _\rB \int _{\rZ _x } \chi (x, z) \int _{\rZ _x } & \left|
(\dot \nabla ^{\hat \rE _\flat } )^i (\dot \partial ^\bs )^j (\dot \partial ^\bt)^k ((g ^{-1})^* \psi ) \right| ^2
(x, y, z) \\
& \mu _x (y) \mu _x (z) \mu _\rB (x).
\end{align*}
Denote by $\bar \Psi ^{- \infty } _m (\rM \times _\rB \rM , \rE ) $ the completion of 
$\Psi ^{- \infty } _\infty (\rM \times _\rB \rM , \rE ) $ with respect to $\| \cdot \| _{\HS m}$.
\end{dfn}

\begin{rem}
If $\psi $ is $\rG$-invariant, 
then Definition \ref{NormDfn} is constant and coincides with \cite[Definition 1.9]{So;CommTorsion}.
\end{rem}

\subsection{Fiber-wise operators}

\begin{dfn}
A fiber-wise operator is a linear operator $A : \Gamma ^\infty _c (\rE _\flat) \to \cW ^0 (\rE)$
such that for all $x \in \rB$,
and any sections $s _1, s _2 \in \Gamma ^\infty _c (\rE _\flat )$,
$$ (A s _1 )(x) = (A s _2) (x), $$
whenever $s _1 (x) = s _2 (x) $.

We say $A $ is smooth if
$A (\Gamma ^\infty _c (\rE )) \subseteq \Gamma ^\infty (\rE )$.
A smooth fiber-wise operator $A $ is said to be bounded of order $m$
if $A $ can be extended to a bounded map from $\cW ^m (\rE )$ to itself.

Denote the operator norm of $A : \cW ^m (\rE ) \to \cW ^m (\rE ) $ by $\| A \| _{\op m} $.
\end{dfn}

Note that
\begin{equation}
\| g^* A \| _{\op m} = \| A \| _{\op m}
\end{equation}
because $g ^*$ is an isometry.

\begin{exam}
\label{SmoothingExam}
An example of smooth fiber-wise operators is $\Psi ^{- \infty } _\infty (\rM \times _\rB \rM , \rE )$,
acting on $\cW ^m (\rE )$ by vector representation, i.e.
$$ \big( \varPsi s \big) (x, y) := \int _{\rZ _x} \psi (x, y, z ) s (x, z) \mu _x (z) .$$
\end{exam}

\begin{nota}
For the fiber-wise operator operator $A : \Gamma ^\infty _c (\rE _\flat) \to \cW ^0 (\rE)$ which is of the form given by Example
\ref{SmoothingExam}, we denote its kernel by $A (x, y, z)$.
We will write
$$ \| A \| _{\HS m} := \| A (x, y, z) \| _{\HS m} ,$$
provided $A (x, y, z) \in \bar \Psi ^{-\infty} _m (\rM \times _\rB \rM , \rE )$.
\end{nota}

Fix a local trivialization
$$ \bx _\alpha : \pi ^{-1} (\rB _\alpha ) \to \rB _\alpha \times \rZ , \quad
p \mapsto (\pi (p) , \varphi ^\alpha (p)),$$
where $\rB = \bigcup _{\alpha } \rB _\alpha $ is a finite open cover (since $B$ is compact), 
and $\varphi ^\alpha |_{\pi ^{-1} (x)} : \rZ _x \to \rZ $ is a diffeomorphism.
Such a trivialization induces a local trivialization of the fiber bundle $\rM \times _\rB \rM \xrightarrow{\bt} \rM $ by
$\rM = \bigcup \rM _{\alpha } , \rM _\alpha := \pi ^{-1} (\rB _\alpha ) $,
$$ \hat \bx _\alpha : \bt ^{-1} (\rM _\alpha ) \to \rM _\alpha \times \rZ , \quad
(p, q) \mapsto ( p , \varphi ^\alpha (q)).$$
On $\rM _\alpha \times \rZ$ the source and target maps are explicitly given by
\begin{equation}
\label{LocalGpoid}
\bs \circ (\hat \bx _\alpha) ^{-1} (p, z) = (\bx _\alpha ) ^{-1} (\pi (p) , z)
\text{ and } \bt \circ (\hat \bx _\alpha) ^{-1} (p, z) = p .
\end{equation}

For such trivialization, one has the natural splitting
$$ T (\rM _\alpha \times \rZ) = \rH ^\alpha \oplus \rV ^\alpha \oplus T \rZ ,$$
where $\rH ^\alpha $ and $ \rV ^\alpha $ are respectively $\rH $ and $\rV$ restricted to $\rM _\alpha \times \{ z \}$,
$z \in \rZ $.
It follows from (\ref{LocalGpoid}) that
$$\rV ^\alpha = d \hat \bx _\alpha (\rV _\bs), \quad T \rZ = d \hat \bx _\alpha (\rV _\bt) .$$
Given any vector field $X$ on $\rB$,
let $X ^\rH , X ^{\hat \rH}$ be respectively the lifts of $X$ to $\rH $ and $\hat \rH $.
Since $d \bt (X ^{\hat \rH} ) = d \bs (X ^{\hat \rH} ) = X ^\rH $,
it follows that
$$d \hat \bx _\alpha (X ^{\hat \rH }) = X ^{\rH ^\alpha } + d \varphi ^\alpha (X ^\rH ).$$
Note that $d \varphi ^\alpha (X ^\rH ) \in T \rZ \subseteq T (\rM _\alpha \times \rZ) $.

Corresponding to the splitting $T (\rM _\alpha \times \rZ) = \rH ^\alpha \oplus \rV ^\alpha \oplus T \rZ$,
one can define the covariant derivative operators.
Let $\nabla ^{T \rM _\alpha } $ be the Levi-Civita connection on $\rM_\alpha$
and $\nabla^{TZ}$ be the Levi-Civita connection on $\rZ$.
Define for any smooth section

Let $A $ be any smooth fiber-wise operator on $\Gamma ^\infty _c (\rE _\flat )$.
Then $A $ induces a fiber-wise operator $\hat A$ on $\Gamma ^\infty _c (\hat \rE _\flat)$ by
\begin{equation}
\label{FiberwiseOp}
\hat A (u \otimes \bs ^* e) := A ( u |_{ \rM _\alpha \times \{ z \}}) \otimes (\bs ^* e)
\end{equation}
on $\bt ^{-1} (\rM _\alpha ) \cong \rM _\alpha \times \rZ $,
for any sections
$e \in \Gamma ^\infty (\rE ') , u \in \Gamma ^\infty (\bt ^* \rE ) $
and $\psi = u \otimes \bs ^* e \in \Gamma ^{ \infty } _c (\hat \rE ) $.

Note that $\hat A$ is independent of trivialization since $A$ is fiber-wise,
and for any $\alpha , \beta $ and $z \in \rZ$,
the transition function $\bx _\beta \circ (\bx _\alpha ) ^{-1}$ maps the sub-manifold
$\rZ _x \times \{ z \} $ to $\rZ _x \times \{ \bx ^\beta _x \circ (\bx ^\alpha _x) ^{-1} (z) \}$
as the identity diffeomorphism.

For any smooth fiber-wise operator $A$ and $g \in \rG$, define 
$$ (g ^* A ) s := g ^* (A (g ^{-1}) ^* s)) .$$
It is easy to check that $g ^* A$ is still a smooth fiber-wise operators.
We will denote the corresponding operator induced on $\Gamma ^\infty _c (\hat \rE _\flat)$ by
$\widehat{g ^* A} $.

Define 
$$ \rS := \{ g \in \rG : \chi (g ^* \chi) \neq 0 \}.$$
Note that $\rS $ is finite since the $\rG$ action is proper.

With these preparations, we state the main result of this section,
which is a slight generalization of \cite[Theorem 2.16]{So;CommTorsion}:
\begin{thm}
\label{AEstCor1}
There exists a finite subset $\rS _1 \subset \rG$ 
such that for any smooth, bounded operator $A$, $\psi \in \Psi ^{- \infty} _\infty (\rM \times _\rB \rM ,E) $,
one has
$$ \| \hat A \psi \| _{\HS 1} (g)
\dot \leq \sum _{g _1 \in \rS _1} \big(\| A \| _{\op 1} 
+ \| A \| _{\op 0} \big) \| \psi \| _{\HS 1} ( g _1^{-1} g).$$
\end{thm}

\begin{proof}
Fix a partition of unity $\{ \theta _\alpha \} \in C ^\infty _c (\rB)$ subordinate to $\{ \rB _\alpha \}$.
We still denote by $\{ \theta _\alpha \}$ its pullback to $\rM$ and $\rM \times _\rB \rM $.
Fix any Riemannian metric on $\rZ$ and denote the corresponding Riamannian measure by $\mu _\rZ$.
Then one writes
$$ (\bx _\alpha ) _* (\mu _x \mu _\rB ) = J _\alpha \mu _\rB \mu _\rZ ,$$
for some smooth positive function $J _\alpha $.
Moreover, over any compact subsets on $\rB _\alpha \times \rZ$, $ \frac{1}{J _\alpha} $ is bounded.

On $\rM _\alpha \times \rZ$, define differential operators as in \cite[Equations (9), (10), (11)]{So;CommTorsion}:
\begin{align}
\nonumber
\big(\dot \nabla ^{\alpha } \phi \big)
(X _0 , X _1 , \cdots, & X _k ; Y _1 , \cdots, Y _l , Z _1 , \cdots ,Z _{l'}) \\
:=& (\bx _\alpha ^* \nabla ^{\hat \rE _\flat }) _{X ^{\rH ^\alpha } _0}
\phi ( X _1 , \cdots ,X _k ; Y _1 , \cdots, Y _l , Z _1 , \cdots ,Z _{l'}) \\ \nonumber
&- \sum _{1 \leq j \leq l} \phi (X _1 , \cdots ,X _k ; Y _1 , \cdots , [X _0 ^{\rH ^\alpha} , Y _j] , \cdots , Y _l ,
, Z _1 , \cdots, Z _{l'}) \\ \nonumber
&- \sum _{1 \leq j \leq l'} \phi (X _1 , \cdots, X _k ;
Y _1 , \cdots , Y _l , Z _1 , \cdots , [X _0 ^{\rH ^\alpha} Z _j ], \cdots , Z _{l'}) \\ \nonumber
&- \sum _{1 \leq i \leq k} \phi (X _1 , \cdots , \nabla ^{\rB } _{X _0 } X _i , \cdots , X _k ;
Y _1 , \cdots ,Y _l , Z _1 , \cdots, Z _{l'}), \\ \nonumber
\big(\dot \partial ^{\alpha } \phi \big) (X _1 , \cdots, X _k ; & Y _0 , Y _1 , \cdots ,Y _l , Z _1 , \cdots ,Z _{l'}) \\
:=& (\bx _\alpha ^* \nabla ^{\hat \rE _\flat}) _{Y _0} \phi ( X _1 , \cdots, X _k ;
Y _1 , \cdots ,Y _l , Z _1 , \cdots, Z _{l'}) \\ \nonumber
&- \sum _{1 \leq j \leq l} \phi (X _1 , \cdots ,X _k ;
Y _1 , \cdots , P ^{\rV ^\alpha} (\nabla ^{\rM _\alpha}_{Y _0 } Y _j), \cdots, Y _l, Z _1 , \cdots ,Z _{l'})\\ 
\nonumber
&- \sum _{1 \leq j \leq l'} \phi (X _1 , \cdots ,X _k ;
Y _1 , \cdots ,Y _l , Z _1 , \cdots , P ^{T \rZ} [Y _0 , Z _j] , \cdots , Z _{l'}), \\ \nonumber
\big( \dot \partial ^{\rZ } \phi \big)(X _1 , \cdots , X _k ; & Y _1 , \cdots ,Y _l , Z _0 , Z _1 , \cdots, Z _{l'}) \\
:=& (\bx _\alpha ^* \nabla ^{\hat \rE _\flat} )_{Z _0} \phi ( X _1 , \cdots ,X _k ;
Y _1 , \cdots ,Y _l , Z _0 , Z _1 , \cdots, Z _{l'}) \\ \nonumber
&- \sum _{1 \leq j \leq l} \phi (X _1 , \cdots, X _k ;
Y _1 , \cdots , P ^{\rV ^\alpha} [Z _0 , Y _j] , \cdots , Y _l , Z _1 , \cdots ,Z _{l'}) \\ \nonumber
&- \sum _{1 \leq j \leq l'} \phi (X _1 , \cdots, X _k ;
Y _1 , \cdots ,Y _l , Z _1 , \cdots , \nabla ^{\rZ} _{Z _0 } Z _j , \cdots , Z _{l'}),
\end{align}
for any smooth section
$\phi \in \Gamma ^\infty (\otimes ^\bullet T ^* \rB \bigotimes \otimes ^\bullet (\rV ^\alpha)' _\flat
\bigotimes \otimes ^\bullet T ^* \rZ _\flat \bigotimes (\hat \bx _\alpha ^{-1} )^* \hat \rE _\flat )$.
 
Given any $\psi \in \Psi ^{- \infty} _\infty (\rM \times _\rB \rM), g \in \rG$,
let $\psi _g ^\alpha := \hat \bx _\alpha ^* (g ^* \psi ) $.
Since by definition
$$ g ^* (\hat A \psi ) = (\widehat{g ^* A} ) (g ^* \psi),$$
the theorem clearly follows from the inequalities
\begin{align}
\label{EstLem1}
\int _{\rB _\alpha} \int _{\rZ _x} \chi (x, z) \int _{\rZ _x} &
| \dot \nabla ^\alpha \widehat {g ^* A} (\theta _\alpha \psi _g ^\alpha ) | ^2
\mu _x (y) \mu _x (z) \mu _\rB (x) \\ \nonumber
\dot \leq & \sum _{g _1 \in \rS }
(\| g^* A \| ^2 _{\op 1} + \| g ^* A \| ^2 _{\op 0} ) \| \psi \| ^2 _{\HS 1} (g _1 ^{-1} g), \\
\label{EstLem2}
\int _{\rB _\alpha} \int _{\rZ _x} \chi (x, z) \int _{\rZ _x} &
| \dot \partial ^\alpha \widehat {g ^* A} (\theta _\alpha \psi _g ^\alpha ) | ^2 
\mu _x (y) \mu _x (z) \mu _\rB (x) \\ \nonumber
\dot \leq & \sum _{g _1 \in \rS }
(\| g^* A \| ^2 _{\op 1} + \| g ^* A \| ^2 _{\op 0} ) \| \psi \| ^2 _{\HS 1} (g _1 ^{-1} g), \\
\label{EstLem3}
\int _\rB \int _{\rZ _x} \chi (x, z) \int _{y \in \rZ _x} &
| \dot \partial ^\rZ \widehat {g ^* A} (\theta _\alpha \psi _g ^\alpha ) | ^2
\mu _x (y) \mu _x (z) \mu _\rB (x) \\ \nonumber
\dot \leq & \| g^* A \| ^2 _{\op 0} \| \psi \| ^2 _{\HS 1} (g) .
\end{align}

Let $\rZ = \bigcup _\lambda \rZ _\lambda $ be a locally finite cover.
Then the support of $\chi \theta _\alpha $ lies in some finite sub-cover.
Let $\chi _\alpha $ be the characteristic function
$$ \chi _\alpha (x, z) = 1 \text{ if } (\chi \theta _\alpha ) (x, z) > 0, \quad 0 \text{ otherwise.}$$
Without loss of generality we may assume $\rE' | _{\rZ _\lambda } $ are all trivial.
For each $\lambda $ fix an orthonormal basis $\{ e ^\lambda _r \}$ of $\rE ' |_{ \rB _\alpha \times \rZ _\lambda }$,
and write
$$\psi ^\alpha := \sum _r u ^\lambda _r \otimes \bs ^* e ^\lambda _r .$$
One directly computes (c.f. \cite[Lemma 2.9]{So;CommTorsion}): 
\begin{align*}
\Big| \dot \nabla ^\alpha (\hat A \theta _\alpha \psi ^\alpha ) & \Big| ^2 (x, y, z) \\
=& \Big| \sum _r (\dot \nabla ^{\rE _\flat} A \theta _\alpha (u ^\lambda _r |_{\rM _\alpha \times \{ z \}})
(x, y)) \otimes \bs ^* e ^\lambda _r
+ (A \theta _\alpha u ^\lambda _r )\otimes \bs ^* (\nabla ^\rE e ^\lambda _r ) \Big| ^2 \\
\dot \leq & \sum _r
\Big( \Big| \dot \nabla ^{\rE _\flat} A \theta _\alpha (u ^\lambda _r |_{\rM _\alpha \times \{ z \}})
(x, y) \Big| ^2
+ \Big| (A \theta _\alpha u ^\lambda _r ) \otimes \bs ^* (\nabla ^\rE e ^\lambda _r ) \Big|^2
\Big).
\end{align*}
Integrating and using the same arguments as the proof of \cite [Theorem 2.17]{So;CommTorsion}, one gets the estimate
\begin{align*}
\int _{\rB _\alpha} \int _{\rZ _x} \chi (x, z) \int _{\rZ _x} &
| \dot \nabla ^\alpha \widehat {g ^* A} (\theta _\alpha \psi _g ^\alpha ) |^2
\mu _x (y) \mu _x (z) \mu _\rB (x) \\
\dot \leq \sum _\lambda  \int _{\rZ _\lambda} \int _{\rB _\alpha} \int _{\rZ _x} &
\sum _r \Big( \big| \dot \nabla ^{\rE _\flat} (g ^* A) \theta _\alpha (u ^\lambda _r |_{\rM _\alpha \times \{ z \}})
(x, y) \big|^2 \\
&+ \big| ( (g^* A) \theta _\alpha u ^\lambda _r )\otimes \bs ^* (\nabla ^\rE e ^\lambda _r ) \big|^2 \Big)
\mu _x (y) \mu _\rB (x) \mu _\rZ (z) \\
\dot \leq \sum _\lambda \int _\rB \int _{\rZ _x} \chi _\alpha & \int _{\rZ _x}
(\| g^* A \| ^2 _{\op 1} + \| g^* A \| ^2 _{\op 0} )
\big( \big| \dot \nabla ^{\hat \rE _\flat} \bx _\alpha ^* ( \theta _\alpha \psi _g ^\alpha ) \big|^2 \\
+& \big| \dot \partial ^\bs \bx _\alpha ^* ( \theta _\alpha \psi _g ^\alpha ) \big| ^2
+ \big| \dot \partial ^\bt \bx _\alpha ^* ( \theta _\alpha \psi _g ^\alpha ) \big| ^2
+ \big| \bx _\alpha ^* ( \theta _\alpha \psi _g ^\alpha ) \big| ^2 \big)
\mu _x (y) \mu _x (z) \mu _\rB (x).
\end{align*}
Equation \eqref{EstLem3} hence follows from
\begin{align*}
\chi _\alpha \big( \big| \dot \nabla ^{\hat \rE _\flat} & \bx _\alpha ^* ( \theta _\alpha \psi _g ^\alpha ) \big| ^2
+ \big| \dot \partial ^\bs \bx _\alpha ^* ( \theta _\alpha \psi _g ^\alpha ) \big| ^2
+ \big| \dot \partial ^\bt \bx _\alpha ^* ( \theta _\alpha \psi _g ^\alpha ) \big| ^2
+ \big| \bx _\alpha ^* ( \theta _\alpha \psi _g ^\alpha ) \big| ^2 \big) \\
=& \sum _{g _1 \in \rS} \chi _\alpha g _1 ^* \chi 
\big( \big| \dot \nabla ^{\hat \rE _\flat} \bx _\alpha ^* ( \theta _\alpha \psi _g ^\alpha ) \big| ^2
+ \big| \dot \partial ^\bs \bx _\alpha ^* ( \theta _\alpha \psi _g ^\alpha ) \big| ^2
+ \big| \dot \partial ^\bt \bx _\alpha ^* ( \theta _\alpha \psi _g ^\alpha ) \big| ^2
+ \big| \bx _\alpha ^* ( \theta _\alpha \psi _g ^\alpha ) \big| ^2 \big) \\ 
\dot \leq & \sum _{g _1 \in \rS} g _1 ^* \chi \sum _{i + j + k \leq 1}
|(\dot \nabla ^{\hat \rE _\flat } )^i (\dot \partial ^{\rV _\bs} )^j (\dot \partial ^{\rV _\bt} )^k  g^* \psi |^2 .
\end{align*}
Using the same arguments with $\dot \partial ^\alpha $ in place of $\dot \nabla ^\alpha $,
one gets the Equation \eqref{EstLem2}.

As for the last inequality, since $\bt ^{* } \rE |_{\rM _\alpha \times \{ z \}} $
and the connection $(\bx ^{-1} _\alpha ) ^* \nabla ^{\bs ^{*} \rE } $ is trivial along $\exp t Z _0 $,
it follows that
$$ \dot \partial ^\rZ \widehat {g^* A} (\theta _\alpha \psi _g ^\alpha )
= \widehat {g ^* A} (\dot \partial ^\rZ (\theta _\alpha \psi _g ^\alpha )) ,$$
and from which Equation \eqref{EstLem3} follows.
\end{proof}

Repeating the arguments leading to Theorem \ref{AEstCor1} for higher derivatives, 
we obtain the analogue of \cite[Corollary 2.18]{So;CommTorsion}:
\begin{cor}
\label{Main1}
For each $m = 0, 1 , \cdots $, there exists a finite subset $\rS _m \subset \rG$ and constants $C _{m, l} \geq 0$,
such that for any smooth bounded $\rG$-inavariant fiber-wise operator $A $,
$$ \| \hat A \psi \| _{\HS m} (g)
\leq \sum _{g _1 \in \rS _m} 
\big( \sum _{0 \leq l \leq m  } C _{m, l} \| A \|_{\op l} \big) \| \psi \| _{\HS m} (g _1 ^{-1} g).$$
\end{cor}

\section{The non-commutative Bismut bundle over the transformation groupoid convolution algebra}
Let $\rB$ be a compact manifold without boundary, $\rG$ be a discrete group acting on $\rB$ from the right.
One defines the transformation groupoid $\rB \rtimes \rG \rightrightarrows \rB = \rB \times \rG$ with groupoid operations
\begin{align*}
\bs (x, g) := x g, \quad \bt (x, g) &:= x, \quad (x, g) ^{-1} := (x g , g ^{-1}), \\
(x _1, g _1)(x _2, g _2) &:= (x _1 , g _1 g _2), \text{ whenever } x _1 g _1 = x _2.
\end{align*}

\begin{dfn}
Write $\cC ^* _c (\rG ) := \Span _\bbC \{ g \} _{g \in \rG} $.
Define, as a vector space, 
$$\cC ^* _c (\rB \rtimes \rG ) := C ^\infty (\rB ) \otimes _\bbC C ^* _c (\rG),$$
where $\otimes $ here denotes algebraic tensor product.
Hence elements in $\cC ^* _c (\rB \rtimes \rG )$ can be written as a finite sum
$$ \sum _{g \in \rG} f ^g g , \quad f \in C ^\infty (\rB) , g \in \rG .$$
Equip $\cC ^* _c (\rB \rtimes \rG )$ with multiplication and involution:
\begin{align*}
f g \star f' g' &:= f (g ^* f' ) (g g' ) \\
(f g ) ^\sharp &:= (g ^* \bar f) g ^{-1}.
\end{align*}
\end{dfn}

\subsection{Non-commutative differential forms}
Following \cite{Lott;EtaleGpoid}, we enlarge $\cC ^* _c (\rB \rtimes \rG )$ and consider the algebra of forms.
\begin{dfn}
The universal differential algebra over $\cC ^* _c (\rG )$ is defined to be
$$ \Omega _c ^\bullet (\rG) := \bigoplus _{k = 0} ^\infty \Omega _c ^k (\rG), \quad \Omega _c ^k (\rG) 
:= \Span _\bbC \{ d g _1 \cdots d g _k g \} _{g _1, \cdots g _k \in \rG \setminus \{ e \}, g \in \rG}$$
with multiplication 
\begin{align*}
(d g _1 \cdots d g _k g) \star (d g' _1 \cdots d g' _{k'} g')
:=& d g _1 \cdots d g _k d (g g' _1) d g' _2 \cdots d g' _{k'} g' \\
&+ \sum _{1 \leq i \leq k'-1} (-1) ^i d g _1 \cdots d g _k d g d g' _1 \cdots d (g' _i g' _{i+1}) \cdots d g' _{k'} g' \\
&+ (-1) ^{k'} d g _1 \cdots d g _k d g d g' _1 \cdots d g' _{k'-1} (g' _{k'} g').
\end{align*}
\end{dfn}

\begin{nota}
To shorten notations, we denote $k$-tuples by $g _{(k)} := (g _1, \cdots , g _k) \in \rG ^k $, and write 
\begin{align*}
d g _{(k)} &:= d g _1 \cdots d g _k \in \Omega _c ^k (\rG) \\
g _{(k)} ^* &:= g _1 ^* \cdots g _k ^* .
\end{align*}
\end{nota}

\begin{dfn}
\label{NonCommForm}
The (compactly supported) non-commutative DeRham differential forms is the vector space
$$ \Omega _c ^\bullet (\rB \rtimes \rG) 
:= \Gamma ^\infty (\wedge ^\bullet T ^* _\bbC \rB) \otimes _\bbC \Omega _c ^\bullet (\rG),$$
equipped with multiplication and involution
\begin{align*}
(\omega d g _ {(k)} g) \star (\omega' d g' _{(k')} g')
:=& (-1) ^{k \deg \omega'}
\omega \wedge (g _{(k)} ^* g ^* \omega ') d g _{(k)} g d g'_{(k')} g' ,\\
(\omega d g _1 \cdots d g _k g) ^\sharp
:=& (-1) ^k g ^{-1} d g ^{-1} _k \cdots d g ^{-1} _1 \star (-1) ^{\frac{\deg \omega (\deg \omega + 1)}{2}} \bar \omega \\
=& (-1) ^{\frac{(\deg \omega + 2 k ) (\deg \omega + 1)}{2}}
\big((g ^*)^{-1 }(g _{(k)} ^*)^{-1} \bar \omega \big) g ^{-1} d g ^{-1} _k \cdots d g ^{-1} _1 .
\end{align*}
\end{dfn}
Let $d _\rB$ be the DeRham differential on $\rB$ and define 
$d : \Omega _c ^\bullet (\rG) \to \Omega _c ^{\bullet + 1} (\rG) $,
$$ d (d g _1 \cdots d g _k g ) := (-1)^k d g _1 \cdots d g _k d g .$$
Then it is easy to see that $d _\rB + d $ is a graded derivation on $\Omega _c ^\bullet (\rB \rtimes \rG) $ of degree 1.
Hence $\Omega _c ^\bullet (\rB \rtimes \rG) $ is a graded differential algebra.

We also need $\ell ^2$ and $\ell ^2$ versions of $\Omega _c ^\bullet (\rB \rtimes \rG) $.
Let $\| \cdot \| _{C ^m}$ be the $C ^m$ norm on $\Gamma ^\infty (\wedge ^\bullet T ^* \rB)$.
We may assume that for any differential forms,
$$ \| \omega _1 \wedge \omega _2 \| _{C ^m} \leq \| \omega _1 \| _{C ^m} \| \omega _2 \| _{C ^m}.$$
\begin{dfn}
For $m = 0, 1, \cdots$, define
\begin{align*} 
\Omega ^{k , l} _{\ell ^2 , m} (\rB \rtimes \rG )
:=& \Big\{ \sum _{ d g _{(k)} g } \omega ^{d g _{(k)} g} d g _{(k)} g 
: \omega ^{d g _{(k)} g} \in \Gamma ^m (\wedge ^l T ^* \rB ), 
\sum _{ d g _{(k)} g } \| \omega ^{d g _{(k)} g} \| ^2 _{C ^m} < \infty \Big\}, \\
\Omega ^\bullet _{\ell ^2 , m} (\rB \rtimes \rG )
:=& \bigoplus _{k, l \geq 0} \Omega ^{k , l} _{\ell ^2 , m} (\rB \rtimes \rG ) ,\\
\quad \Omega ^\bullet _{\ell ^2} (\rB \rtimes \rG ) 
:=& \bigcap _m \Omega ^\bullet _{\ell ^2 , m} (\rB \rtimes \rG ). %\\
%\Omega ^{k , l} _{\ell ^2 , m} (\rB \rtimes \rG )
%:=& \Big\{ \sum _{ d g _{(k)} g } \omega ^{d g _{(k)} g} d g _{(k)} g 
%: \omega ^{d g _{(k)} g} \in \Gamma ^m (\wedge ^l T ^* \rB ), 
%\sup _{ d g _{(k)} g } \| \omega ^{d g _{(k)} g} \| _{C ^m} < \infty \Big\}, \\
%\Omega ^\bullet _{\ell ^2 , m} (\rB \rtimes \rG )
%:=& \bigoplus _{k, l \geq 0} \Omega ^{k , l} _{\ell ^2 , m} (\rB \rtimes \rG ) ,\\
%\quad \Omega ^\bullet _{\ell ^2} (\rB \rtimes \rG ) 
%:=& \bigcap _m \Omega ^\bullet _{\ell ^2 , m} (\rB \rtimes \rG ).
\end{align*}
We endow $\Omega ^{k , l} _{\ell ^2 , m} (\rB \rtimes \rG )$ with the norm
$$ \Big\| \sum _{ d g _{(k)} g } \omega ^{d g _{(k)} g} d g _{(k)} g \Big\| ^2 _{C ^m}
:= \sum _{ d g _{(k)} g } \big\| \omega ^{d g _{(k)} g} \big\| ^2 _{C ^m} ;$$
$\Omega ^\bullet _{\ell ^2 , m} (\rB \rtimes \rG )$ 
with the topology induced by degree-wise convergence,
and $\Omega^{\bullet}_{\ell ^2}( \rB\rtimes \rG)$ with the natural inductive limit topology.
\end{dfn}

Since the DeRham differential $d _\rB : \Gamma ^m (\wedge ^\bullet T ^* \rB )
\to \Gamma ^{m + 1} (\wedge ^{\bullet + 1} T ^* \rB )$ is a bounded operator,
it extends to a bounded operator from 
$\Omega ^\bullet _{\ell ^2 , m} (\rB \rtimes \rG )$ to $\Omega ^\bullet _{\ell ^2 , m - 1} (\rB \rtimes \rG )$.
Hence $d _\rB$ is a well defined continuous map on $\Omega ^\bullet _{\ell ^2} (\rB \rtimes \rG ) $.

Let 
$$[\Omega _c ^\bullet (\rB \rtimes \rG), \Omega _c ^\bullet (\rB \rtimes \rG)] 
\subseteq \Omega _c ^\bullet (\rB \rtimes \rG) \subseteq \Omega _{\ell^2} ^\bullet (\rB \rtimes \rG)$$
be the subspace spanned by graded commutators and consider
$$ \Omega _{\ell ^2} ^\bullet (\rB \rtimes \rG ) _{\Ab}:=
\Omega _{\ell ^2} ^\bullet (\rB \rtimes \rG )  
\Big/ \overline{[\Omega _c ^\bullet (\rB \rtimes \rG), \Omega _c ^\bullet (\rB \rtimes \rG)]},$$
where the over-line denotes the closure. 
Observe that the bi-grading of $\Omega _{\ell ^2} ^\bullet (\rB \rtimes \rG )$ descends to
$\Omega _{\ell ^2} ^\bullet (\rB \rtimes \rG ) _{\Ab}$:
$$ \Omega _{\ell ^2} ^\bullet (\rB \rtimes \rG ) _{\Ab}
= \bigoplus _{k, l } \Omega _{\ell ^2} ^{k, l} (\rB \rtimes \rG ) \Big/ 
\overline{[\Omega _c ^\bullet (\rB \rtimes \rG), \Omega _c ^\bullet (\rB \rtimes \rG)]}.$$
It follows the derivation property that the differential $(d _\rB + d ) $ preserves 
$ [\Omega _c ^\bullet (\rB \rtimes \rG), \Omega _c ^\bullet (\rB \rtimes \rG)]$.
Therefore $d _\rB + d $ also descends to 
$ \Omega _{\ell ^2} ^\bullet (\rB \rtimes \rG ) _{\Ab}$ with total degree 1.

Following \cite{Lott;NonCommTorsion}, 
we also consider a further quotient of $\Omega _{\ell ^2} ^\bullet (\rB \rtimes \rG) _{\Ab}$.
\begin{dfn}
Define
$$ \widetilde \Omega _{\ell ^2} ^\bullet (\rB \rtimes \rG) _{\Ab}
:= \frac{\Omega _{\ell ^2} ^\bullet (\rB \rtimes \rG) _{\Ab} }
{\oplus _k \Ker \big(d _\rB: \Omega _{\ell ^2} ^{k, k} (\rB \rtimes \rG) _{\Ab}
\to \Omega _{\ell ^2} ^{k, k+1} (\rB \rtimes \rG) _{\Ab} \big)
\bigoplus \oplus _{k > l} \Omega _{\ell ^2} ^{k, l} (\rB \rtimes \rG) _{\Ab}}.$$
The differential $(d _\rB + d ) $ descends to $ \widetilde \Omega _{\ell ^2} ^\bullet (\rB \rtimes \rG ) _{\Ab}$.
\end{dfn}
Equivalently,
one may regard 
$$\widetilde \Omega _{\ell ^2} ^\bullet (\rB \rtimes \rG ) _{\Ab}
= \big( \oplus _k \Omega _{\ell ^2} ^{k, k} (\rB \rtimes \rG) _{\Ab} \big/ \Ker d \big)
\bigoplus \oplus _{k < l} \Omega _{\ell ^2} ^{k, l} (\rB \rtimes \rG) _{\Ab},$$
by defining the differential on the
$\oplus _k \Omega _{\ell ^2} ^{k, k} (\rB \rtimes \rG) _{\Ab} \big/ \Ker d $ part to be $d _\rB$.

We shall denote the cohomologies of $(\Omega _\infty ^\bullet (\rB \rtimes \rG ) _{\Ab} , d _\rB + d)$ and 
$ (\widetilde \Omega _{\ell ^2} ^\bullet (\rB \rtimes \rG ) _{\Ab}, d _\rB + d )$ by
\begin{equation}
\mathbf H^\bullet (\Omega _{\ell ^2} ^\bullet (\rB \rtimes \rG) _{\Ab})
\text { and }
\mathbf H^\bullet (\widetilde \Omega _{\ell ^2} ^\bullet (\rB \rtimes \rG) _{\Ab}) 
\end{equation}
respectively.

\begin{rem}
In this paper, we will construct the torsion form and prove the transgression formula in 
$\widetilde \Omega _{\ell ^2} ^\bullet (\rB \rtimes \rG) _{\Ab}$.
Note that in \cite{Lott;EtaleGpoid}, the authors consider the smooth subalgebra of super-exponential decay
(with respect to the length function defined by some generators),
and prove that the trace of the heat kernel lies in that space.
Thus their result is stronger than ours.
However we need to consider the $t \to \infty $ behavior of the heat kernel.
\end{rem}

\subsection{The vector representation}
Let $\rE \to \rB$ be a (possibly infinite dimensional) contravariant vector bundle.
\begin{dfn}
\label{VectorDfn}
The vector representation $\nu$ is the left action of 
$\cC ^* _c (\rB \rtimes \rG )$ on $\Gamma ^\infty (\rE )$ defined by
$$ \nu (f g) s := f (g ^* s ), 
\quad \Forall f g \in \cC ^\star _c (\rB \rtimes \rG ), s \in \Gamma ^\infty (\rE ).$$
\end{dfn}

The vector representation extends naturally to a left action of $\Omega _c ^\bullet (\rB \rtimes \rG)$ on 
$\Omega _c ^\bullet (\rB \rtimes \rG) \otimes _{\cC ^* _c (\rB \rtimes \rG)} \Gamma ^\infty (\rE)$.
Here, we write down the action explicitly.
Denote
$$ \Omega _e ^k (\rE \rtimes \rG ) 
:= \Span \{d g _1 \cdots d g _k \} _{g _1, \cdots g _k \in \rG \setminus \{ e \}}
\otimes _\bbC \Gamma ^\infty (\wedge ^\bullet T ^* \rB \otimes \rE ).$$
Observe that 
$$ \omega d g _1 \cdots d g _k g = (-1)^{k \deg \omega } d g _1 \cdots d g _k 
\star ((g _k ^{-1} \cdots g _1 ^{-1})^* \omega ) g .$$ 
Hence $\Omega _e ^k (\rE \rtimes \rG ) $ is isomorphic to 
$\Omega _c ^\bullet (\rB \rtimes \rG) \otimes _{\cC ^* _c (\rB \rtimes \rG)} \Gamma ^\infty (\rE)$.
Moreover the action is given by
\begin{align}
\label{Module}
\nonumber
\nu (\sum _{d g_{(k)} g} \omega ^{d g _{(k)} g} d g _{(k)} g ) 
\big( \sum _{d g' _{(k')}} d g' _{(k')} & \otimes u ^{d g' _{(k')}} \big) \\ \nonumber
= \sum _{d g_{(k)} g } \sum _{ d g' _{(k')}}
\Big( (-1) ^{(k+k') \deg \omega'} & \big( d g _{(k)} d (g g' _1) d g' _2 \cdots d g' _{k'} \\
+ \sum _{1 \leq i \leq k'-1} & (-1) ^i d g _{(k)} 
d g d g' _1 \cdots d (g' _i g' _{i+1}) \cdots d g' _{k'} \big) \\ \nonumber
&\otimes (\pi ^* (g _{(k)} ^* g ^* (g' _{(k')})^* ) ^{-1} \omega ^{d g _{(k)} g}) u ^{d g' _{(k')}}  \\ \nonumber
+ (-1) ^{(k+k') \deg \omega' + k'} & d g _{(k)} d g d g' _1 \cdots d g' _{k'-1} \\ \nonumber
&\otimes (\pi ^* ( g _{(k)} ^* g ^* ( g' _1 \cdots g' _{k'-1} )^* )^{-1} \omega ^{d g _{(k)} g}) 
g _{k'} ^* u ^{d g' _{(k')}} \Big).
\end{align}

We specialize to the case of the Bismut bundle $\rE _\flat \to \rB$.
%The vector representation $\nu$ extends to an action of $\Omega ^{0 , 0} _{\ell ^1 , m}(B\rtimes G)$
%on $\cW ^{m'} (\rE)$, for $m' \leq m$, by
%\begin{equation}
%\nu ( \sum _{g \in \rG} f ^g g) s := \sum _{g \in \rG} \pi ^* f ^g (g ^* s ) .
%\end{equation}
%The sum converges absolutely in $\cW ^{m'} (\rE)$. 
We define an $\ell ^\infty$ version of $\Omega _e ^k (\rE _\flat \rtimes \rG )$:
\begin{dfn}
Define 
$$\Omega ^{k,l} _{\ell ^\infty , m} (\rE _\flat \rtimes \rG )
:= \Big\{ \sum _{d g _{(k)}} d g _{(k)} u ^{d g _{(k)}}
: u ^{d g _{(k)}} \in \cW ^m (\wedge ^l T ^* \rM \otimes \rE ),
\sup _{d g _{(k)}} \| u ^{d g _{(k)}} \| _{\cW _m} < \infty \Big\} ,$$
$$\Omega^{\bullet} _{\ell ^\infty} (\rE_{\flat} \rtimes \rG) 
:= \bigcap _{m=0} \Omega^{\bullet} _{\ell ^\infty, m} (\rE_{\flat}\rtimes \rG).$$
\end{dfn}

Clearly by extending the vector representation 
$\Omega^{\bullet} _{\ell ^\infty} (\rE_{\flat} \rtimes \rG) $ becomes a $\Omega _c ^\bullet (\rB \rtimes \rG ) $ module.

\subsection{$\Omega _c ^\bullet (\rB \rtimes \rG ) $-linear maps}
In this section, let $\rE _\flat \to \rB$ be the Bismut bundle,
induced from the fiber bundle $\rM \to \rB$ and vector bundle $\rE \to \rM$, 
with compatible $\rG$-action, 
as described in Section \ref{Geom}. 

\begin{dfn}
A $\bbC$-linear map $K : \Omega ^\bullet _e (\rE _\flat \rtimes \rG) 
\to \Omega ^\bullet _{\ell ^\infty} (\rE _\flat \rtimes \rG) $ is said to be 
$\Omega _c ^\bullet (\rB \rtimes \rG ) $-linear if for any $f \in C ^* _c (\rB \rtimes \rG)$,
$ s \in \Omega ^\bullet _e (\rE _\flat \rtimes \rG) $,
$$ \nu (f) (K s) = K (\nu (f) s).$$
\end{dfn}

We begin with writing down some necessary conditions for a $\Omega _c ^\bullet (\rB \rtimes \rG ) $-linear map $K $.
We may assume $K$ is of the form
\begin{equation}
\label{Hom0}
K s = \sum _{g _{(k)}} d g _{(k)} \otimes (g _{(k)} ^*) ^{-1} (K ^{g _{(k)}} s ),
\end{equation}
where $K ^{g _{(k)}}$ are $\bbC$-linear maps.
For the moment we regard $K$ and $K s$ as formal sums.
Then one has for any $f \in C ^\infty (\rB)$
$$ \nu (f e) (K s) 
= \sum _{g _{(k)}} d g _{(k)} \otimes ((g _{(k)} ^* )^{-1} \pi ^* f) \big((g _{(k)} ^*)^{-1} (K ^{g _{(k)}} s ) \big). $$
Therefore $\Omega _c ^\bullet (\rB \rtimes \rG ) $-linearity implies $K ^{g _{(k)}} $ are fiber-wise operators. 

Comparing $\nu (g' ) K s $ with $K (\nu (g') s) $ for arbitrary $g' \in \rG$, using Equation \eqref{Module},
one finds
\begin{align}
\label{Hom1}
\nonumber
\nu (g') (K s) =& d g' \sum _{g _{(k)}} \Big(  
(-1) ^k d g _1 \cdots d g _{k - 1} \otimes (g _1 \cdots g _{k-1}) ^{-1}) ^* K ^{g _{(k)}} s \\
&+ \sum _{i = 1} ^{k-1} (-1) ^i d g _1 \cdots d (g _{i-1} g _i) \cdots d g _k 
\otimes (g _{(k)} ^{-1} ) ^* K ^{g _{(k)}} s \Big) \\ \nonumber
&+ \sum _{g _{(k)}} d (g ' g _1) d g _2 \cdots d g _k \otimes (g _{(k)} ^*) ^{-1} K ^{g _{(k)}} s ,\\
\label{Hom2}
K (\nu (g') (s)) =& \sum _{g _{(k)}} d g _{(k)} \otimes (g _{(k)} ^{-1} ) ^* K ^{g _{(k)}} ((g') ^* s ).
\end{align}
Comparing terms in (\ref{Hom1}) and (\ref{Hom2}) not beginning with $d g'$, we get
$$ (g _{(k)} ^{-1} ) ^* (g ') ^* K ^{((g ') ^{-1} g _1, g _2 , \cdots , g _k )} s
= (g _{(k)} ^{-1} ) ^* K ^{g _{(k)}} ((g') ^* s ) .$$ 
It follows that
\begin{equation}
\label{Hom3}
K ^{(g _1 , g _2 , \cdots g _k)} = g _1 ^* \tilde K ^{(g _2 , \cdots , g _k)},
\end{equation}
for some (fiber-wise) maps $\tilde K ^{(g _2 , \cdots , g _k)} $.

The upshot of Equation \eqref{Hom3} is that it is necessary to consider infinite sums.
Here we consider the simplest example where Equation \eqref{Hom0} makes sense.

\begin{exam}
\label{LocalTensor}
Suppose that $\tilde K ^{(g _2 , \cdots , g _k)}$ in Equation \eqref{Hom3} are compactly supported tensors,
and such that only finitely many $\tilde K ^{(g _2 , \cdots , g _k)}$ differ from zero.
Then for any $s \in \Gamma ^\infty _c (\rE)$ there are at most finitely many $g _1 \in \rG$
such that $(g _1 ^* \tilde K ^{(g _2 , \cdots , g _k)}) s \neq 0 $.
In other words, $K$ is a well defined map from $ \Gamma ^\infty _c (\rE) $ to itself.
It is clear that $K$ furthermore extends to $\Omega ^\bullet _{ \ell ^\infty} (\rE _{\flat}\rtimes \rG)$.

Specializing to the case $k = 1$.
Comparing the $d g'$ term in (\ref{Hom1}) and (\ref{Hom2}) and using Equation (\ref{Hom3}),
one gets
\begin{equation}
\label{Circular}
\sum _{g _1 \in \rG} K ^{g _1} s = \sum _{g _1 \in \rG} (g _1 ^* \tilde K) s = 0 , \quad \Forall s .
\end{equation}
Note that one gets the same equation for all $g '$.
Thus a concrete example for a $\Omega _c ^\bullet (\rB \rtimes \rG ) $ is given by $\tilde K = d \chi $,
where $\chi \in C ^\infty _c (\rG)$ is defined in Equation \eqref{Partition}.
\end{exam}

Suppose that 
$ K = \sum _{g _{(k)}} d g _{(k)} \otimes (g _{(k)}^*) ^{-1} K ^{g _{(k)}} $ and 
$ K' = \sum _{g' _{(k')}} d g' _{(k)} \otimes (g _{(k)}^{\prime *}) ^{-1} K ^{\prime g' _{(k')}} 
: \Omega ^\bullet _{\ell ^\infty} (\rE_{\flat} \rtimes \rG ) \to \Omega ^\bullet _{\ell ^\infty} (\rE_{\flat} \rtimes \rG )$.
Then the composition is well defined.
It is explicitly given by
$$ (K \star K') s
:= \sum _{g _{(k)}, g' _{(k')}} d g' _{(k')} d g _{(k)}
\otimes (g _{(k)} ^* g ^{\prime *} _{(k')} ) ^{-1} ( (g _{(k')} ^{\prime *} K ^{g _{(k)}}) K ^{\prime g' _{(k')}} s).$$  

\begin{rem}
In this paper, we will mainly consider the sub-algebra of operators generated by 
$\Psi ^{- \infty} _{\infty} (\rM \times _\rB \rM , \rE) ^\rG $ and tensors as in Example \ref{LocalTensor}.
\end{rem}

\subsection{Hilbert-Schmit norms on $\Omega _c ^\bullet (\rB \rtimes \rG ) $-linear operators}
In this section, we expand the (semi)-norm in Definition \ref{NormDfn}.

\begin{dfn}
Define 
$$ \Psi ^{- \infty} _{\ell ^2 , m} (\rM \times _\rB \rM , \rE ) $$
to be the set of $\Omega _c ^\bullet (\rB \rtimes \rG ) $-linear operators of the form
$$ K = \sum _{g _{(k)}} d g _{(k)} \otimes (g _{(k)}^*) ^{-1} K ^{g _{(k)}} 
: \Omega ^\bullet _{e} (\rE _\flat \rtimes \rG ) \to \Omega ^\bullet _{\ell ^\infty} (\rE _\flat \rtimes \rG ),$$
such that $K ^{g _{(k)}} \in \varPsi ^{- \infty} _{\infty} (\rM \times _\rB \rM , \rE )$
satisfy the estimate
$$ \sum _{g _{(k)}} \| K ^{g _{(k)}} \| _{\HS m} (e) < \infty .$$
For any $K \in \Psi ^{- \infty} _{\ell ^2 , m} (\rM \times _\rB \rM , \rE ) $ define 
\begin{equation}
\label{HSDfn}
\| K \| _{\HS m} := \sum _{g _{(k)}} \| K ^{g _{(k)}} \| _{\HS m} (e) .
\end{equation}
Also, we denote 
$$ \Psi ^{- \infty} _{\ell ^2} (\rM \times _\rB \rM , \rE ) 
:= \bigcap _{m=0} ^\infty \Psi ^{- \infty} _{\ell ^2 , m} (\rM \times _\rB \rM , \rE ) .$$
\end{dfn}

Here we derive a formula for $\| \cdot \| _{\HS m}$.
Write $K ^{(g _1 , g _2 , \cdots g _k)} = g _1 ^* \tilde K ^{(g _2 , \cdots , g _k)}$.
Then 
\begin{align}
\label{NoPartition}
\nonumber
\| K \| ^2 _{\HS m} = \sum _{g _{(k)}} \| g _1 ^* \tilde K ^{(g _2 , \cdots , g _k)} \| ^2 _{\HS m} (e) & \\ \nonumber
= \sum _{i+j+k \leq m} \sum _{g _{(k)}} \Big( \int _\rB \int _{\rZ _x } \chi (x, z) & \int _{\rZ _x } \big|
(\dot \nabla ^{\hat \rE _\flat } )^i (\dot \partial ^\bs )^j (\dot \partial ^\bt)^k 
(g _1 ^* \tilde K ^{(g _2 , \cdots , g _k)} ) \big| ^2
(x, y, z)  \\
& \mu _x (y) \mu _x (z) \mu _\rB (x) \Big) \\ \nonumber
= \sum _{i+j+k \leq m} \sum _{g _2, \cdots g _k} \Big( \int _\rB \int _{\rZ _x } \int _{\rZ _x } \big| &
(\dot \nabla ^{\hat \rE _\flat } )^i (\dot \partial ^\bs )^j (\dot \partial ^\bt)^k 
(\tilde K ^{(g _2 , \cdots , g _k)} ) \big| ^2
(x, y, z) \\ \nonumber
& \mu _x (y) \mu _x (z) \mu _\rB (x) \Big) .
\end{align}
Clearly $\| \cdot \| _{\HS m} $ is positive definite,
therefore it defines a norm on $ \Psi ^{- \infty} _{\ell ^2 , m} (\rM \times _\rB \rM , \rE )$.

Next we generalize Corollary \ref{Main1} to 
$ \Psi ^{- \infty} _{\ell ^2 , m} (\rM \times _\rB \rM , \rE )$.

\begin{thm}
\label{Main4}
For any smooth, bounded $\rG$ invariant operator $A$, 
and $K \in \Psi ^{- \infty} _{\ell ^2 , m} (\rM \times _\rB \rM , \rE )$,
$$ A \star K , K \star A \in \Psi ^{- \infty} _{\ell ^2 , m} (\rM \times _\rB \rM , \rE ).$$
Moreover, there are constants $C_{m,l}>0$ such that
\begin{align*}
\| A \star K \| _{\HS m} \leq &
\big( \sum _{0 \leq l \leq m  } C _{m, l} \| A \|_{\op l} \big) \| K \| _{\HS m} ,\\
\| K \star A \| _{\HS m} \leq &
\big( \sum _{0 \leq l \leq m  } C _{m, l} \| A \|_{\op l} \big) \| K \| _{\HS m}.
\end{align*}
\end{thm}
\begin{proof}
Since $A $ is $\rG$-invariant, one has 
$$ A \star K = \sum _{g _{(k)}} d g _{(k)} \otimes (g _{(k)}^*) ^{-1} (A K ^{g _{(k)}}) .$$
The first inequality follows immediately from Corollary \ref{Main1}.

As for the second inequality, we use Equation \eqref{NoPartition} to get
\begin{align*}
\sum _{g _{(k)}} \big\| g _1 ^* (\tilde K ^{(g _2 , \cdots , g _k)} A ) \big\| ^2 _{\HS m} (e) \\
= \sum _{i+j+k \leq m} \sum _{g _2, \cdots g _k} \Big( \int _\rB \int _{\rZ _x } \int _{\rZ _x } \big| &
(\dot \nabla ^{\hat \rE _\flat } )^i (\dot \partial ^\bs )^j (\dot \partial ^\bt)^k 
(\tilde K ^{(g _2 , \cdots , g _k)} A ) \big| ^2
(x, y, z) \mu _x (y) \mu _x (z) \mu _\rB (x) \Big),
\end{align*}
and observe that one can interchange the roles of $y$ and $z$ in the last line.
\end{proof}

Similar to Theorem \ref{Main4}, we have
\begin{lem}
\label{Main5}
For any $F = \sum _{d g _{(k)}} d g _{(k)} g _1 ^* \tilde F ^{(g _2 , \cdots , g _k)}$ as in Example \ref{LocalTensor},
$K' = \sum _{g' _{(k')}} d g' _{(k)} \otimes (g _{(k)}^{\prime *}) ^{-1} K ^{\prime g' _{(k')}}
\in \Psi ^{- \infty} _{\ell ^2 , m} (\rM \times _\rB \rM , \rE )$, then
$$ F \star K , K \star F \in \Psi ^{- \infty} _{\ell ^2 , m} (\rM \times _\rB \rM , \rE ).$$
Moreover there exists $C' _m > 0$ (depending only on $F$) such that
\begin{align*}
\| F \star K' \| _{\HS m} \leq &
C' _m \| K' \| _{\HS m}, \\
\| K' \star F \| _{\HS m} \leq &
C' _m \| K' \| _{\HS m}.
\end{align*}
\end{lem}
\begin{proof}
We only prove the first inequality. The second is similar.
Since we have 
\begin{align*}
F \star K' =& 
\sum _{g _{(k)}, g' _{(k')}} d g' _{(k')} d g _{(k)}
\otimes (g _{(k)} ^* g ^{\prime *} _{(k')} ) ^{-1} \big( (g _{(k')} ^{\prime *} F ^{g _{(k)}}) 
(g _1 ^{\prime *} \tilde K ^{\prime g' _{(k')}} ) \big) ,\\
\| F \star K' \| ^2 _{\HS m} \leq&
 \sum _{g _2, \cdots g _k}  \sum _{g' _2, \cdots g' _k} \sum _g \sum _{i+j+k \leq m} 
\int _\rB \int _{\rZ _x } \int _{\rZ _x } \\
& \big| (\dot \nabla ^{\hat \rE _\flat } )^i (\dot \partial ^\bs )^j (\dot \partial ^\bt)^k 
( (g ^* \tilde F ^{(g _2, \cdots g _k)}) \tilde K ^{(g' _2 , \cdots , g' _k)} ) \big| ^2
(x, y, z) \mu _x (y) \mu _x (z) \mu _\rB (x).
\end{align*}
The integrand is bounded by
$$ g ^* \tilde \chi (x, y) \| \tilde F ^{(g _2, \cdots g _k)} \| _{C ^m} 
\big| (\dot \nabla ^{\hat \rE _\flat } )^i (\dot \partial ^\bs )^j (\dot \partial ^\bt)^k 
( \tilde K ^{(g' _2 , \cdots , g' _k)} ) \big| ^2,$$
for some compactly supported function $\tilde \chi \geq 0$,
which depends only on the support of $ \tilde F ^{(g _2, \cdots g _k)}$.
Therefore $\sum _g g^* \tilde \chi $ is bounded.
Our inequality then follows from \eqref{NoPartition}.
\end{proof}

\subsection{Trace class operators.}
\begin{dfn}
\label{trDfn1}
Given any $\Omega ^\bullet _c (\rM \rtimes \rG ) $-linear map 
$K s = \sum _{g _{(k)}} d g _{(k)} \otimes (g _{(k)} ^* ) ^{-1} (K ^{g _{(k)}} s) $,
where $ K ^{g _{(k)}} \in \Psi ^{- \infty } (\rM \times _\rB \rM , \rE) $.
We say that $K$ is of trace class if for all $m$
\begin{equation}
\label{TrDfn}
\sum _{g _{(k)}} \Big\| \int _{\rZ _x} \chi (x, z) \tr (K ^{g _{(k)}} (x, z, z)) \mu _x (z) \Big\|^2 _{C ^m} 
< \infty .
\end{equation}
For a trace class operator,  we define 
\begin{align}
\label{TrDfn2} 
\tr _\Psi (K) 
:=& \Ab \Big(\sum _{g _{(k)}} 
\int _{\rZ _x} \chi (x, z) \tr (K ^{g _{(k)}} (x, z, z)) \mu _x (z) d g _{(k)} (g _k ^{-1} \cdots g _1 ^{-1} ) \Big) \\ \nonumber
& \in \Omega _{\ell ^2} ^\bullet (\rM \rtimes \rG ) _{\Ab},
\end{align}
where $\tr $ is the point-wise trace (c.f. \cite[(3.22)]{Lott;EtaleGpoid}).

\begin{rem}
Using similar arguments as the proof of Lemma \ref{dTr} below, 
one can show that $\tr _\Psi $ does not depend on $\chi$. 
\end{rem}

If $\rE ^\bullet $ is a $\bbZ$ graded vector bundle, define the super-trace $\str _\Psi $ as in 
\eqref{TrDfn2} with $\tr (\cdot)$ replaced by the super-trace $\str (\cdot)$.
\end{dfn}

It is well known that $\tr _\Psi $ is indeed a trace.
\begin{lem}
\cite[Proposition 3]{Lott;EtaleGpoid}
For any $\Omega _c ^\bullet (\rB \rtimes \rG)$-linear, trace class smoothing operators $K _1, K _2$,
$\tr _\Psi [K _1 , K _2] = 0 .$
\end{lem}

Also one has the identity:
\begin{lem}
\label{dTr}
(c.f. \cite[Proposition 3]{Lott;FoliationInd})
Given any $\rG$-invariant connection $\nabla $ on $\rE _\flat ^\bullet$, 
and $\Omega _c ^\bullet (\rB \rtimes \rG)$-linear smoothing operator
$ K = \sum _{d g _{(k)}} d g _{(k)} (g _{(k)} ^*) K ^{d g _{(k)}} $ of trace class,
$$ (d _\rB  + d) (\tr _\Psi (K)) = \tr _\Psi ([\nabla + \nabla ^\rG , K]).$$  
\end{lem}
\begin{proof}
For simplicity we only prove the case when $k = 1$, the other cases are similar.
It is well know that 
\begin{align*}
d _\rB (\tr _\Psi (K)) =& \sum _{ d g_1 } \int _{\rZ _x} \chi \tr \big([\nabla , K ^{d g_1}] (x, z, z) \big) \mu _x (z) d g _1 g _1 ^{-1} \\
&+ \sum _{ d g_1 } \int _{\rZ _x} ( d _\rH \chi ) \tr \big([\nabla , K ^{d g _1}] (x, z, z) \big) \mu _x (z) d g _1 g _1 ^{-1}.
\end{align*}
We must prove the second integral vanishes.
The operator $[\nabla , K]$ is also $\Omega _c ^\bullet (\rB \rtimes \rG)$-linear.
By \eqref{Hom3}, we may write $[\nabla , K ^{d g _1}] = g _1 ^* \tilde \varPsi $ for some smoothing operator $\tilde \varPsi$.
Consider for arbitrary $g \in \rG$
\begin{align*}
\int _{\rZ _x} & (g ^* \chi )( d _\rH \chi ) \tr \big(g _1 ^* \tilde \varPsi (x, z, z) \big) \mu _x (z) d g _1 g _1 ^{-1} \\
=& g ^* \int _{\rZ _x} \chi ( (g ^{-1}) ^* d _\rH \chi ) \tr \big( (g ^{-1} g _1) ^* \tilde \varPsi (x, z, z) \big) \mu _x (z) d g _1 g _1 ^{-1} \\
=& - \int _{\rZ _x} \chi ( (g ^{-1}) ^* d _\rH \chi ) \tr \big( (g ^{-1} g _1) ^* \tilde \varPsi (x, z, z) \big) \mu _x (z) 
g ^{-1} d g _1 g _1 ^{-1} g \\
& \mod [\Omega _c ^\bullet (\rB \rtimes \rG) , \Omega _c ^\bullet (\rB \rtimes \rG)] \\
=& - \int _{\rZ _x} \chi ( (g ^{-1}) ^* d _\rH \chi ) \tr \big( (g ^{-1} g _1) ^* \tilde \varPsi (x, z, z) \big) \mu _x (z) 
\big( d (g ^{-1} g _1) g _1 ^{-1} g - (d g ^{-1}) g \big)
\end{align*}
Summing over all $g \in \rG , g _1 \in \rG \setminus \{ e \}$ and using \eqref{Circular},
it follows that
$$ \sum _{ d g_1 } \int _{\rZ _x} ( d _\rH \chi ) \tr \big([\nabla , K ^{d g _1}] (x, z, z) \big) \mu _x (z) d g _1 g _1 ^{-1}
= 0 \in \Omega _{\ell ^2} ^\bullet (\rB \rtimes \rG ) _{\Ab}.$$
On the other hand, it is straightforward to compute
$$ d (\tr _\Psi (K)) = \tr _\Psi ([\nabla ^\rG , K]).$$
Hence the lemma.
\end{proof}

To construct examples of trace class operators, one uses the following lemma:
\begin{lem}
\label{TrClassCri}
For any $F s = \sum _{g _{(k)}} d g _{(k)} \otimes (g _{(k)} ^*) ^{-1} (F ^{g _{(k)}} s )$ 
as in Example \ref{LocalTensor},
$K \in \Psi ^{- \infty} _{\infty} (\rM \times _\rB \rM , \rE ^\bullet) ^\rG $
and $K' = \sum _{g' _{(k')}} d g' _{(k)} \otimes (g _{(k)}^{\prime *}) ^{-1} K ^{\prime g' _{(k')}}
\in \Psi ^{- \infty} _{\ell ^2 } (\rM \times _\rB \rM , \rE )$. Then
$K \star F \star K'$ is a trace class operator.
\end{lem}
\begin{proof}
We use similar arguments as the proof of \cite[Theorem 4.6]{So;CommTorsion}.
For simplicity we only consider $k = 1$. The general cases are similar.
Denote by $\tilde \theta $ the characteristic function of support of $\tilde F$, and write 
$$ G (x, y, z) := K (x, z, y) ( (g' _{(k')} g )^* \tilde F) (x, y) K ^{\prime g' _{(k')}} (x, y, z) .$$
Then by the Cauchy-Schwarz inequality
\begin{align*}
\Big\| \int _{\rZ _x} & \chi (x, z) \tr \Big( \int _{\rZ _x} G (x, y, z) \mu _x (y) \Big) \mu _x (z) \Big\| ^2 _{\cL ^2 (\rB)} \\
\dot \leq &
\| \tilde F \| ^2 _{C ^0} 
\Big( \int _\rB \int _{\rZ _x } \chi \int _{\rZ _x } 
((g' _{(k')} g) ^* \tilde \theta) \big| K (x, z, y) \big| ^2 \mu _x (y) \mu _x (z) \mu _\rB (x) \Big) \\
& \times \Big( \int _\rB \int _{\rZ _x } \chi \int _{\rZ _x }
((g' _{(k')} g) ^* \tilde \theta) \big| K ^{\prime g' _{(k')}} (x, y, z) \big| ^2 \mu _x (y) \mu _x (z) \mu _\rB (x) \Big).
\end{align*}
Sum over all $g$ and then $g' _{(k')}$, 
and using the fact that for each $g' _{(k')} $ fixed, 
the support of $(g' _{(k')} g) ^* \tilde F $ is a locally finite cover of $\rM $, 
one gets
\begin{align*}
\sum _{g' _{(k')}} \sum _{g} \Big\| \int _{\rZ _x} & \chi (x, z) \tr \Big( \int _{\rZ _x} 
K (x, z, y) ((g' _{(k')} g )^* \tilde F) (x, y) K ^{\prime g' _{(k')}} (x, y, z) \mu _x (y) \Big) \mu _x (z) \Big\| ^2 _{\cL ^2 (\rB)} \\
\dot \leq & \sum _{g' _{(k')}} \| \tilde F \| ^2 _{C ^0 } \| K \| ^2 _{\HS 0} (e) \| K ^{\prime g' _{(k')}} \| ^2 _{\HS 0} (e) \\
=& \| \tilde F \| ^2 _{C ^0 } \| K \| ^2 _{\HS 0} \| K ^{\prime g' _{(k')}} \| ^2 _{\HS 0} .
\end{align*}
We turn to estimate its derivative.
Differentiating under the integral sign, one gets
\begin{align*}
\Big| \nabla ^{\wedge ^\bullet T ^* \rB} \Big( \int _{\rZ _x} & \chi (x, z) \tr \Big( \int _{\rZ _x} 
G \mu _x (y) \Big) \mu _x (z) \Big) \Big| \\
\leq & \int _{\rZ _x} (L ^\flat \chi (x, z)) \tr \Big( \int _{\rZ _x} 
G \mu _x (y) \Big) \mu _x (z) \\
&+ \int _{\rZ _x} \chi (x, z) \Big( \nabla ^{(\pi ^* \wedge ^\bullet T \rB) _\flat } 
\tr \Big( \int _{\rZ _x} G \mu _x (y) \Big) \Big) \mu _x (z) \\
&+ \int _{\rZ _x} \chi (x, z) \tr \Big( \int _{\rZ _x} G \mu _x (y) \Big) (L ^\flat \mu _x (z)),
\end{align*}
where $L ^\flat $ is the $(1, 0)$ component of $D _\rB$ in Definition \ref{BismutDfn} (with $\rE $ trivial),
which is a $C ^\infty (\rB)$ connection.
Since $|L ^\flat \mu _x (z) |$ equals $|\mu _x (z) |$ multiplied by some bounded function,
it follows that 
$$ \sum _{g} \Big\| \int _{\rZ _x} \chi \tr \Big( \int _{\rZ _x} 
G \mu _x (y) \Big) (L ^\flat \mu _x (z)) \Big\| ^2 _{\cL ^2 (\rB)} 
\dot \leq \| \tilde F \| ^2 _{C ^0 } \| K \| ^2 _{\HS 0} (e) \| K ^{\prime g' _{(k)}} \| ^2 _{\HS 0} (e).$$
Similarly, write 
$ L ^\flat \chi (x, z) = \sum _{g' \in \rG } (g ^{\prime *} \chi) (x, z) (L ^\flat \chi ) (x, z).$
The sum is finite because $L ^\flat \chi $ is compactly supported.
Then
\begin{align*}
\sum _{g} \Big\| \int _{\rZ _x} & (L ^\flat \chi (x, z)) \tr \Big( \int _{\rZ _x} 
G \mu _x (y) \Big) \mu _x (z) \Big\| ^2 _{\cL ^2 (\rB)} \\
\leq & \sum _{g' \in \rS } \sum _{g} \Big\| \int _{\rZ _x} \chi (x, z) ((g ^{\prime *}) ^{-1} L _{X ^\rH } \chi ) 
\tr \Big( \int _{\rZ _x} (g ^{\prime *}) ^{-1} G \mu _x (y) \Big) \mu _x (z) \Big\| ^2 _{\cL ^2 (\rB)} \\
\dot \leq & \sum _{g' \in \rS } \| \tilde F \| ^2 _{C ^0 } \| K \| ^2 _{\HS 0} (e) \| K ^{\prime g' _{(k)}} \| ^2 _{\HS 0} ((g ' ) ^{-1}).
\end{align*}
Lastly, by the Leibniz rule, we have
\begin{align*}
\Big| \nabla ^{(\pi ^* \wedge ^\bullet T \rB) _\flat } & \tr \Big( \int _{\rZ _x} G \mu _x (y) \Big) \Big|^2 \\
\dot \leq & \int _{\rZ _x} (|\dot \nabla ^{\hat \rE _\flat } K | +| K |)^2
(|\nabla ^{\hat \rE} g^* F | + | g^* F |)^2 (|\dot \nabla ^{\hat \rE _\flat } K' | +| K' |)^2 \mu _x (y).
\end{align*}
It follows that 
\begin{align*}
\sum _{ g} \Big\| \int _{\rZ _x} & \chi (x, z) \Big(  \nabla ^{(\pi ^* \wedge ^\bullet T \rB) _\flat } \tr 
\Big( \int _{\rZ _x} G \mu _x (y) \Big) \Big) \mu _x (z) \Big\| ^2 _{\cL ^2 (\rB)} \\
\dot \leq & \sum _{g' \in \rS } \| \tilde F \| ^2 _{C ^1 } \| K \|^2  _{\HS 1} (e) \| K ^{\prime g' _{(k)}} \| ^2 _{\HS 1} ((g ' ) ^{-1}). 
\end{align*}
Adding these estimates together, we have proven that
\begin{align}
\label{TrEst1}
\sum _{ g} \Big\| & \nabla ^{\wedge ^\bullet T ^* \rB} \Big( \int _{\rZ _x} \chi (x, z) \tr \Big( \int _{\rZ _x} 
G \mu _x (y) \Big) \mu _x (z) \Big) \Big\| ^2 _{\cL ^2 (\rB)} \\ \nonumber
& \dot \leq \sum _{g' \in \rS } \| \tilde F \| ^2 _{C ^1 } \| K \| ^2 _{\HS 1} (e) \| K ^{\prime g' _{(k')}} \| ^2 _{\HS 1} ((g ' ) ^{-1}) \\
\sum _{g _{(k)}} \sum _g \Big\| & \int _{\rZ _x} \chi (x, z) 
\tr \big( (K ( (g' _{(k)})^* F ^g ) K ^{\prime g' _{(k)}}) (x, z, z) \big) \mu _x (z) \Big\|^2 _1 \\ 
\nonumber
& \dot \leq \| \tilde F \| ^2 _{C ^1 } \| K \|^2  _{\HS 1} \| K' \| ^2 _{\HS 1}.
\end{align}
Clearly, the same arguments for Equation \eqref{TrEst1} can be repeated for all derivatives, 
and one gets for any $m$,
\begin{align}
\label{TrEst2}
\sum _{g _{(k)}} \sum _g \Big\| & \int _{\rZ _x} \chi (x, z) 
\tr \big( (K ( (g' _{(k)})^* F ^g ) K ^{\prime g' _{(k)}} ) (x, z, z) \big) \mu _x (z) \Big\|^2 _m \\ 
\nonumber
& \dot \leq \| \tilde F \| ^2 _{C ^m } \| K \|^2  _{\HS m} \| K' \| ^2 _{\HS m},
%\sum _{d g} \Big\| & (\dot \nabla ^{\wedge ^\bullet T ^* \rB})^m 
%\Big( \int _{\rZ _x} \chi (x, z) \tr \Big( \int _{\rZ _x} 
%G ^g \mu _x (y) \Big) \mu _x (z) \Big) \Big\| _{\cL ^1 (\rB)} \\ \nonumber
%& \dot \leq \sum _{g' \in \rS _m} \| \tilde F \| _{C ^m } \| K \| _{\HS m} (e) \| K' \| _{\HS m} ((g ' ) ^{-1}),
\end{align}
for some finite sets $\rS _m$.
By the Sobolev embedding theorem (for Sobolev spaces on the compact manifold $\rB$), it follows that for any $m'$, 
there exists $m$ such that
\begin{align}
\label{TrEst3}
\sum _{g _{(k)}} \sum _g \Big\| & \int _{\rZ _x} \chi (x, z) 
\tr \big( (K ( (g' _{(k)})^* F ^g ) K ^{\prime g' _{(k)}} ) (x, z, z) \big) \mu _x (z) \Big\|^2 _{ C^ {m'}} \\ 
\nonumber
& \dot \leq \| \tilde F \| ^2 _{C ^m } \| K \|^2  _{\HS m} \| K' \| ^2 _{\HS m}.
\end{align}
Hence $K \star F \star K'$ satisfies \eqref{TrDfn}. 
\end{proof}

\subsection{The Bismut super-connection over $\rB \rtimes \rG$}
In this section, we generalize the Bismut super-connection to the convolution algebra. 
Let $\rE \to \rM$ be a flat $\rG$-contravariant vector bundle with a flat connection $\nabla ^\rE$. 
One regards $\rE _\flat $ as a contravariant vector bundle over $\rB$.
Hence one has a $\cC ^* (\rB \rtimes \rG) $ module $\Gamma ^\infty _c (\rE _\flat ) $ 
by Definition \ref{VectorDfn}.

\begin{dfn}
Let $\chi \in C ^\infty _c (\rM)$ be as in Equation \eqref{Partition}.
Define the operator
$\nabla ^\rG : \Gamma _c ^\infty (\rE _\flat) 
\to \Omega ^{1} _{e} (\rE _\flat \rtimes \rG )$
by the formula
\begin{equation}
\label{NonCommConn}
\nabla ^\rG u := \sum _{g \in \rG} d g \otimes \chi ((g ^{-1}) ^* u ) .
\end{equation}
\end{dfn}

\begin{lem}
\label{ConnConstruction}
The operator $D _\rB + \nabla ^\rG$ is a connection of the
$\cC ^\infty (\rB \rtimes \rG)$ module $\Gamma ^\infty (\rE _\flat )$.
\end{lem}
\begin{proof}
It suffices to check $\nabla ^\rG (\nu (f g ) u) = \nu (f g ) (\nabla ^\rG u ) + \nu (f d g ) u $
for any $f g \in \cC ^\infty (\rB \rtimes \rG), u \in \Gamma ^\infty (\rE _\flat)$.
Indeed one has
\begin{align*}
\nabla ^\rG (\nu (f g ) u)
=& \sum _{g _1 \in \rG } d g _1 \otimes \chi ((g _1 ^{-1} )^* f )((g _1 ^{-1} g ) ^* u ), \\
\nu (f g) (\nabla ^\rG u ) =& - \sum _{g _1 \in \rG } d g \otimes ((g ^{-1} ) ^* f ) (g _1 ^* \chi ) u
+ \sum _{g _1 \in \rG } d (g g _1 ) \otimes (((g g _1 ) ^{-1} ) ^* f ) \chi ((g _1 ^{-1} ) ^* u ) \\
=& - \nu (f d g ) u +\nabla ^\rG (\nu (f g ) u) . \qedhere
\end{align*}
\end{proof}

The $\rG$-invariant inner product $\langle \, , \rangle _\rE $ on $\rE _\flat$ defined in Equation \eqref{BismutInner}
induces a $\cC ^* (\rB \rtimes \rG )$ valued inner product on $\Gamma ^\infty (\rE _\flat)$ by the formula
\begin{equation}
\label{NCInner}
\langle s _1 , s _2 \rangle _{\rE _\flat \rtimes \rG} (x, g )
:= \langle s _1 , (g ^* s _2 ) \rangle _{\rE _\flat} (x).
\end{equation}
Note that for any 
$s _1 , s _2 \in \Gamma ^\infty _c (\rE _\flat)$,
$ \langle s _1 , (g ^* s _2 ) \rangle _{\rE _\flat} (x) = 0$ for all but finitely many $g $.
This new inner product 
$\langle \, , \rangle _{\rE _\flat \rtimes \rG}$ defines a pre-Hilbert $C ^* _c (\rB \rtimes \rG )$ module structure. 
More precisely:
\begin{lem}
For any $f \in C ^* _c (\rB \rtimes \rG ), s _1 , s _2 \in \Gamma _c ^\infty (\rE _\flat) $,
\begin{align*}
\langle s _2 , s _1 \rangle _{\rE _\flat \rtimes \rG} 
=& (\langle s _1 , s _2 \rangle _{\rE _\flat \rtimes \rG}) ^\sharp ,\\
f \star \langle s _1 , s _2 \rangle _{\rE _\flat \rtimes \rG} 
=& \langle \nu (f) (s _1 ) , s _2 \rangle _{\rE _\flat \rtimes \rG}.
\end{align*}
\end{lem} 
\begin{proof}
Equation \eqref{NCInner} is equivalent to
$ \langle s _1 , s _2 \rangle = \sum _{g _1 \in \rG} \langle s _1 , (g _1 ^* s _2 ) \rangle _{\rE _\flat} g _1 .$
Hence one verifies the first formula:
$$ ( \langle s _1 , s _2 \rangle _{\rE _\flat \rtimes \rG}) ^\sharp
= \sum _{g _1 \in \rG} ( g _1 ^{-1} )^* \langle (g _1 ^* s _2 ), s _1 \rangle _{\rE _\flat} g _1 ^{-1} 
= \langle s _2 , s _1 \rangle _{\rE _\flat \rtimes \rG} .$$
As for the second equality, it suffices to verify for any $f ^{g _0 } g _0 \in \cC ^* _c (\rM \rtimes \rG ),$
$$ (f ^{g _0 } g _0 ) \star \langle s _1 , s _2 \rangle _{\rE _\flat \rtimes \rG}
= \sum _{g _1 \in \rG} f ^{g _0 } g _0 ^* (\langle s _1 , g _1 ^* s _2 \rangle _{\rE _\flat} ) g _0 g _1 \\
= \sum _{g _1 \in \rG} \langle f ^{g _0 } (g _0 ^* s _1 ) , (g _0 g _1 )^* s _2 \rangle _{\rE _\flat} g _0 g _1 .$$
Relabeling $g _2 = g _0 g _1 $ yields the desired result.
\end{proof} 

One extends naturally the inner product $\langle \, , \rangle _{\rE _\flat \rtimes \rG}$
to $\Omega _e (\rE _\flat \rtimes \rG)$,
and defines the notion of adjoint connection by Equation \eqref{CommAd}
(with $\langle \, , \rangle _{\rE _\flat \rtimes \rG}$ in place of $\langle \, , \rangle _{\rE _\flat}$).

\begin{lem}
For any sections $u _1, u _2 \in \Gamma ^\infty _c (\rE ^\bullet _\flat )$, we have
\begin{equation}
\label{AdConn}
(d _\rB + d) \langle u _1 , u _2 \rangle  
= \langle (D _\rB + \nabla ^\rG ) u _1 , u _2 \rangle  
- \langle u _1 , (D' _\rB + \nabla ^\rG ) u _2 \rangle .
\end{equation}
In other words, the adjoint connection of $D _\rB + \nabla ^\rG$ with respect to the 
$\cC ^* (\rB \rtimes \rG )$ valued inner product 
$\langle \, , \rangle _{\rE _\flat \rtimes \rG}$ is $D_\rB ' + \nabla ^\rG$.
\end{lem}
\begin{proof}
Since the DeRham differential $d _\rB$ commutes with pull-back, 
it suffices to check
\begin{align*}
\langle \nabla ^\rG u _1 , u _2 \rangle 
=& \sum _{g _0 , g _1 \in \rG } 
d g _0 \star \langle \chi ( g _0 ^{-1} )^* u _1 , g _1 ^* u _2 \rangle _{\rE _\flat} g _1 \\
=& \sum _{g _0 , g _1 \in \rG } \langle (g _0 ^* \chi) u _1 , (g _0 g _1) ^* u _2 \rangle _{\rE _\flat} (d g _0 )  g _1 ,\\
\langle u _1 , \nabla ^\rG u _2 \rangle
=& - \sum _{g _1 \in \rG } 
\langle u _1 , \chi (g _1 ^{-1})^* u _2 \rangle _{\rE _\flat \rtimes \rG} \star d g _1 ^{-1} \\
=& - \sum _{g _0 , g _1 \in \rG } 
\langle u _1 , g _0 ^* (\chi (g _1 ^{-1})^* u _2) \rangle _{\rE _\flat} (g _0 d g _1 ^{-1}), \\
\langle \nabla ^\rG u _1 , u _2 \rangle - \langle u _1 , \nabla ^\rG u _2 \rangle
=& \sum _{g _0 , g _1 \in \rG } \langle (g _0 ^* \chi) u _1 , (g _0 g _1 )^* u _2 \rangle _{\rE _\flat} d (g _0 g _1 )
= d \langle u _1 , u _2 \rangle _{\rE _\flat \rtimes \rG} . \qedhere
\end{align*}
\end{proof}

Summarizing the results in this section, we define:
\begin{dfn}
The (non-commutative) Bismut super-connection on the Bismut bundle is the connection
$$ D := D _\rB + \nabla ^\rG ;$$
its adjoint connection is
$$ D' := D ' _\rB + \nabla ^\rG.$$
\end{dfn}

\subsection{The bundle $\Ker (\varDelta)$}
Define the (fiber-wise) Laplacian operator 
$$\varDelta := \big(d ^{\nabla ^\rE} _\rV + \big(d ^{\nabla ^\rE} _\rV \big)^* \big)^2 .$$
Since $\varDelta $ is fiber-wise, its kernel, $\Ker (\varDelta ) $ is a module over $C ^\infty (\rB)$.
One may also regard $\Ker (\varDelta ) $ as a fiber bundle with typical fiber $\Ker (\varDelta |_{\rZ _x}) $.
Since $\varDelta $ is $\rG$-invariant, 
$\Ker (\varDelta )$ is a contravariant vector bundle.

Denote also respectively by $\Rg ( d ^{\nabla ^{\rE}} _\rV )$ and $\Rg (( d ^{\nabla ^{\rE}} _\rV )^*)$
the image of (the adjoint extension of) $d ^{\nabla ^{\rE}} _\rV $ and $( d ^{\nabla ^{\rE}} _\rV )^*$.
Recall \cite{Lopez;FoliationHeat} that one has Hodge decomposition
$$ \cW _m (\rE) = \Ker (\varDelta ) 
\oplus \overline {\cW _m (\rE) \cap \Rg ( d ^{\nabla ^{\rE}} _\rV )}
\oplus \overline {\cW _m (\rE) \cap \Rg (( d ^{\nabla ^{\rE}} _\rV )^*)} $$
for all Sobolev spaces. 
Let $\varPi _0, \varPi _{d } , \varPi _{d ^*} $ be the projections onto the respective components.
Then $\varPi _0, \varPi _{d} , \varPi _{d ^*} $ are all smooth, bounded, fiber-wise operators.

The Bismut super-connection $D _\rB $ induce a connection on $\Ker (\varDelta)$.
Namely, it is straightforward to verify that 
$$ \varPi _0 L ^{\rE ^\bullet _\flat } \varPi _0  \text{ and } \varPi _0 (L ^{\rE ^\bullet _\flat } )' \varPi _0 $$ 
are both flat connections on $\Ker (\varDelta )$ as a $C ^\infty (\rB)$ module 
(c.f. \cite[Section 3(f)]{Bismut;AnaTorsion}).
Hence by the same arguments as above,
\begin{equation}
\label{KerConn}
\nabla ^{\Ker (\varDelta )} (r)
:= \varPi _0 \big(r L ^{\rE ^\bullet _\flat } + (1 - r) (L ^{\rE ^\bullet _\flat })' + \nabla ^\rG \big) \varPi _0 
\end{equation}
is a connection on $\Ker (\varDelta )$ as a $C ^* _c (\rB \rtimes \rG)$ module. 

We compute the curvature of $\nabla ^{\Ker (\varDelta )} (r) $.
Define
\begin{align*}
\Omega :=& \frac{1}{2} \big((L ^{\rE ^\bullet _\flat })' - L ^{\rE ^\bullet _\flat } \big) \\
L (r) :=& r L ^{\rE ^\bullet _\flat} + (1 - r) (L ^{\rE ^\bullet _\flat})' .
\end{align*}
Since $D _ \rB ^2 = (D _\rB ' )^2 = 0$, it follows that
\begin{align}
\label{Flat}
L ^{\rE ^\bullet _\flat } d ^{\nabla ^\rE} _\rV 
+ d ^{\nabla ^\rE} _\rV L ^{\rE ^\bullet _\flat } =& 0, \\ \nonumber
\big(L ^{\rE ^\bullet _\flat }\big)' (d ^{\nabla ^\rE} _\rV ) ^* +(d ^{\nabla ^\rE} _\rV ) ^* \big(L ^{\rE ^\bullet _\flat }\big)'
=& 0 ,
\end{align}
which imply 
$ \varPi _0 L ^{\rE ^\bullet _\flat } \varPi _{d} 
= \varPi _{d ^*} L ^{\rE ^\bullet _\flat } \varPi _0 
= \varPi _0 (L ^{\rE ^\bullet _\flat })' \varPi _{d ^*} 
= \varPi _{d} (L ^{\rE ^\bullet _\flat } )' \varPi _0 = 0 $.
Direct computation yields
\begin{align*} 
(\nabla ^{\Ker (\varDelta)} (r))^2 
=& (1-r)\varPi_0 (L ^{\rE ^\bullet _\flat })' \varPi_0 (L ^{\rE ^\bullet _\flat })' \varPi_0- 4 r (1 - r) \varPi _0 \Omega \varPi _0 \Omega \varPi _0 
+ \varPi _0 [ L (r) , \nabla ^\rG ] \varPi _0 \\
&+ 2 \varPi _0 (r \Omega \varPi _{d^*} - (1 - r) \Omega \varPi _{d} ) \nabla ^\rG \varPi _0 \\
&+ 2 \varPi _0 \nabla ^\rG (r \varPi _{d} \Omega - (1 - r) \varPi _{d ^*} \Omega ) \varPi _0
+ \varPi _0 \nabla ^\rG \varPi _0 \nabla ^\rG \varPi _0 .
\end{align*}

\begin{dfn}
Let 
$$ e ^{- (\nabla ^{\Ker (\varDelta)} (r))^2}
:= \varPi _0 + \sum _{i =1} \frac{1}{i !} \big(\nabla ^{\Ker (\varDelta)} (r) \big)^{2 i} .$$
The Chern-Simon form for the $\Ker (\varDelta )$ bundle is defined to be
$$
\CS ^{\Ker (\varDelta)} \big(L ^{\rE _\flat ^\bullet }, \big(L ^{\rE _\flat ^\bullet }\big)' \big)
:= - \int _0 ^1 \str _{\Psi} \big( 2\varPi _0 \Omega \varPi _0 e ^{- (\nabla ^{\Ker (\varDelta)} (r))^2} \big) d r ,$$ 
which lies in
$\Omega _{\ell ^2} ^\bullet (\rB \rtimes \rG ) _{\Ab} $ if $\dim \rZ$ is odd, 
and $\widetilde \Omega _{\ell ^2} ^\bullet (\rB \rtimes \rG ) _{\Ab}$ if $\dim \rZ$ is even.
\end{dfn}

\section{Large time limit of the heat kernel}
Denote by $N $ and $N _{\Omega} $ respectively the grading operator on
$\rE ^\bullet := \rE \otimes \wedge ^\bullet \rV'$ and the total horizontal grading on 
$\Omega ^\bullet _e (\rE ^\bullet \rtimes \rG)$.
Let $ D _t $ be the rescaled Bismut super-connection 
$$ D _t := t ^{\frac{1}{2}} t ^{- \frac{N _\Omega}{2}} D t ^{\frac{N _\Omega}{2}}
= t ^{\frac{1}{2}} d ^{\nabla ^{\rE}} _\rV + L ^{\rE ^\bullet _\flat } + \nabla ^\rG 
+ t ^{- \frac{1}{2}} \iota _\varTheta .$$
Its adjoint connection is 
$$ D ' _t = t ^{\frac{1}{2}} (d ^{\nabla ^\rE} _\rV )^* + (L ^{\rE ^\bullet _\flat})'
+ \nabla ^\rG -t ^{- \frac{1}{2}} \varTheta\wedge .$$ 
Define 
$$D _t (r) := r D _t + (1 - r) D ' _t , \quad 0 \leq r \leq 1 .$$
Also, for convenience, we will denote
$$ D (r) := r d ^{\nabla ^{\rE}} _\rV + (1 - r) ( d ^{\nabla ^{\rE}} _\rV )^* .$$
Note that $D (r) ^2 = r (1 - r) \varDelta $.

By Duhamel's expansion, we have
\begin{align}
\label{HeatDfn}
\nonumber
e ^{ - D _t (r) ^2 } := e ^{- r (1 - r) t \varDelta } & \\ 
+ \sum _{n = 1} ^\infty \int _{(s _0 , \cdots , s _k ) \in \Sigma ^n} & e ^ {- s _0  r (1 - r) t \varDelta } 
\star (D _t (r) ^2 - r (1 - r) t \varDelta) \star e ^ {- s _1 r (1 - r) t \varDelta }
\\ \nonumber
& \star \cdots 
\star (D _t (r) ^2 - r (1 - r) t \varDelta) \star e ^{- s _n r (1 - r) t \varDelta} d \Sigma ^n,
\end{align}
where $\Sigma ^n := \{(s _0 , s _1 \cdots , s _n ) \in [0, 1] ^{n + 1} : s _0 + \cdots + s _n = 1 \}$
and $ e ^{- r (1 - r) t \varDelta } $ is the usual fiber-wise heat operator.
Note that the coefficient of each 
$d g _{(k)} $ on the right hand side of \eqref{HeatDfn} is determined by a finite number of terms. 

\begin{rem}
Note that we regard the heat operator and the projection operator $\varPi _0$ as kernels,
as described in Example \ref{SmoothingExam}.
\end{rem}

\subsection{The Novikov-Shubin invariant}

\begin{dfn}
We say that $\rM \to \rB $ has positive Novikov-Shubin invariant if there exist $\gamma > 0 $ and $C_0>0$ such that
for sufficiently large $t$,
$$ \sup _{x \in \rB} \Big\{ \int _{\rZ _x} \chi (x, z) \int _{\rZ _x}
| e ^{- t \varDelta } - \varPi _0 |^2 \mu _x (y) \mu _x (z) \Big\} \leq C _0 t ^{- \gamma }.$$
\end{dfn}

\begin{rem}
\label{SquareNS}
Since $e ^{- \frac{t}{2} \varDelta } - \varPi _0 $ is non-negative, selfadjoint and
$(e ^{- \frac{t}{2} \varDelta } - \varPi _0 ) ^2 = e ^{- t \varDelta } - \varPi _0 $, one has
$$ \sup _{x \in \rB} \Big\{ \int _{\rZ _x} \chi (x, z) \int _{\rZ _x}
| e ^{- \frac{t}{2} \varDelta } - \varPi _0 |^2 \mu _x (y) \mu _x (z) \Big\}
= \| e ^{- t \varDelta } - \varPi _0 \| _\tau .$$
Hence our definition of having positive Novikov-Shubin is equivalent to that of \cite{Schick;NonCptTorsionEst}.
Our argument here is similar to the proof of \cite[Theorem 7.7]{BMZ;ToeplitzTOrsion}.
\end{rem}

In this paper, we will always assume $\rM \to \rB $ has positive Novikov-Shubin invariant.
From this assumption, it follows by integration over $\rB$ that
\begin{equation}
\|  e ^{- t \varDelta } - \varPi _0 \| _{\HS 0} < C t ^{- \gamma },
\end{equation}
as $t\to \infty$.

\subsection{A degree reduction trick}
Rearranging Equation \eqref{Flat}, one has
\begin{equation}
\label{RedEq}
L ^{\rE ^\bullet _\flat } (d ^{\nabla ^\rE} _\rV ) ^* +(d ^{\nabla ^\rE} _\rV ) ^* L ^{\rE ^\bullet _\flat }
= - 2 \Omega (d ^{\nabla ^\rE} _\rV ) ^* - 2 (d ^{\nabla ^\rE} _\rV ) ^* \Omega .
\end{equation}
Moreover, observe that $\Omega $ is a tensor 
(see \cite[Proposition 3.7]{Bismut;AnaTorsion} and \cite{Lopez;FoliationHeat} for explicit formulas for 
$L ^{\rE ^\bullet _\flat} $ and $\Omega $)
and
$(d ^{\nabla ^\rE} _\rV ) + (d ^{\nabla ^\rE} _\rV ) ^* + L ^{\rE ^\bullet _\flat }
+ (L ^{\rE ^\bullet _\flat } ) ^* $
is an elliptic operator.

As a first application of Equation \eqref{RedEq}, recall the main result of \cite[Section 3]{So;CommTorsion}:
\begin{lem}
Suppose the Novikov-Shubin invariant is positive.
The heat operator $e ^{- t \varDelta }$ is $\rG$-invariant, moreover,
$$\big\| (e ^{- t \varDelta } - \varPi _0)(x, y, z) \big\| _{\HS m} (g) = O (t ^{- \gamma }) ,$$
for all $m \in \bbN $ as $t \to \infty$.
\end{lem}

Recall that in \cite{Schick;NonCptTorsionEst}, 
the main observation is that $\frac{1}{2}(D _\rB + D' _\rB)$ is a flat connection, which implies
$$ (D _\rB + D' _\rB)^2 = - (D _\rB - D' _\rB)^2 .$$
Since the r.h.s. is a fiber-wise operator,
one can estimate the size of the rescaled heat kernel,
using known results on fiber-wise estimates.
Here $D (r)$ is {\em not} flat.
Instead we have the following important lemma,
which is another consequence of Equation \eqref{RedEq}:
\begin{lem}
\label{RedLem}
One has the identity:
\begin{align}
D_t (r) ^2
=& tD (r ) ^2 + t ^{\frac{1}{2}} ( \Omega _1 D (r) + D (r) \Omega _2 ) + \Omega _0,
\end{align}
where we denoted
\begin{align*}
\Omega _0 :=& - 4 r (1 - r ) \Omega ^2 
+ [L (r) , t ^{- \frac{1}{2}} (r \iota_\varTheta- (1 - r) \varTheta\wedge ) ] \\
&-r  ( d ^{\nabla ^\rE} _\rV \iota _\varTheta + \iota _\varTheta d ^{\nabla ^\rE} _\rV )
+ (1 - r ) ((d ^{\nabla ^\rE} _\rV ) ^* \varTheta\wedge + \varTheta\wedge (d ^{\nabla ^\rE} _\rV ) ^*) \\
&+ t ^{- 1} (r \iota_{\varTheta } - (1 - r) \varTheta \wedge)^2 
+ [L (r) , \nabla ^\rG ] + (\nabla ^\rG ) ^2 \\
\Omega _1 :=& 2 \Omega ((1 - r) \varPi _{d} - r \varPi _{d ^* }) + \nabla ^\rG 
+ t ^{- \frac{1}{2}} (r\iota_\varTheta-(1-r)\varTheta\wedge) \\
\Omega _2 :=& 2 ((1 - r ) \varPi _{d ^*} - r \varPi _{d}) \Omega + \nabla ^\rG 
+ t ^{- \frac{1}{2}}(r\iota_\varTheta-(1-r)\varTheta\wedge) . 
\end{align*}
\end{lem}
\begin{proof}
One directly computes
\begin{align*}
D _t (r) ^2
=& tD (r) ^2 + t ^{\frac{1}{2}}[D (r) , L (r) + \nabla ^\rG 
+ t ^{- \frac{1}{2}} (r \iota _{\varTheta } - (1 - r) {\varTheta}\wedge )] \\ \nonumber
&+ (L (r) + \nabla ^\rG 
+ t ^{- \frac{1}{2}} (r \iota _{\varTheta } - (1 - r) \varTheta\wedge )) ^2 .
\end{align*}
By Equation \eqref{RedEq}, one has 
$$ [D (r) , L (r) ] 
= 2 \Omega ((1 - r) \varPi _{d} - r \varPi _{d ^*}) D (r) 
+ 2 D (r) ((1 - r ) \varPi _{d ^*} - r \varPi _{d}) \Omega ,$$
and since both $D _\rB$ and $D ' _\rB$ are flat,
$$ (L (r) )^2 
= -r( d ^{\nabla ^\rE} _\rV \iota _\varTheta + \iota _\varTheta d ^{\nabla ^\rE} _\rV )
+  (1 - r ) ((d ^{\nabla ^\rE} _\rV ) ^* \varTheta\wedge + \varTheta\wedge (d ^{\nabla ^\rE} _\rV ) ^*)
- 4 r (1 - r) \Omega ^2 .$$
The lemma clearly follows by combining these equations.
\end{proof}

The key observation from Lemma \ref{RedLem} is that 
$\Omega _0 , \Omega _1 , \Omega _2 $ are all smooth fiber-wise operators with respect to the foliation $\rM \to \rB$.

\subsection{The large time estimation of Azzali-Goette-Schick}
In this section,
we follow \cite[Section 4]{Schick;NonCptTorsionEst} to estimate the Hilbert-Schmit norms of 
$$e ^{ - D _t (r) ^2 } \in \Psi^{-\infty}_{\ell^2}(E_{\flat}\rtimes G)$$
(see Lemma \ref{Main3} below).

Let $\gamma ' := 1 - (1 + \frac{2 \gamma }{n+2+2\gamma } ) ^{-1} $,
$\bar r (t) := (r (1 -r) t) ^{- \gamma '}$. Fix $\bar t $ such that $\bar r ( \bar t) < (n + 1 ) ^{-1} $.
Recall that in \cite{So;CommTorsion} the authors proved the following counterparts of \cite[Lemma 4.2]{Schick;NonCptTorsionEst}:
\begin{lem}
For $c = 0, 1, 2, \cdots$, and for all $0 \leq r \leq 1, 0 < r' < 1, r (1 -r) t > \bar t $,
$$ \big\| \big(\sqrt {t} D (r) \big) ^c e ^{- r' r (1 -r) t \varDelta } \big\| _{\op ' m} \dot \leq r ^{\prime - \frac{c}{2}} ;$$
For all $0 \leq r \leq 1, \bar r (t) < r < 1, t > \bar t $, 
\begin{align*}
\big\| e ^{- r' r (1 -r) t \varDelta } - \varPi _0 \big\| _{\HS m} \dot \leq & (r' r (1 -r) t) ^{- \gamma} , \\
\big\| \big(\sqrt {t} D (r) \big) ^c e ^{- r' r (1 -r) t \varDelta } \big\| _{\HS m}
\dot \leq & r ^{\prime - \frac{c}{2}} (r' r (1 -r) t) ^{- \gamma}, \text{ if $c \geq 1$}.
\end{align*}
\end{lem}
\begin{proof}
To prove the first equality,
write 
$$ 2 D (r) = (d _\rV + (d ^{\nabla ^\rE} _\rV )^*) - (2 r - 1)(d _\rV - (d ^{\nabla ^\rE} _\rV ) ^*).$$
Clearly $d _\rV + (d ^{\nabla ^\rE} _\rV )^*$ anti-commutes with $d _\rV - (d ^{\nabla ^\rE} _\rV ) ^*$,
and both commute with $\varDelta$.
Therefore $D (r) ^c e ^{- r' r (1 -r) t \varDelta } $ can be written as sum of the form
$$ C (r') (d _\rV + (d ^{\nabla ^\rE} _\rV )^*) ^k  e ^{- \frac{r' r (1 -r) t \varDelta }{2}}
(d _\rV - (d ^{\nabla ^\rE} _\rV ) ^*) ^{c - k} e ^{- \frac{ r' r (1 -r) t \varDelta }{2}},$$
where $k =0, \cdots , c$.
The first inequality then follows form \cite{Lopez;FoliationHeat}.

The second inequality is \cite[Theorem 3.13]{So;CommTorsion}.

To prove the third inequality one writes
$$ D (r) ^c e ^{- r' r (1 -r) t \varDelta } 
= (D (r) ^c e ^{- \frac{r' r (1 -r) t \varDelta }{2}}) e ^{- \frac{r' r (1 -r) t \varDelta }{2}}, $$
then take the $\| \cdot \| _{\op' m}$ norm for the first factor, and $\| \cdot \| _{\HS m}$ for the second.
\end{proof}

We furthermore observe that the arguments leading to the main result 
\cite[Theorem 4.1]{Schick;NonCptTorsionEst} still hold if one replaces the operator and 
$\|\cdot\|_{\tau}$ norm respectively by $\|\cdot\|_{\op' m}$ and $\|\cdot\|_{\HS m}$ for any $m$.

The arguments in \cite[Section 4]{Schick;NonCptTorsionEst} are elementary, so we will only recall some key steps. 
First, one splits the domain of integration
$\Sigma ^n = \bigcup _{I \neq \{ 0, \cdots , n \}} \Sigma ^n _{\bar r , I} $,
where
$$ \Sigma ^n _{\bar r , I} := \{ (r _0 , \cdots , r _n ) 
: r _i \leq \bar r , \Forall i \in I, r _j \geq \bar r, \Forall j \not \in I \}.$$
Then from Equation \eqref{HeatDfn} and grouping terms involving $D (r)$ together, one has
$$ e ^{- D _t (r) ^2} = \sum K (t, n, I, c _0 , \cdots c _n ; a _1 , \cdots a _n ) ,$$
where
\begin{align}
\label{KDfn}
K (t, n, I, c _0 , \cdots c _n ; & a _1 , \cdots a _n ) := \\ \nonumber
\int _{\Sigma ^n _{\bar r , I}} &
(t ^{\frac{1}{2}} D (r) ) ^{c _0} e ^ {- r _0 r (1 - r) t \varDelta } 
\star \Omega _{a _1} \star (t ^{\frac{1}{2}} D (r) ) ^{c _1} e ^ {- r _1 r (1 - r) t \varDelta } \\ \nonumber
& \star \cdots \star \Omega _{a _n}
\star (t ^{\frac{1}{2}} D (r) ) ^{c _n} e ^ {- r _n r (1 - r) t \varDelta } d \Sigma ^n ,
\end{align}
for $c _i = 0, 1, 2, a _j = 0, 1, 2 $.
We follow the proof of \cite[Proposition 4.6]{Schick;NonCptTorsionEst}
(see also \cite[Lemma 4.3]{So;CommTorsion}) to estimate $K (t , n, I, c _0 , \cdots c _n ; a _1 , \cdots , a _n )$.

\begin{rem}
Note that the integrand in \eqref{KDfn}, in particular $\nabla ^\rG$ is not $\Omega ^\bullet _c (\rB \rtimes \rG)$ linear.
However, $\nabla ^\rG$ still satisfies the condition \eqref{Hom3}.
Observe that all results in Sections 3.4 and 3.5 only uses \eqref{Hom3},
therefore they still hold for $K$,
provided we abuse notation and define $\| K \| _{\HS m}$ as in Equation \eqref{HSDfn} whenever $K$ only satisfies \eqref{Hom3} but not necessary 
$\Omega ^\bullet _c (\rB \rtimes \rG)$ linear.
\end{rem}

\begin{lem}
\label{Schick4.6}
Suppose $c_0 , \cdots c_n = 0, 1$.
There exists $\varepsilon>0$ such that as $t \to \infty $, 
\begin{align*}
K (t , n, I, c _0 , \cdots c _n &, a _1, \cdots , a _n ) (x, y, z) \\
=&
\left\{
\begin{array}{ll}
(\frac{1}{n !} \varPi _0 \Omega _{a _1} \varPi _0 \cdots \varPi _0) (x, y, z)+ O (t ^{- \varepsilon })
& \text{ if } I = \emptyset , c _0 , \cdots , c _n = 0 \\
O (t ^{- \varepsilon }) & \text{ otherwise}
\end{array}
\right.
\end{align*}
in the $\| \cdot \| _{\HS m}$-norm.
\end{lem}
\begin{proof}
We generalize the proof of \cite[Lemma 4.2]{So;CommTorsion}.

First suppose $I = \emptyset , c _q \geq 1$ for some $q$.
We take the $\| \cdot \|_{\HS m} $ norm of the $(t ^{\frac{1}{2}} D (r) )^{c _q} e ^ {- r _q r (1 - r) t \varDelta }$ term.
Since $\Omega _{a _i}$ are $C ^\infty $ bounded tensors with bounds independent of $t$
by Theorem \ref{Main4} and Lemma \ref{Main5}, $\| \cdot \| _{\HS m} $ of the integrand in \eqref{KDfn} is bounded,
for some constants $C _{a _i}$ independent of $t$, by
\begin{align*}
\big\| (t ^{\frac{1}{2}} D (r) ) ^{c _0} e ^ {- r _0 r (1 - r) t \varDelta } & \big\| _{\op' m} 
C _{a _1} \cdots C _{a _q} \\
\big\| (t ^{\frac{1}{2}} D (r) ) ^{c _q} e ^ {- r _q r (1 - r) t \varDelta } \big\| _{\HS m} &
C _{a _{q + 1}} \cdots \big\| (t ^{\frac{1}{2}} D (r) ) ^{c _n} e ^ {- r _n r (1 - r) t \varDelta } \big\| _{\op' m} \\
\dot \leq & r _0 ^{- \frac{c _0}{2}} 
\cdots r _q ^{- \frac{c _q}{2}} (r _q r (1 - r) t ) ^{- \gamma }
\cdots r _n ^{- \frac{c _n}{2}} \\
\dot \leq & \bar r ^{- \frac{n}{2} - \gamma } t ^{- \gamma }.
\end{align*}
Integrating, we have the estimate
\begin{align*}
\Big\| K (t , n, I, c _0 , \cdots c _n ; a _1 , \cdots a _n ) \Big\| _{\HS m}
\leq & C' _m t ^{- \gamma + \gamma' (\frac{n}{2} + \gamma) } \int d \Sigma ^n ,
\end{align*}
which is $O ((r (1 - r)t) ^{- \varepsilon})$ with $\varepsilon = \gamma (1 - \frac{n + 2 \gamma}{n + 2 + 2 \gamma})$.

Next, suppose $I=\emptyset$ and $c_i=0$ for all $i$. 
Write $e^{-r_i t r (1 - r) \varDelta} = (e^{-r_i t r (1 - r) \varDelta} - \varPi _0 ) + \varPi_0 $, 
and split the integrand in \eqref{KDfn} into $2^{n+1}$ terms. 
If any term contains a $e^{-r_i t\Delta}-\Pi_0$ factor, 
similar arguments as in the first case shows that it is $O((r (r (1 - r)t)^{-\gamma})$. 
Hence the only term that dose not converge to $0$ is 
$$(\varPi_0 \Omega _{a_1} \varPi_0 \cdots \varPi_0 )(x, y, z).$$
Since the volume of $\Sigma^{n}_{\bar{r}(t),I}$ converges to $\frac{1}{n!}$ as $t\to \infty$, the claim follows. 

It remains to consider the case when $I$ is non-empty. 
For $t$ sufficiently large $I\neq \{0,\cdots,n\}$. 
Write $I=\{i_1,\cdots,i_s\}$, $\{0,\cdots,n\}\setminus I=:\{k_1,\cdots,k_{s'}\}\neq \emptyset$.
If $k_1,\cdots,k_{s'} = 0$, take $\|\cdot\|_{\HS m}$-norm for $(t^{1\over 2} D_0) ^{c _{k _1}} e^{-r_ {k _1} r (1 -r) t \Delta}$ term.
Then 
\begin{align*}
\big\| K(t,n, I, c_0,\cdots,c_n;a_1,\cdots,a_n) (x, y, z) \big\|_{\HS m} & \\
\dot \leq \int_{0}^{\bar{r}(t)}\cdots \int_{0}^{\bar{r}(t)}
\Big(\int_{\{(r_{k_1},\cdots,r_{k_s'}):(r_0,\cdots,r_n)\in\Sigma^{n}_{\bar{r}(t),I}\}} & 
r_0^{-{{c_{i _1}}\over 2}} \cdots r_n^{-{{c_ {i _s}}\over 2}} \\
& d(r_{k_1}\cdots r_{k_{s'}})\Big) dr_{i_1}\cdots d r_{i_s}.
\end{align*}
Since $\int_{0}^{\bar{r}(t)}r_i^{{c_i}\over 2} d r_i  =O( (r (1 - r)t)^{-\gamma'(1-{{c_i}\over 2})})$;
while the integral over the variables $r _{k_1}, \cdots, r_{k_{s'}}$ is bounded.

If there is some $c _{k _q} \geq 1$, 
we take the $\| \cdot \|_{\HS m} $ norm of the $(t ^{\frac{1}{2}} D (r) )^{c _{k_q}} e ^ {- r _{k _q} r (1 - r) t \varDelta }$ term,
and the claim follows by similar arguments as the first case.
\end{proof}

One then turns to the case when some $c _i = 2$.
If $I$ and $J$ are disjoint subsets of $\{ 0, \cdots , n \}$ with $I = \{ i _1, \cdots , i _s \}$,
and $\{ 0, \cdots , n \} \setminus (I \cup J) =: \{ k _0, \cdots , k _q \} \neq \emptyset $,
denote by
$$ \Sigma ^n _{\bar r , I, J}
:= \{ (r _0 , \cdots , r _n) \in \Sigma ^n _{\bar r , I} : r _j = \bar r (t), \text{ whenever } j \in J \} ,$$
and define for any smooth, bounded $\Omega _c ^* (\rM \rtimes \rG )$-linear operators $B _1 , \cdots B _n$
\begin{align*}
K (t, n, I, J, c _0 , \cdots c _n &; B _1 , \cdots B _n ) \\
:= \int _0 ^{\bar r (t)} \cdots \int _0 ^{\bar r (t)} &
\int _{ \{ (r _{k _0}, \cdots r _{k _q}) : (r _0 , \cdots , r _n) \in \Sigma ^n _{\bar r , I} \}} 
(t ^{\frac{1}{2}} D (r) ) ^{c _0} e ^ {- r _0 r (1 - r) t \varDelta } \\
& \prod _{i=1} ^n ( B _i (t ^{\frac{1}{2}} D (r) )^{c _i} e ^ {- r _i r (1 - r) t \varDelta })
\Big| _{\Sigma ^n _{\bar r , I, J}}
d ^q (r _{k _0}, \cdots r _{k _q}) d r _{i _1} \cdots d r _{i _s} .
\end{align*}
Suppose for some $i _p \in I, c _{i _p} = 2$,
then one has the integration by parts formula \cite[Equation (4.17)]{Schick;NonCptTorsionEst}:
\begin{align}
\label{IntPart}
\nonumber
K(t, &n, I , J; \cdots, c _{i _p}, \cdots , c _{k _0}, \cdots ; \cdots , B _{i_p}, B _{i_p +1}, \cdots ) \\
=& \left\{
\begin{array}{ll}
K(t, n, I \setminus \{i _p \}, J \cup \{ i _p \} 
; \cdots, 0, \cdots , c _{k _0}, \cdots ; \cdots , B _{i_p} , B _{i_p +1}, \cdots ) \\
- K(t, n - 1, I \setminus \{i _p \}, J; \cdots, \cdots, c _{k _0}, \cdots ; \cdots , B _{i _p} B _{i _p +1}, \cdots ) \\
+ K(t, n, I , J \cup \{ k _0 \} ; \cdots , 0, \cdots , c _{k _0}, \cdots ;
\cdots , B _{i _p} , B _{i _p +1}, \cdots ) \\
+ K(t, n, I, J; \cdots , 0, \cdots, c _{k _0} + 2, \cdots ;
\cdots, B _{i _p} , B _{i _p +1}, \cdots ) \text{ if $q > 0$,}\\
K(t, n, I \setminus \{i _p \}, J \cup \{ i _p \} ; \cdots, 0, \cdots , c _{k _0}, \cdots ; 
\cdots , B _{i_p} , B _{i_p + 1}, \cdots ) \\
- K(t, n - 1, I \setminus \{i _p \}, J; \cdots, \cdots, c _{k _0}, \cdots ; \cdots , B _{i_p} B _{i_p +1}, \cdots ) \\
+ K(t, n, I, J; \cdots , 0, \cdots, c _{k _0} + 2, \cdots ;
\cdots, B _{i _p} , B _{i _p +1}, \cdots ) \text{ if $ q = 0$}.
\end{array}
\right.
\end{align}
We remark that the proof of \cite[Equation (4.17)]{Schick;NonCptTorsionEst} does not involve any norm,
therefore we omit the details here.

On the other hand one has the following straightforward generalization of Lemma \ref{Schick4.6}
(compare \cite[Proposition 4.7]{Schick;NonCptTorsionEst}):
\begin{lem}
\label{Schick4.7}
Suppose $c _i = 0, 1$ for all $ i \in I$.
There exists $\varepsilon > 0 $ such that as $t\to \infty$
\begin{align*}
\big\| K (t , n, I, J, c _0 , \cdots c _n ; a _1 , \cdots , a _n ) 
- ((n - |&J|)!) ^{-1} \varPi _0 \Omega _{a _1} \varPi _0 \cdots \Omega _{a _n} \varPi _0 \big\| _{\HS m} \\
=& O ( (r (1 - r) t) ^{- \varepsilon }) 
\text{ if $I = \emptyset , c _0 , \cdots , c _n = 0 $}; \\
\big\| K (t , n, I, J, c _0 , \cdots c _n ; a _1 , \cdots , a _n ) \big\| _{\HS m} 
=& O ( (r (1 - r) t) ^{- \varepsilon }) \text{ otherwise.}
\end{align*}
\end{lem}

Thus the term $ K (t, n, I, c _0 , \cdots c _n ; a _1 , \cdots a _n ) $ converges to $0$ unless
$$ c _i = 0 \text{ whenever } i \in I , \quad c _i = 2 \text { whenever } i \not \in I. $$
%Adding back all terms with $I $ and $c _i$ fixed, it follows that
%$$ K (t, n, I, c _0 , \cdots c _n ; a _1 , \cdots a _n ) 
%= \sum _{g _{(k)}} d g _{(k)} K (t, n, I, c _0 , \cdots c _n ; a _1 , \cdots a _n ) ^{g _{(k)}} $$
%converges to $0$ (point-wisely) unless
%$ c _i = 0 $ whenever $i \in I $, $c _i = 2 $ whenever $ i \not \in I $.
Whenever $c _i = 2 $ and $i \in I$, the corresponding part of the integrand in such a term is of the form
\begin{equation}
\label{Word1}
\cdots e ^{- r _{i-1} r (1 - r) t \varDelta } \star \Omega _1 
\star t D (r) ^2 e ^{- r _{i} r (1 - r) t \varDelta } \star \Omega _2 
\star e ^{- r _{i+1} r (1 - r) t \varDelta } \cdots ;
\end{equation}
on the other hand if $i -1, i \not \in I$, 
then the corresponding part of the integrand is of the form
\begin{equation}
\label{Word2}
\cdots e ^{- r _{i-1} r (1 - r) t \varDelta } \star \Omega _0 
\star e ^{- r _{i} r (1 - r) t \varDelta } \cdots .
\end{equation}
By Equation \eqref{IntPart} and Lemma \ref{Schick4.7}, for each fixed $0 < r < 1$,
\begin{align*}
\cdots e ^{- r _{i-1} r (1 - r) t \varDelta } & \star \Omega _1
\star t D (r) ^2 e ^{- r _{i} r (1 - r) t \varDelta } \star \Omega _2
\star e ^{- r _{i+1} r (1 - r) t \varDelta } \cdots \\
=& \cdots \varPi _0 \star (2 \Omega ((1 - r) \varPi _{d} - r \varPi _{d ^*}) + \nabla ^\rG ) \\
& \star (\varPi _0 - \id) \star (2 ((1 - r ) \varPi _{d ^*} - r \varPi _{d}) \Omega + \nabla ^\rG ) 
\star \varPi _0 \cdots \\
=& \cdots \varPi _0 \big( -(\nabla ^{\Ker (\varDelta )} )^2 - \Omega _ 0 \big) \varPi _0 \cdots
\end{align*}
modulo terms of $O ((r (1 - r) t ) ^{- \varepsilon})$.

One then proceeds as \cite[Section 4.5]{Schick;NonCptTorsionEst} to compute the limit of $e ^{- D _t (r) ^2}$ as $t \to \infty$.
Since $K (t , n, I, c _0 , \cdots c _n ; a _1 , \cdots a _n )$ is of non-commutative degree at least $n - \dim \rB$,
therefore given any degree, $e ^{- D _t (r)^2 }$ is determined by a finite number of terms.
Moreover, we have seen $ K (t , n, I, c _0 , \cdots c _n ; a _1 , \cdots a _n )$ converge to its limit with an error of 
$O ((r (1 -r) t)^{- \varepsilon _n})$ (note that the rate of convergence depends on $n$).

To simplify notation, we denote
\begin{nota}
Given a sequence of positive numbers $\{ \gamma _n \}$,
and a family of kernels $\psi (t) \in \Psi ^{- \infty } _{\ell ^2 } (\rM \times _\rB \rM), t \in (0, \infty )$,
we write
$$ \psi (t) = \dot O (t ^{- \{\gamma _n \}}) $$
if the degree $n$ component of $\psi $ is $O (t ^{- \gamma _n})$ in the $\| \cdot \| _{\HS m} $ norm for all $m$.
\end{nota}

Summing over all $ K (t , n, I, c _0 , \cdots c _n ; a _1 , \cdots a _n )$, one gets:
\begin{lem}
\label{Main3}
For all $0 < r < 1$, as $t\to \infty$,
$$ \big\| e ^{- D _t (r) ^2} - e ^{ -(\nabla ^{\Ker (\varDelta )}(r) )^2} \big\| _{\HS m} 
= \dot O ((r (1 - r) t ) ^{- \{\varepsilon _n \}}),$$
for some sequence $\{ \varepsilon _n \}$. 
\end{lem}

Next, we turn to study the large time limit of 
$$ (D_t - D ' _t) e ^{- D _t (r) ^2} .$$
From Equation \eqref{HeatDfn} one has
$$ e ^{- D _t (r) ^2} = \sum K' (t, n, I, c _0 , \cdots c _n ; a _1 , \cdots a _n ) ,$$
where
\begin{align*}
K' (t, n, I, c _0 , \cdots c _n ; & a _1 , \cdots a _n ) := \\
\int _{\Sigma ^n _{\bar r , I}} &
(D _t - D' _t) (t ^{\frac{1}{2}} D (r) ) ^{c _0} e ^ {- r _0 r (1 - r) t \varDelta } 
\star \Omega _{a _1} \star (t ^{\frac{1}{2}} D (r) ) ^{c _1} e ^ {- r _1 r (1 - r) t \varDelta } \\
& \star \cdots \star \Omega _{a _n}
\star (t ^{\frac{1}{2}} D (r) ) ^{c _n} e ^ {- r _n r (1 - r) t \varDelta } d \Sigma ^n ,
\end{align*}
for $c _i = 0, 1, 2, a _j = 0, 1, 2 $.
For $0 < r < 1$, write
$$D _t - D' _t 
= t ^{ \frac{1}{2}} (r ^{-1} \varPi _d - (1 - r) ^{-1} \varPi _{d ^*} ) D (r)
+ 2 \Omega + t ^{- \frac{1}{2}} (\iota _\Theta+\Theta \wedge) .$$
It is clear that $K'$ is essentially of the same form as $K$,
therefore the same arguments as above apply.
We conclude that
$K' (t, n, I, c _0 , \cdots c _n ; a _1 , \cdots a _n ) $ is 
$O ((r (1 - r)t) ^{- \varepsilon})$ unless $K' (t, n, I, c _0 , \cdots c _n ; a _1 , \cdots a _n )$ equals
\begin{align*}
\int _{\Sigma ^n _{\bar r , I}} &
(2 \Omega + t ^{- \frac{1}{2}} ( \iota _\Theta+\Theta \wedge))
 e ^ {- r _0 r (1 - r) t \varDelta } \Omega _0 e ^ {- r _1 r (1 - r) t \varDelta } \cdots d \Sigma ^n 
\text{ or} \\
\int _{\Sigma ^n _{\bar r , I}} &
(r ^{-1} \varPi _d - (1 - r) ^{-1} \varPi _{d ^*} ) (t D (r) ^2)
e ^ {- r _0 r (1 - r) t \varDelta } \Omega _2 e ^ {- r _1 r (1 - r) t \varDelta } \cdots d \Sigma ^n ,
\end{align*} 
where for $i \geq 1$, $ c _i = 0 $ whenever $i \in I$, $c _i = 2 $ whenever $ i \not \in I $.
One has 
\begin{align*}
(r ^{-1} \varPi _d - (1 - r) ^{-1} \varPi _{d ^*} ) (t D (r) ^2)
& e ^ {- r _0 r (1 - r) t \varDelta } \Omega _2 e ^ {- r _1 r (1 - r) t \varDelta } \\
=& 2 (\varPi _0 - \id) \Omega + (r ^{-1} \varPi _d - (1 - r) ^{-1} \varPi _{d ^*} ) \star \nabla ^\rG 
\end{align*}
modulo terms of $ \dot O ((r (1 - r)t) ^{- \{\varepsilon _n \}})$. It follows that
\begin{lem}
\label{Main6}
For all $0 < r < 1$, as $t\to\infty$,
\begin{align*}
\Big\| (D _t - D' _t) e ^{- D _t (r) ^2} 
- (2 \varPi _0 \Omega + (r ^{-1} \varPi _d - (1 - r) ^{-1} \varPi _{d ^*} ) & \star \nabla ^\rG)
\star e ^{ -(\nabla ^{\Ker (\varDelta )}(r) )^2} \Big\| _{\HS m} \\
=& \dot O ((r (1 - r) t ) ^{- \{\varepsilon _n \} }).
\end{align*}
%For any $0 < r' < 1$,
%\begin{align*}
%\Big\| e ^{- D _{r' t} (r)^2} (D _t - D' _t) e ^{- D _{(1 - r') t} (r) ^2} 
%- e ^{ -(\nabla ^{\Ker (\varDelta )}(r) )^2} \star (2 \Omega ) & \star e ^{ -(\nabla ^{\Ker (\varDelta )}(r) )^2} \Big\| _{\HS m}\\
%=& \dot O ((r' (1 - r') r (1 -r) t)^{- \{\varepsilon _n \}}).
%\end{align*}
\end{lem}

The case for $ (D _t - D' _t) e ^{- D _t (r) ^2} (D _t - D' _t)$ is similar.
We simply state the result:
\begin{lem}
\label{Decay3}
For all $0 < r < 1$, as $t\to\infty$,
\begin{align*}
\Big\| ( & D _t - D' _t) e ^{- D _t (r) ^2} (D _t - D' _t)\\
-& \big(2 \varPi _0 \Omega + \big(\frac {\varPi _d}{r} - \frac{\varPi _{d ^*}}{1 - r} \big) \star \nabla ^\rG \big)
\star e ^{ -(\nabla ^{\Ker (\varDelta )} (r))^2} 
\star \big(2 \Omega \varPi _0 + \nabla ^\rG \star \big(\frac{\varPi _{d ^*}}{r} - \frac{\varPi _d}{1 - r} \big) \big) 
\Big\| _{\HS m} \\
&= \dot O ((r (1 - r) t ) ^{- \{\varepsilon _n\} }).
\end{align*}
\end{lem}

\subsection{Large time behavior of the super-trace}
By Lemma \ref{TrClassCri}, 
$e ^{- D _t (r) ^2} $, $(D _t - D' _t) e ^{- D _t (r) ^2} $ 
and their limits as $t \to \infty $ are trace class operators.
We compute their (super)-trace as $t \to \infty $
(we do not need the super-trace of $(D _t - D' _t) e ^{- D_t (r) ^2} (D _t - D' _t)$).

\begin{thm}
\label{TrDecay0}
As $t \to \infty$,
\begin{align*}
\big\| \str _{\Psi} (e^{- D _{t} (r) ^2} - e ^{ -(\nabla ^{\Ker (\varDelta )} (r))^2}) \big\| _{C ^m}
=& \dot O ((r (1 - r) t ) ^{- \{\varepsilon _n \}}), \\
\big\| \str _{\Psi} \big( (D_t - D' _t) e^{- D _{t} (r) ^2} 
- 2 \varPi _0 \Omega e^{ -(\nabla ^{\Ker (\varDelta )} (r))^2} \big) \big\| _{C ^m}
=& \dot O ((r (1 - r) t ) ^{- \{\varepsilon _n \}}).
\end{align*}
\end{thm}
\begin{proof}
We begin with $\str _\Psi (e ^{- D_t (r) ^2})$.
Write 
\begin{align*}
e ^{- D_{t} (r) ^2} =& 2 ^{- \frac{N _\Omega }{ 2}} e ^{- D _{t / 2} (r) ^2} 
e ^{- D _{t / 2} (r) ^2} 2 ^{\frac{N _ \Omega }{ 2}}.
\end{align*}
Then 
\begin{align*}
e^{- D_{t} (r) ^2} - e ^{ - (\nabla ^{\Ker (\varDelta )} (r))^2}
=& 2 ^{- \frac{N _\Omega }{ 2}} \big(
e ^{- D _{t / 2} (r) ^2} ( e ^{- D _{t / 2} (r) ^2} - e ^{ -(\nabla ^{\Ker (\varDelta )} (r))^2}) \\
&+ ( e ^{- D _{t / 2} (r) ^2} - e ^{ -(\nabla ^{\Ker (\varDelta )} (r))^2}) e ^{ -(\nabla ^{\Ker (\varDelta )} (r))^2}
\big) 2 ^{ \frac{N _ \Omega }{ 2}}.
\end{align*}
Denote by $P _k$ the projection to (total) degree $k$ component, $k =0, 1, 2, \cdots$.
By the same arguments as in the proof of Lemma \ref{TrClassCri} (in particular Equation \eqref{TrEst3}),
one estimates the $C ^m$ norms (for $\Omega ^\bullet _{\ell ^2 , m} (\rB \rtimes \rG)$):
\begin{align*}
\big\| P _k \big( & \str _{\Psi} (e^{- D _t (r) ^2} - e ^{ - (\nabla ^{\Ker (\varDelta )}(r) )^2}) ) \big\| _{C ^m} \\
=& \Big\| 2 ^{- \frac{N _\Omega}{2}} \str _{\Psi} \Big(
\sum _{k' = 0} ^k (P _{k'} e ^{- D _{t / 2} (r) ^2}) \big(P _{k - k'} ( e ^{- D _{t /2} (r) ^2} - e ^{ -(\nabla ^{\Ker (\varDelta )} (r))^2}) \big) \\
&+ \sum _{k' = 0} ^k \big(P _{k'}( e ^{- D_{t / 2} (r) ^2} - e ^{ -(\nabla ^{\Ker (\varDelta )} (r))^2}) \big)
P _{k -k'} e ^{ -(\nabla ^{\Ker (\varDelta )} (r))^2}
\Big) \Big\| _{C ^m} \\
\dot \leq &  
\sum _{k' = 0} ^k \big\| P _{k'} e ^{- D _{t / 2} (r) ^2} \big\| _{\HS m'}
\big\| P _{k - k'} (e ^{- D _{t / 2} (r) ^2} - e ^{ -(\nabla ^{\Ker (\varDelta )}(r) )^2}) \big\| _{\HS m'} \\
&+ \sum _{k' = 0} ^k \big\| P _{k'} (e ^{- D _{t / 2} (r) ^2} - e ^{ -(\nabla ^{\Ker (\varDelta )}(r) )^2}) \big\| _{\HS m'}
\big\| P _{k - k'} e ^{ -(\nabla ^{\Ker (\varDelta )}(r) )^2} \| _{\HS m'},
\end{align*}
for some $m'$.
By Lemma \ref{Main3}, 
$ \big\| P _{k'} (e ^{- D_{t / 2} (r) ^2} - e ^{ -(\nabla ^{\Ker (\varDelta )} (r))^2}) \big\| _{\HS m'} \\
= O ((r (1 - r) t ) ^{- \varepsilon _{k'}})$
for some $\varepsilon _{k'} > 0$.
The first estimate follows.

As for the second estimate, we have
\begin{align*}
(D _t - D' _t) e ^{- D _t (r) ^2} =& 
2 ^{\frac{1}{2} - \frac {N _\Omega}{2} } \big((D _{t / 2} - D' _{t / 2}) e ^{- D_{t / 2} (r) ^2}  e ^{- D_{t / 2} (r) ^2} 2 ^{\frac{N _\Omega }{2}} \\
2 \varPi _0 \Omega e ^{ -(\nabla ^{\Ker (\varDelta )}(r) )^2}
=& 2 ^{\frac{1}{2} - \frac{N _\Omega}{2}} 2 \varPi _0 \Omega e ^{ -(\nabla ^{\Ker (\varDelta )}(r) )^2} e ^{ -(\nabla ^{\Ker (\varDelta )}(r) )^2} 
2 ^{\frac{N _\Omega }{2}}.
%( e ^{- D _{t / 2} (r) ^2} - e ^{ -(\nabla ^{\Ker (\varDelta )}(r) )^2}) 
%- \varPi _0 \Omega e ^{ -(\nabla ^{\Ker (\varDelta )} (r))^2} \\
%=& 2 ^{\frac{1}{2} - \frac {N _\Omega}{2} } \big(
%(D _{t / 2} - D' _{t / 2}) e ^{- D_{t / 2} (r) ^2} 
%( e ^{- D _{t / 2} (r) ^2} - e ^{ -(\nabla ^{\Ker (\varDelta )}(r) )^2}) 
%+ ((D _{t / 2} - D' _{t / 2}) e ^{- D _{t / 2} (r) ^2} \\
%&- (2 \varPi _0 \Omega + (r ^{-1} \varPi _d - (1 - r) ^{-1} \varPi _{d ^*} ) \star \nabla ^\rG)
%e ^{ -(\nabla ^{\Ker (\varDelta )} (r))^2}) e ^{ -(\nabla ^{\Ker (\varDelta )} (r))^2}
%\big) 2 ^{\frac{N _ \Omega}{2}} \\
%&+ O (t ^{- \frac{1}{2}}),
\end{align*}
Therefore in $\Omega _{\ell ^2} ^\bullet (\rB \rtimes \rG)$
\begin{align*}
\str _\Psi & \big( (D _t - D' _t) e ^{- D _t (r) ^2} - 2 \Omega e ^{ -(\nabla ^{\Ker (\varDelta )}(r) )^2} \big) \\
=& 2 ^{\frac{1}{2} - \frac{N _\Omega }{2}} \str _\Psi \Big(
e ^{- D_{t / 2} (r) ^2} (D _{t / 2} - D' _{t / 2}) e ^{- D_{t / 2} (r) ^2} \\
&- 2 e ^{ -(\nabla ^{\Ker (\varDelta )}(r) )^2} \varPi _0 \Omega e ^{ -(\nabla ^{\Ker (\varDelta )}(r) )^2} \Big) \\
=& 2 ^{\frac{1}{2} - \frac{N _\Omega }{2}} \str _\Psi \Big(
\big( e ^{- D_{t / 2} (r) ^2} - e ^{ -(\nabla ^{\Ker (\varDelta )}(r) )^2} \big) (D _{t / 2} - D' _{t / 2}) e ^{- D_{t / 2} (r) ^2} \\
&+ e ^{ -(\nabla ^{\Ker (\varDelta )}(r) )^2} 
\big( (D _{t / 2} - D' _{t / 2}) e ^{- D_{t / 2} (r) ^2} - 2 \varPi _0 \Omega e ^{ -(\nabla ^{\Ker (\varDelta )}(r) )^2} \big) \Big).
\end{align*}
Because $ e ^{ -(\nabla ^{\Ker (\varDelta )}(r) )^2} \varPi _d = e ^{ -(\nabla ^{\Ker (\varDelta )}(r) )^2} \varPi _{d ^*} = 0 $,
\begin{align*}
e ^{ -(\nabla ^{\Ker (\varDelta )}(r) )^2} &
\big( (D _{t / 2} - D' _{t / 2}) e ^{- D_{t / 2} (r) ^2} - 2 \varPi _0 \Omega e ^{ -(\nabla ^{\Ker (\varDelta )}(r) )^2} \big) \\
=& e ^{ -(\nabla ^{\Ker (\varDelta )}(r) )^2} 
\big( (D _{t / 2} - D' _{t / 2}) e ^{- D_{t / 2} (r) ^2} \\
&- (2 \varPi _0 \Omega + (r ^{-1} \varPi _d - (1 - r) ^{-1} \varPi _{d ^*}) \star \nabla ^\rG) \star e ^{ -(\nabla ^{\Ker (\varDelta )}(r) )^2} \big), 
\end{align*}
and the claim follows by the same arguments above and applying Lemma \ref{Main4}.
\end{proof}

\section{The non-commutative torsion form and characteristic classes}
We follow \cite{Lott;NonCommTorsion} and \cite{BGV;Book} to study the $r \to 0, r \to 1$ and 
$t \to 0$ behavior of the heat kernel.
We first need a more explicit description of the curvature of the Bismut super-connection.

\begin{nota}
Let $\tau ^{\alpha}$ be a local basis of $\pi ^* (T^* \rB)$ and let $\Lambda _\alpha$ denote exterior multiplication by $\tau ^\alpha$. 
Let $\{e_j\}_{j=1}^{\dim \rZ}$ be a local orthonormal basis of $\rV$, with dual basis 
$\{\tau^j\}_{j=1}^{\dim \rZ}$. Let $\Lambda _j$ denote exterior multiplication by $\tau^j$ and let $\iota _j$ denote interior multiplication by $e_j$. 
Put $$c^j := \Lambda _j - \iota _j, \hat{c} ^j := \Lambda _j + \iota _j.$$
\end{nota}
Set
$$\psi :=(\nabla^{\rE} )'- \nabla ^E,
2\Omega := (L ^{\rE ^\bullet _\flat })'- L ^{\rE ^\bullet _\flat }\ {\rm and}\ \nabla^{E,u}=\nabla^E+{\psi\over 2}.$$
We will use the Einstein summation convention freely. Denote the Chirstoffel symbols by
$$\omega_{IJK}=\tau^{I}\left(\nabla^{TM}_{e_K}e_J\right),$$
%As there is a vertical metric, we may raise and lower vertical indices freely.
and the twisting curvature by %$\mathcal{R} $ by
$$\mathcal{R} :={1\over 4} (g ^\rV ( e_j, R^{\rM / \rB} e_k )) \hat{c}^j \hat{c}^k \otimes I_{\rE} 
-{1\over 4} ( I_{\wedge ^\bullet \rV'} \otimes \psi^2 ) \in \Omega^2 (\rM, \Hom ( \wedge ^\bullet \rV' \otimes E)).$$
Let $\nabla^{TZ\otimes E,u}$ be the tensor of $\nabla^{\rM / \rB}$ and $\nabla^{E,u}$, 
and $R\in C^{\infty}(\rM)$ be the scalar curvature of the fibers. 
For $t>0$, put
$$\mathcal{D}_{j} :=\nabla^{TZ\otimes E,u}_{e_j}-{1\over{2\sqrt{t}}}\omega_{\alpha jk}E^{\alpha}c^k-{1\over {4t}}\omega_{\alpha\beta j}E^{\alpha}E^{\beta},$$
$$\mathcal{D}^{2} :=\mathcal{D}_{j}\mathcal{D}_{j}-\mathcal{D}_{\nabla^{\rM / \rB}_{e_j}e_j}.$$

Recall that $(D _\rB ) ^2 = (D _\rB ') ^2 = 0$, 
hence 
$ (r D _\rB + (1 - r) D_\rB ')^2 
= 4 r(1 - r) (\frac{1}{2} D _\rB + \frac{1}{2} D _\rB ') ^2$.
Since $(L ^{\rE ^\bullet _\flat })' - L ^{\rE ^\bullet _\flat }$ is a $\rG$-invariant tensor, 
which in particular anti-commutes with $\nabla ^\rG$,
we have by direct computation the Lichnerowicz formula (cf. \cite[(6.29)]{Lott;NonCommTorsion}),
\begin{multline}
\label{Lich}
(D_{t} (r) )^2 s
=4r(1-r) \Big(
{t\over 4}\big(-\mathcal{D}^2+{R\over 4}\big)
+{t\over 8}c^i c^j \mathcal{R}(e_i,e_j) + {\sqrt{t}\over 2} c^i \Lambda _{\alpha} \mathcal{R} (e_i, e_\alpha) \\
+{1\over 2} \Lambda _{\alpha} \Lambda _{\beta} \mathcal{R}(e_{\alpha},e_{\beta})
+{t\over 4} \Big({1\over 4}\psi^2_j + {1\over 8}\hat{c}^j \hat{c}^k [\psi_j,\psi_k]-{1\over 2} c^j \hat{c}^k 
(\nabla^{TZ\otimes E,u}_{e_j}\psi_k ) \Big)\\
- {\sqrt{t}\over 4} \Lambda _{\alpha} \hat{c}^{j} (\nabla^{TZ\otimes E,u}_{e_\alpha} \psi_j ) \Big) s
-{\sqrt{t}\over 2} \sum_{g\in G} d g (c (d _\rV \chi )) (g^{-1})^*s\\
+2 \big({1 \over 2} - r \big) {\sqrt{t}\over 2} \sum_{g\in G} d g (\hat{c}(d _\rV \chi )) (g^{-1})^* s
-\sum_{g\in G} d g \Lambda _{d _\rH \chi }  (g^{-1} )^*s + (\nabla^G )^2 s,
\end{multline}
where $d _\rV$ and $d _\rH$ respectively denote the vertical and horizontal DeRham differential operators.

Define the non-commutative degree operator $N _\rG := k$ on $ \Omega ^{k , l} _{\ell ^2} (\rB \rtimes \rG )$.
We consider the rescaled operator
$$r ^{N _\rG } \left(r D _\rB + (1 - r) D _\rB ' + \nabla ^\rG \right) ^2 r ^{- N _\rG } 
= r \tilde \varDelta ,$$
where 
$$\tilde \varDelta := (1 - r) (D_\rB + D _\rB ') ^2 + \nabla ^\rG (r D_\rB + (1 - r) D _\rB ')
+ (r D _\rB + (1 - r) D_\rB ') \nabla ^\rG + r (\nabla ^\rG ) ^2 .$$
Its heat kernel is just 
$$ r ^{N _\rG} \big( e ^{- t(r D _\rB + (1 - r) D _\rB ' + \nabla ^\rG )^2} (x, y, z) \big)$$
(corresponding to the operator $r ^{N _\rG} e ^{- t(r D_\rB + (1 - r) D _\rB ' + \nabla ^\rG )^2} r ^{- N _\rG}$),
which is the unique solution of 
\begin{equation}
\label{RescaledHeat}
\big(\frac{d}{d t} + r \tilde \varDelta _y \big)
\big(r ^{N _\rG} \big(e ^{- t (r D _\rB + (1 - r) D_\rB ' + \nabla ^\rG)^2} \big)( x, y, z) \big) = 0.
\end{equation}
Let $\tilde t := r t$, 
then Equation \eqref{RescaledHeat} is equivalent to
\begin{equation}
\label{RescaledHeat2}
\big(\frac{d}{d \tilde t} + \tilde \varDelta _y \big) 
\big(r ^{N _\rG} (e ^{- \tilde t \tilde \varDelta} (x, y, z))\big) = 0 .
\end{equation}
One can solve \eqref{RescaledHeat2} using the Levi parameterix method as in \cite[Chapter 2]{BGV;Book}.
It follows in particular that one has asymptotic expansion as $\tilde t = r t \to 0 $:
\begin{equation}
\label{Expansion1}
r ^{N _\rG} ( e ^{- \tilde t \tilde \varDelta} (x, y, z) )
\sim (4 \pi \tilde t) ^{- \frac{\dim \rZ}{2}} e ^{- \frac{ \bd (y, z) ^2}{4 \tilde t}}
\sum _{i = 0} \tilde t ^i \tilde \Phi _i (x, y, z),
\end{equation}
where $\tilde \Phi _i$ can be computed explicitly as in \cite[Theorem 2.26]{BGV;Book}.
Namely, in normal coordinates around arbitrary $z \in \rZ _x$, $y = \exp _z \by$,
\begin{align}
\label{expansion2}
\tilde \Phi _0 (x, y, z) :=& I \\
\nonumber
\tilde \Phi _i (x, y, z) :=&
\tau \Big( - \int _0 ^1 s ^{i - 1} \tau (x, \exp _z s \by , z) 
(J ^{\frac{1}{2}} \tilde \varDelta J ^{- \frac{1}{2}} \tilde \Phi _{i - 1}) (x, \exp _z s \by , z) d s \Big).
\end{align}
Observe that $\tilde \Phi _i$ is at most of non-commutative degree $i$.
Therefore one can rescale and obtain the asymptotic expansion for fixed $t > 0$ and $r \to 0 $:
\begin{equation}
\label{expansion3}
e ^{- (D _t (r))^2} (x, y, z) 
\sim (4 \pi r t) ^{- \frac{\dim \rZ}{2}} e ^{- \frac{ \bd (y, z) ^2}{4 r t}}
\sum _{i = 0} r ^i \Phi _i (x, y, z, t), 
\end{equation}
in the sense that the coefficients of each $d g _{(k)}$ is an asymptotic expansion.
Differentiating Equation \eqref{expansion3}, one gets for fixed $t > 0$, $r \to 0$
\begin{equation}
\label{expansion4}
(D _t - D _t ' ) e ^{- D _t (r) ^2} (x, z, z)
\sim (4 \pi r t) ^{- \frac{\dim \rZ}{2}} 
\sum _{i = 0} r ^i (D_t - D _t ' ) \Phi _i (x, z, z, t).
\end{equation}

\subsection{The Chern character and Chern-Simon form}
Consider the point-wise super trace of \eqref{expansion4}.
From Equation \eqref{expansion2}, we observe that
each $\Phi _i$ is a sum of product of terms in \eqref{Lich} and their derivatives.
Moreover, in order for $(D _t - D _t ' ) \Phi _i$ to have non-zero point-wise super-trace 
it must have degree $\dim \rZ$ in both $\{ \Lambda _j \} $ and $\{ \iota _j \}$.

We write $c _j, \hat c _j$ in terms of $\iota _j, \Lambda _j$.
Note in particular that by \cite[(3.16)]{BGV;Book}, 
the twisting curvature term $\sum _{i, j} c _i c _j \mathcal R (e _i , e _j)$
is of the form $\sum _{i,j,i',j'} \iota _i \Lambda _j \iota _{i'} \Lambda _{j'} \mathcal R _{i j i' j'}$. 
It follows that each factor $\Lambda _i $ is multiplied by factor of $r ^{\frac{1}{2}}$ (or higher power),
therefore $(D _t - D _t ' ) e ^{ - D _t (r) ^2 } = O (r ^{-\frac{1}{2}}) $ as $r \to 0$.
The case for for $r \to 1$ is similar.
Hence it makes sense to define:

\begin{dfn}
The Chern character of $D _t (r)$, $0 \leq r \leq 1$, is
$$ \Ch (D _t (r)) := \str _\Psi \big(e ^{ - D _t (r) ^2 } \big) 
\in \Omega _{\ell ^2} ^\bullet (\rB \rtimes \rG ) _{\Ab}.$$
The Chern-Simon form is
$$ \CS (D_t , D ' _t ) 
:= - \int _0 ^1 \str _\Psi \big( (D _t - D_t ' ) e ^{ - D _t (r) ^2 } \big) d r
\in 
\left\{
\begin{array}{ll}
\Omega _{\ell ^2} ^\bullet (\rB \rtimes \rG ) _{\Ab} & \text{if $\dim \rZ$ is odd,} \\
\widetilde \Omega _{\ell ^2} ^\bullet (\rB \rtimes \rG ) _{\Ab} & \text {if $\dim \rZ$ is even.}
\end{array}
\right.
$$
\end{dfn}

Consider $\Ch (\eth _t (r) )$ as $r \to 0$.
Again one considers the asymptotic expansion \eqref{expansion3}.
By similar arguments as above, 
one concludes $\lim _{r \to 0} \Ch (\eth _t (r) )$ exists,
moreover if $\dim \rZ $ is odd
$$ \lim _{r \to 0} \Ch (\eth _t (r) ) = 0 $$
since the only non-commutative term involving $\Lambda _j$ is of $O (r)$;
If $\dim \rZ$ is even then modulo $\oplus _{k > l} \Omega ^{k, l} _{\ell ^2} (\rM \rtimes \rG) _{\Ab} $,
$\lim _{r \to 0} \Ch (\eth _t (r) )$ is a combination of
$$4 r (1-r) \Big( \frac{t}{8} c _i c _j \mathcal R (E _i , E _j) 
+ \frac{t}{4} ( \frac{1}{8} \hat c _j \hat c _k [\psi _j , \psi _k] 
- \frac{1}{2} c _j \hat c _k \nabla ^{T \rZ \otimes \rE, u} _{E _j} \psi _k ) \Big) \text{ and }
\Lambda _\alpha (\frac{\partial }{\partial x ^\alpha} (g ^* \chi)).$$
It follows that in both cases
$$\lim _{r \to 0} \Ch (\eth _t (r)) = \lim _{r' \to 0} \Ch (\eth _t (1 - r')).$$
%in $\tilde \Omega _c ^\bullet (\rM \rtimes \rG ) _{\Ab}$. 
Hence, our construction implies
\begin{equation}
\label{dCS}
(d _\rB + d ) \CS (\eth _t, \eth ' _t ) 
= \lim _{r \to 0} \Ch (\eth _t (r)) - \lim _{r \to 1} \Ch (\eth _t (r)) = 0 .
\end{equation}

\subsection{The analytic torsion form and transgression formula}
Consider the fiber bundle $\rM \times \bbR ^+ \to \rB \times \bbR ^+$, 
with $\rG$ acting trivially on the $\bbR ^+$ factor.
Define the super-connection
$$ \widetilde {D} := D _t + d t \partial _t $$
on $\rB \times \bbR ^+$.
The adjoint connection of $\widetilde{D}$ with respect to the metric
$$ \langle s , s' \rangle _t := t ^{N _\rV} \langle s , s' \rangle $$
is
$ \widetilde {D}' := D' _t + d t (\partial _t + t ^{-1} N ) .$
Denote
$$ \widetilde {D} (r) := r \widetilde {D} + (1 - r) \widetilde {D}' .$$
One has 
\begin{align*}
\widetilde D (r) ^2 
=& d t (- \partial _t (1 - r) \widetilde {D}' 
+ [(1 - r) t ^{-1} N , r \widetilde {D} + (1 - r) \widetilde {D}'] ) + D _t (r) ^2 \\
=& r (1 - r) d t [ t ^{-1} N , \widetilde {D} - \widetilde {D}'] + D _t (r) ^2.
\end{align*}
By Duhamel's formula
$$ e ^{- \widetilde {D} (r) ^2}
= e^{-D _t (r) ^2 }
+ d t \int _0 ^1 e ^{ - r' D _t (r) ^2 } r (1 - r) [ t ^{-1} N , \widetilde {D} - \widetilde {D}'] 
e ^{ -(1 - r') D _t (r) ^2 } d r' .$$
Consider the Chern-Simon form
\begin{align*}
\CS (\widetilde{D} , \widetilde {D}') 
=& - \int _0 ^1 \str _\Psi \big( (\partial _r \widetilde{D} (r)) e ^{- \widetilde{D} (r) ^2} \big) d r 
 \\
=& - \int _0 ^1 \str _\Psi \big( (D_t - D ' _t - t ^{-1} N d t ) e ^{- \widetilde{D} (r) ^2} \big) d r 
\in \Omega _{\ell ^2} ^\bullet (\rB \times \bbR ^+ \rtimes \rG ) _{\Ab}.
\end{align*}
We compute its $d t $ term:
\begin{align*}
\CS (\widetilde {D} , \widetilde {D}') &- \CS (D _t , D _t') \\
=& d t \int _0 ^1 \str _\Psi ( t ^{-1} N e ^{- D _t (r) ^2 } ) d r \\
&+ d t \int _0 ^1 \str _\Psi \Big((D_t  - D_t') 
\int _0 ^1 e ^{ - r' D _t (r) ^2 } r (1 - r) [ t ^{-1} N , \widetilde{D} - \widetilde {D}'] 
e ^{ -(1 - r') \eth _t (r) ^2 } d r' \Big) d r \\
=& d t \int _0 ^1 \str _\Psi ( t ^{-1} N e ^{- D _t (r) ^2 } ) d r \\
&+ d t \int _0 ^1 r (1 - r) \int _0 ^1 \str _\Psi \Big(t ^{-1} N 
\big[D _t - D_t', e ^{ - r' D _t (r) ^2 } (D _t - D_t' )
e ^{ -(1 - r') D _t (r) ^2 }  \big]dr' \Big) d r .
\end{align*}
Define 
\begin{align*}
T (t) 
\in &
\left\{
\begin{array}{ll}
\Omega _{\ell ^2} ^\bullet (\rB \rtimes \rG ) _{\Ab} & \text{if $\dim \rZ$ is odd,} \\
\widetilde \Omega _{\ell ^2} ^\bullet (\rB \rtimes \rG ) _{\Ab} & \text {if $\dim \rZ$ is even,}
\end{array}
\right. \\
T (t) :=-& \int _0 ^1 \str _\Psi ( N e ^{- D _t (r) ^2 } ) d r \\
&- \int _0 ^1 r (1 - r) \int _0 ^1 \str _\Psi \big(N 
\big[D _t - D _t ', e ^{ - r' D _t (r) ^2 } (D _t - D _t ') 
e ^{ -(1 - r') D _t (r) ^2 } \big] \big) d r' d r .
\end{align*}
Since $ (d _\rB + \partial _t d t + d ) \CS (\widetilde{D} , \widetilde {D}') = 0 $, by Equation \eqref{dCS}, it follows that
\begin{equation}
\label{Transgression1}
\partial _t \CS (D _t , D _t ') = t ^{- 1} (d _\rB + d ) T (t)
\in 
\left\{
\begin{array}{ll}
\Omega _{\ell ^2} ^\bullet (\rB \rtimes \rG ) _{\Ab} & \text{if $\dim \rZ$ odd,} \\
\widetilde \Omega _{\ell ^2} ^\bullet (\rB \rtimes \rG ) _{\Ab} & \text {if $\dim \rZ$ even.}
\end{array}
\right. 
\end{equation}

\subsection{$t \to 0$ asymptotic of the characteristic classes}
The $t \to 0$ behavior of the Chern characteristic is well known.
Define the Euler class
$$e (R ^{\rM / \rB}) :=  
\left\{
\begin{array}{ll} 
\Pf \big( \frac{R ^{\rM / \rB} }{2 \pi} \big) & \text{if $\dim \rZ$ is odd,} \\
0 & \text {if $\dim \rZ$ is even,}
\end{array}
\right. $$
where $R^{\rM / \rB}$ is the curvature of $\nabla ^{\rM / \rB}$ and $\Pf $ is the Pfaffian.
Then one has
\begin{lem}\label{l5.3}
\cite[Theorem 2]{Lott;EtaleGpoid}
As $t \to 0$,
$$ \str _\Psi \big(e ^{ - D _t (r) ^2 } \big) \to  \int _{\rZ _x} \chi e (R ^{\rM / \rB}) 
\tr \big(e ^{- (r \nabla ^{\rE} + (1 - r) (\nabla ^\rE)' +\nabla^G) ^2} \big) .$$
\end{lem}

\begin{proof}
The proof of the lemma is similar to \cite[Proposition 22]{Lott;NonCommTorsion}. 
Consider a rescaling in which $\partial_j\to \varepsilon^{-1/2}\partial_j$, $c^j\to \varepsilon^{-1/2}E^j-\varepsilon^{1/2}I^j$, 
$E^\alpha\to \varepsilon^{-1/2}E^\alpha$, $\widehat{c}^j\to \widehat{c}^j$ and $\nabla^G\to \varepsilon^{-1/2}\nabla^G$. 
One finds from (\ref{Lich}) that as $\varepsilon\to 0$, 
in adapted coordinates the rescaling of $\varepsilon (D_{4}(r))^2$ approaches 
\begin{align}
-4r(1-r)\big(\partial_j-{1\over 4}R^{\rM / \rB}_{j k} x^k \big)^2+4r(1-r)\mathcal{R} + d^M (\nabla ^\rG )+ (\nabla ^\rG )^2.
\end{align}
Using local index method as in \cite[Theorem 3.15]{Bismut;AnaTorsion}, one finds
\begin{align*}
\lim_{t\to 0}\str _\Psi \big(e ^{ - D _t (r) ^2 } \big) 
=  \int_{\rZ _x} \chi (4r(1-r))^{-n/2} & {\Pf} \Big(\frac{4r(1-r)R^{\rB / \rM}}{2 \pi} \Big) \\
&\wedge \tr \Big( e ^{-(d _\rM ( \nabla^G )+ (\nabla ^\rG )^2-r(1-r)\psi^2 )} \Big).
\end{align*}
The claim follows since
$$\big( r \nabla^E + (1-r) (\nabla^E)' + \nabla^G \big)^2
= d _M ( \nabla^G ) + (\nabla^G )^2 - r(1-r) \psi^2 . \qedhere $$
\end{proof}

Next, we turn to the $t \to 0$ limit of the Chern-Simon class. 
The computation is similar to \cite[Proposition 24]{Lott;NonCommTorsion}.
\begin{lem}
\label{ShortTime3}
One has as $t \to 0$, 
$$ \CS (D _t, D _t ')
\to \int _{\rZ _x} \chi e (R ^{\rM / \rB}) \wedge 
\int _0 ^1 \tr \big(\psi e ^{- (r \nabla ^{\rE} + (1 - r) (\nabla ^\rE)' +\nabla^G) ^2} \big) d r.$$
\end{lem}
\begin{proof}
The argument is similar to \cite[Theorem 3.16]{Bismut;AnaTorsion}.
Let $z$ be a Grassmann variable with $z ^2 = 0$ and anti-commutes with all Grassmann variables.
Then 
$$ \str _\Psi \left( (D _t - D' _t) e ^{- (D _t (r) )^2} \right)
= \str _\Psi \Big( \left.\frac{\partial }{\partial z}\right|_{z = 0} 
\frac{1}{2 r (1 - r)}e ^{- (D _t (r) )^2 + 2 z r (1 - r) (D _t - D' _t)} \Big) .$$
Rescale as in Lemma \ref{l5.3}, with $z \to \varepsilon ^{\frac{1}{2}}z $ in addition. One finds from (\ref{Lich}) that as $\varepsilon\to 0$, 
in adapted coordinates the rescaling of 
$\varepsilon((D_4(r))^2+2r(1-r)z(D_t-D_t'))$ approaches 
$$ -4r(1-r)\Big(\partial_j-{1\over 4} R^{\rM / \rB}_{jk }x^k\Big)^2
+ 4r(1-r)\mathcal{R}-2r(1-r)z\psi+ d ^\rM (\nabla^G )+ (\nabla^G )^2.
$$
Proceeding as in the proof of \cite[Theorem 3.16]{Bismut;AnaTorsion}, one obtains
\begin{align*}
\lim _{t \to 0} \CS (D _t &, D_t ') \\
=& \left.\frac{\partial }{\partial z} \right|_{z = 0} 
\int _0 ^1 \frac{1}{2 r (1 - r)} \int _{\rZ _x} \chi e (R ^{\rM / \rB}) 
\tr \Big(e ^{- ( (\nabla^G )^2 - r(1-r)\psi^2 - 2r(1-r)z \psi )} \Big) d r\\
=& \int _{\rZ _x} \chi e (R ^{\rM / \rB}) 
\int _0 ^1 \tr\big(\psi e ^{- (r \nabla ^{\rE} + (1 - r) (\nabla ^\rE)'+\nabla^G) ^2}\big) d r,
\end{align*}
which is the desired result. 
\end{proof}

As for $T (t)$, one has
\begin{lem}
\label{Main8}
(See \cite[Proposition 25]{Lott;NonCommTorsion}) As $t \to 0$,
\begin{align*}
T (t) =& O (t^{\frac{1}{2}}) & \text{if $\dim \rZ$ is odd,} \\
T (t) =& -{n\over 2}\int_{Z _x}\chi e (R ^{\rM / \rB}) 
\int_{0}^{1}\tr \big(e ^{- (r \nabla ^{\rE} + (1 - r) (\nabla ^\rE)' +\nabla^G) ^2}\big)dr + O(t) & \text {if $\dim \rZ$ is even.}
\end{align*}
\end{lem}
\begin{proof}
Let $\widehat{\rM}= \rM \times \bbR^+$ and $\widehat{B}=B\times \mathbb{R}^+$. 
Define $\widehat{\pi}:\widehat{M}\to \widehat{B}$ by $\widehat{\pi}(p,s) := (\pi(p),s)$. 
Let $\widehat{Z}$ be the fiber of $\widehat{\pi}$. Let $g^{\widehat{\rV}}$ be the metric on $\ker (d \widehat{\pi}) $, 
which restricts to $s^{-1}g^{\rV}$ on $M \times \{s\}$. 
Using the method of proof of \cite[Theorem 3.21]{Bismut;AnaTorsion}, one has
\begin{align*}
\widehat{D}_t =& \sqrt{t}d^ \rV +L^{E^{\bullet}_{\flat}}+{1\over{\sqrt{t}}} \iota_{\Theta}+ d s \partial _s+\nabla^G
= s^{-N/2}D_{st}s^{N/2}+ds\partial_s, \\
\widehat{D}'_{t}
=& s\sqrt{t} (d^\rV)^* + (L^{E^{\bullet}_{\flat}})'
-{1\over{s\sqrt{t}}} \varTheta \wedge +ds \Big(\partial_s + {1\over s}\Big( N-{n\over 2} \Big) \Big)+\nabla^G\\
=& s^{-N/2}D'_{st}s^{N/2}+ds \Big(\partial_s+{1\over s} \Big(N-{n\over 2} \Big) \Big).
\end{align*}
Then we compute
\begin{align*}
\widehat{D}_t(r) =& r\widehat{D}_t+(1-r)\widehat{D}'_t\\
=& rs^{-N/2}D_{st}s^{N/2}+(1-r)s^{-N/2}D'_{st}s^{N/2}+ds \partial_s +(1-r)ds{1\over s}\left(N-{n\over 2}\right)\\
=& s^{-N/2}D_{st}(r)s^{N/2}+ds \partial_s +(1-r)ds{1\over s}\left(N-{n\over 2}\right).
\end{align*}
Using Duhamel's formula, one gets a formula similar to \cite[(6.45)]{Lott;NonCommTorsion} (cf. \cite[Proposition 9]{Lott;NonCommTorsion})
and finds that
$$ T (t) = \left\{
\begin{array}{ll}
-{n\over 2}\int_{0}^{1} \str _\Psi \big(e^{-(D_{t}(r))^2} \big)d r + O(t) & \text{ if $\dim \rZ$ is even,} \\
O (t^{\frac{1}{2}}) & \text{ if $\dim \rZ$ is odd.}
\end{array}
\right.
$$
By Lemma \ref{l5.3}, we have
\begin{equation}
\lim_{t\to 0}\int_{0}^{1} \str \big(e^{-(D_{t}(r))^2} \big) d r
=\int_{Z _x}\chi e(R ^{\rM / \rB})\int _0 ^1 \tr\left(e ^{- (r \nabla ^{\rE} + (1 - r) (\nabla ^\rE)' +\nabla^G) ^2}\right) dr.
\end{equation}
Hence the lemma.
\end{proof}

\subsection{A non-commutative Riemann-Roch-Grothendieck index theorem}
One obtains a Riemann-Roch-Grothendieck index theorem by integrating Equation \eqref{Transgression1} 
from $t = 0 $ to $t = \infty$.
We begin with computing the limit of $T (t)$ as $t \to \infty$.
\begin{lem}
\label{Main7}
As $t \to \infty $,
$$ T (t) = -\int_{0}^{1}\str _{\Psi} \big(N e ^{ -(\nabla ^{\Ker (\varDelta )} (r))^2}\big)dr 
+ \dot O (t ^{- \{\varepsilon _n ' \} }).$$
%for some $\varepsilon' > 0$ in all $C ^m $ norms.
\end{lem}
\begin{proof}
First consider the first term of $T (t)$, i.e.
$\int _0 ^1 \str _\Psi ( N e ^{- D _t (r) ^2 } ) d r$.
We split the domain on integration in to 
$0 \leq r \leq t ^{- \frac{1}{2}}, t ^{- \frac{1}{2}} \leq r \leq 1 - t ^{- \frac{1}{2}},
1 - t ^{- \frac{1}{2}} \leq r \leq 1$ (for sufficiently large $t$).
It clearly follows from the asymptotic expansion \eqref{Expansion1} that 
$\str _\Psi ( N e ^{- D _t (r) ^2 } )$ is uniformly bounded as $r \to 0$ and $r \to 1$,
therefore
$$ \int _0 ^{t ^{- \frac{1}{2}}} \str _\Psi ( N e ^{- D _t (r) ^2 } ) d r = O (t ^{-\frac{1}{2}}),$$
and similar for the third integral.

By the first estimate of Theorem \ref{TrDecay0} and since $N $ is bounded, one directly gets
$$ \big\| \str _{\Psi} (N e^{- \eth _{t} (r) ^2} - N e ^{ -(\nabla ^{\Ker (\varDelta )} (r))^2}) \big\| _{C ^m}
= \dot O ((r (1 - r) t ) ^{- \{ \varepsilon _n \} }).$$
Since by construction $r (1 - r) t \geq {t ^{\frac{1}{2}}}$, 
it follows that
$$\int _0 ^1 \str _\Psi ( N e ^{- D _t (r) ^2 } ) d r
= \int_{0}^{1}\str _{\Psi} (N e ^{ -(\nabla ^{\Ker (\varDelta )} (r))^2})dr  + \dot O (t ^{- \{ \varepsilon _n / 2 \}}).$$

We turn to the second term of $T (t)$.
Again, we split the domain of integration into $S := t ^{- \frac{1}{2}} \leq r , r' \leq 1 - t ^{- \frac{1}{2}}$
and $[0, 1] \times [0, 1] \setminus S$.
The volume of $[0, 1] \times [0, 1] \setminus S$ is $O (t ^{- \frac{1}{2}})$, 
hence also the integral over $[0, 1] \times [0, 1] \setminus S$.

On $S$, by Lemmas \ref{Main3} and \ref{Decay3}, 
\begin{align*}
\big[D_t &- D _t ', e ^{ - r' D _t (r) ^2 } (D _t - D _t ') 
e ^{ -(1 - r') D _t (r) ^2 } \big] \\
=& \Big[\big(2 \varPi _0 \Omega + \big(\frac {\varPi _d}{r} - \frac{\varPi _{d ^*}}{1 - r} \big) 
\star \nabla ^\rG \big) 
\star e ^{ -(\nabla ^{\Ker (\varDelta )} )^2} 
\star \big(2 \Omega \varPi _0 + \nabla ^\rG \star \big(\frac{\varPi _{d ^*}}{r} - \frac{\varPi _d}{1 - r} \big) \big), \\
& e ^{ -(\nabla ^{\Ker (\varDelta )} )^2} \Big]
+ O (t ^{- \varepsilon'}),
\end{align*}
in all $\| \cdot \| _{\HS m}$ norms.
Observe that all terms in the bracket preserve the grading in $\wedge ^\bullet \rV'$,
therefore they commute with the grading operator $N$. 
It follows that 
\begin{align*}
N \Big[ & \big(2 \varPi _0 \Omega + \big(\frac {\varPi _d}{r} - \frac{\varPi _{d ^*}}{1 - r} \big) 
\star \nabla ^\rG \big) 
\star e ^{ -(\nabla ^{\Ker (\varDelta )} )^2} 
\star \big(2 \Omega \varPi _0 + \nabla ^\rG \star \big(\frac{\varPi _{d ^*}}{r} - \frac{\varPi _d}{1 - r} \big) \big), \\
& e ^{ -(\nabla ^{\Ker (\varDelta )} )^2} \Big] \\
= & \Big[\big(2 \varPi _0 \Omega + \big(\frac {\varPi _d}{r} - \frac{\varPi _{d ^*}}{1 - r} \big) 
\star \nabla ^\rG \big) 
\star e ^{ -(\nabla ^{\Ker (\varDelta )} )^2} 
\star \big(2 \Omega \varPi _0 + \nabla ^\rG \star \big(\frac{\varPi _{d ^*}}{r} - \frac{\varPi _d}{1 - r} \big) \big), \\
& N e ^{ -(\nabla ^{\Ker (\varDelta )} )^2} \Big].
\end{align*}
By the same arguments as Theorem \ref{TrDecay0}, the $\str _\Psi $ of the above bracket vanishes.

As for the remainder, by the same arguments as Theorem \ref{TrDecay0} one sees that
its trace is also $\dot O (t ^{- \{\varepsilon _n \}})$ in the $C ^m$ norm.
\end{proof}

\begin{dfn}
\label{TorsionDfn}
The analytic torsion form is defined to be
$$
T := \int _0 ^\infty \big( T (t) + T _\infty - (T _0 + T _\infty) (1 - \frac {t}{2}) e ^{- \frac{t}{4}} \big)
\frac {d t}{t} 
\in \left\{
\begin{array}{ll}
\Omega _{\ell ^2} ^\bullet (\rB \rtimes \rG ) _{\Ab} & \text{if $\dim \rZ$ is odd,} \\
\widetilde \Omega _{\ell ^2} ^\bullet (\rB \rtimes \rG ) _{\Ab} & \text {if $\dim \rZ$ is even,}
\end{array}
\right.
$$
where 
\begin{align*}
T _0 :=& -{n\over 2} \int_{\rZ _x} \chi e (R ^{\rM / \rB})
\int_{0}^{1} \tr \big(e ^{- (r \nabla ^{\rE} + (1 - r) (\nabla ^\rE)' +\nabla^G) ^2}\big) d r \\
T _\infty :=& \int_0^1 \str _{\Psi} \big( N e ^{ -(\nabla ^{\Ker (\varDelta )} (r))^2} \big) d r .
\end{align*}
The integral converges and is smooth by Lemmas \ref{Main6} and \ref{Main7}.
\end{dfn}

Integrating Equation \eqref{Transgression1} from $t = 0$ to $\infty$, 
and using Lemma \ref{ShortTime3} and the second equation of Theorem \ref{TrDecay0} to evaluate the limits for 
$\CS (D _t , D' _t)$, one gets:
\begin{thm}
\label{Transgression2}
One has the transgression formula
\begin{align*}
\int _{\rZ _x} \chi e( \nabla^{\rM / \rB}) 
\int _0 ^1 \tr\big(\psi e ^{- (r \nabla ^{\rE} + (1 - r) (\nabla ^\rE)' +\nabla^G) ^2} \big)d t
-& \CS ^{\Ker (\varDelta)} (L ^{\rE _\flat ^\bullet }, (L ^{\rE _\flat ^\bullet })') \\
=& (d + d _\rB) T .
\end{align*}
\end{thm}
\begin{proof}
It remains to prove 
\begin{align*}
(d _\rB + d ) T _\infty =& 0 \\
(d _\rB  +d) T _0 =& 0 
\in \widetilde \Omega _{\ell ^2} ^\bullet (\rB \rtimes \rG ) _{\Ab} \text{ if $\dim \rZ$ is even.}
\end{align*}
For the first equality, we use Lemma \ref{dTr} and consider 
$$ (d _\rB + d ) T _\infty 
= \int_0^1 \str _{\Psi} \big( \big[\nabla ^{\Ker (\varDelta )} (r), N e ^{ -(\nabla ^{\Ker (\varDelta )} (r))^2} \big] \big) d r ,$$
where 
$ \nabla ^{\Ker (\varDelta )} (r) 
= \varPi _0 \big(r L ^{\rE ^\bullet _\flat } + (1 - r) \big(L ^{\rE ^\bullet _\flat }\big)' + \nabla ^\rG \big) \varPi _0 $,
as in \eqref{KerConn}.
Because $L ^{\rE ^\bullet _\flat }$ is the degree $(1, 0)$ component of $D _\rB$, 
it follows that $\nabla ^{\Ker (\varDelta )} (r) $ preserves the grading of 
$\Ker (\varDelta) = \oplus (\wedge ^\bullet \rV' \otimes \rE) \cap \Ker (\varDelta)$,
and hence commutes with $N$.
Therefore
$$\big[\nabla ^{\Ker (\varDelta )} (r), N e ^{ -(\nabla ^{\Ker (\varDelta )} (r))^2} \big] = 0 .$$
As for the second equality, 
observe that by Lemma \ref{Main8}, 
$T _0$ is the $t \to 0$ limit of the family of closed forms
$ -{n\over 2}\int_{0}^{1} \str _\Psi \big(e^{-(D_{t}(r))^2} \big)d r $.
\end{proof}

\begin{rem}
In \cite{Lott;NonCommTorsion} it was furthermore proven that both $T _\infty $ and $T _0$ are exact in  
$\widetilde \Omega _{\ell ^2} ^\bullet (\rB \rtimes \rG ) _{\Ab}$. 
\end{rem}

A non-commutative Riemann-Roch-Grothendieck index theorem immediately follows from Theorem \ref{Transgression2}, 
which can be stated as:
\begin{cor}
Suppose $\dim \rZ$ is even.
One has the equality
$$ \CS ^{\Ker (\varDelta)} (L ^{\rE _\flat ^\bullet }, (L ^{\rE _\flat ^\bullet })')
= \int _{\rZ _x} \chi e( \nabla^{\rM / \rB}) 
\int _0 ^1 \tr\big(\psi e ^{- (r \nabla ^{\rE} + (1 - r) (\nabla ^\rE)' +\nabla^G) ^2} \big)d t $$
in $\mathbf H^\bullet (\widetilde \Omega _{\ell ^2} ^\bullet (\rB \rtimes \rG) _{\Ab}) $.
\end{cor}
Note that $\CS ^{\Ker (\varDelta)} (L ^{\rE _\flat ^\bullet }, (L ^{\rE _\flat ^\bullet })')$ 
is just the Chern-Simon form on the (flat) bundle $\Ker (\varDelta)$.

\begin{rem}
If on the other hand,
$\dim \rZ $ is odd and $(\rE ^\bullet , d ^{\nabla ^\rE})$ is acyclic (i.e. $\varPi _0 = 0$), 
then $(d + d _\rB) T = 0$ and $T $ defines a class in $\mathbf H^\bullet (\Omega _{\ell ^2} ^\bullet (\rB \rtimes \rG) _{\Ab}) $.
Using the arguments in \cite[Theorem 3.24]{Bismut;AnaTorsion}, 
it can be shown that the class of $T$ does not depend on the choice of $\rG$-invariant Riemannian metric $g ^\rM$.
Also note that $T \in \Omega _{\ell ^2} ^\bullet (\rB \rtimes \rG) _{\Ab} $ is non-trivial even if $\rB$ is a point.
\end{rem}

\section{Concluding remarks}
In this paper, we generalized the Bismut-Lott analytic torsion form (Definition \ref{TorsionDfn})
to the non-commutative transformation groupoid convolution algebra,
following the local index theory formalism established in \cite{Lott;EtaleGpoid};
we showed that this torsion form satisfies a transgression formula (Theorem \ref{Transgression2})
-- as expected for a torsion form.
It should be straightforward, but still interesting, to generalize our torsion form to general Etale groupoids
and holonomy groupiods (i.e. foliations),
and compare with \cite{Heitsch;FoliationTorsion}.
%A bigger challenge would be to construct non-commutative versions of topological torsion for groupoids and differentiable stacks,
%and try to prove a Cheeger-Miller type theorem in the non-commutative settings.

%\bibliography{Torsion}

\begin{thebibliography}{10}

\bibitem{Schick;NonCptTorsionEst}
S.~Azzali, S.~Goette, and T.~Schick.
\newblock Large time limit and {$L^2$} local index for families.
\newblock {\em J. Noncommu. Geom.}, {\bf 9}(2):621--664, 2015.

\bibitem{BGV;Book}
N.~Berline, E.~Getzler, and M.~Vergne.
\newblock {\em Heat kernels and Dirac operators}.
\newblock Springer-Verlag, 1992.

\bibitem{Bismut;AnaTorsion}
J.M. Bismut and J.~Lott.
\newblock Flat bundles, direct images and higher real analytic torsion.
\newblock {\em J. Amer. Math. Soc.}, {\bf 8}(2):291--363, 1995.

\bibitem{BMZ;ToeplitzTOrsion}
J.M. Bismut, X.~Ma, and W.~Zhang.
\newblock Asymptotic torsion and toeplitz operators.
\newblock preprint {http://www.math.u-psud.fr/$\sim$bismut/liste-prepub.html},
  2011.

\bibitem{Nistor;EtaleCycilcHomology}
J.-L. Brylinski and V.~Nistor.
\newblock Cyclic homology of {Etale} groupoids.
\newblock {\em K-Theory}, {\bf 8}:341--365, 1994.

\bibitem{Connes;Book}
A.~Connes.
\newblock {\em Noncommutative geometry}.
\newblock Academic press, 1994.

\bibitem{Lott;EtaleGpoid}
A.~Gorokhosky and J.~Lott.
\newblock Local index theory over {E}tale groupoids.
\newblock {\em J. Reine. Angew. Math.}, {\bf 560}:151--198, 2003.

\bibitem{Lott;FoliationInd}
A.~Gorokhosky and J.~Lott.
\newblock Local index theory over foliation groupoids.
\newblock {\em Adv. Math.}, {\bf 244}(4):351--386, 2007.

\bibitem{Heitsch;FoliationHeat}
J.L. Heitsch.
\newblock Bismut super-connections and the {Chern} character for {Dirac}
  operators on foliated manifolds.
\newblock {\em K-Theory}, {\bf 9}:507--528, 1995.

\bibitem{Heitsch;FoliationTorsion}
J.L. Heitsch and C.~Lazarov.
\newblock {Riemann-Roch-Grothendieck} and torsion for foliations.
\newblock {\em J. Geom. Anal.}, {\bf 12}(3):437--468, 2002.

\bibitem{Piazza;NonCommEta}
E.~Leichtnam and P.~Piazza.
\newblock Etale groupoids, eta invariants and index theory.
\newblock {\em J. Reine Angew. Math.}, {\bf 587}:169--233, 2005.

\bibitem{Lopez;FoliationHeat}
J.A.~Alvarez Lopez and Y.A. Kordyukov.
\newblock Long time behavior of leafwise heat flow for {Riemannian} foliations.
\newblock {\em Compositio Math.}, {\bf 125}(2):129--153, 2001.

\bibitem{Lott;NonCommTorsion}
J.~Lott.
\newblock Diffeomorphisms and noncommutative analytic torsion.
\newblock {\em Mem. Amer. Math. Soc.}, {\bf 141}:1--56, 1999.

\bibitem{NWX;GroupoidPdO}
V.~Nistor, A.~Weinstein, and P.~Xu.
\newblock Pseudodifferential operators on differential groupoids.
\newblock {\em Pac. J. Maths}, {\bf 189}(1):117--152, 1999.

\bibitem{Shubin;BdGeom}
M.~A. Shubin.
\newblock Spectra of elliptic operators on non-compact manifolds.
\newblock In {\em Methodes semi-classiques Vol 1}, volume 207 of {\em
  Asterisque}, pages 35--108, 1992.

\bibitem{So;CommTorsion}
B.K. So and G.~Su.
\newblock Regularity of analytic torsion form on families of normal coverings.
\newblock To appear in Pacific Journal of Mathematics arXiv:1405.4631, 2014.

\end{thebibliography}

%\bibliographystyle{elsarticle-num}

%%%%%%%%%%%%%%
\end{document}